\documentclass[a4paper,11pt]{article}
\usepackage[utf8]{inputenc}
\usepackage{tikz}
\usepackage{amsthm}
\usepackage{amsmath,stmaryrd, mathrsfs}
\usepackage{amsfonts}
\usepackage{amssymb}
\usepackage{dsfont}
\usepackage{tikz-cd}
\usepackage{subcaption}
\usepackage{float}
\usepackage{url}
\usepackage{hyperref}
\hypersetup{bookmarksdepth=subsection}
\usepackage{booktabs}
\usepackage{import}
\usepackage{enumitem}
\usepackage{etoolbox}
\usepackage[normalem]{ulem}
\usepackage{soul}
\usepackage{upgreek}

\numberwithin{equation}{section}

\usepackage{geometry}
\makeatletter
\patchcmd{\l@section}
	{\addvspace{1.0em \@plus\p@}}
	{\addvspace{0.4em \@plus\p@}}
	{}{}
\makeatother

\newtheorem{theorem}{Theorem}[section]
\newtheorem{lemma}[theorem]{Lemma}
\newtheorem{proposition}[theorem]{Proposition}
\newtheorem{corollary}[theorem]{Corollary}
\newtheorem{remark}[theorem]{Remark}

\newtheorem{definition}[theorem]{Definition}

\newcommand{\RR}{\mathbf{R}}
\newcommand{\NN}{\mathbf{N}}
\newcommand{\ZZ}{\mathbf{Z}}

\newcommand{\TT}{\mathbf{T}}

\newcommand{\EE}{\mathbb{E}}

\newcommand{\PP}{\mathbb{P}}

\newcommand{\mN}{\mathcal{N}}
\newcommand{\mO}{\mathcal{O}}
\newcommand{\mP}{\mathcal{P}}
\newcommand{\mQ}{\mathcal{Q}}

\newcommand{\mF}{\mathcal{F}}
\newcommand{\mD}{\mathcal{D}}

\newcommand{\mS}{\mathcal{S}}

\newcommand{\mA}{\mathcal{A}}

\newcommand{\mZ}{\mathcal{Z}}

\newcommand{\mf}[1]{\mathfrak{#1}}
\newcommand*{\ud}{\mathrm{\,d}}

\def\gen{\mathcal{G}} % full particle generator
\def\gens{\mS} % symmetric part (no scaling)
\def\gena{\mA} %asymmetric part ( no scaling)
\def\var{\mathbb{V}\mathrm{ar}}
\def\ff{\mf{f}}

\title{\resizebox{0.98\textwidth}{!}{
    Equilibrium fluctuations of the weakly asymmetric inclusion process
  }}
\author{Simon Gabriel\footnote{Email: \href{mailto:s.gabriel@berkeley.edu}{s.gabriel@berkeley.edu},
Department of Mathematics, University of California, Berkeley, USA.} }
\date{\ }

\begin{document}

\maketitle

\vspace{-1.7cm}
\begin{abstract}
	We prove that the stationary density fluctuation field of the one-dimensional weakly
	asymmetric simple inclusion process converges to the energy solution of the stochastic
	Burgers equation. The key technical ingredient is a second-order Boltzmann--Gibbs
	principle tailored to the microscopic current.
	As a consequence, we also obtain Ornstein--Uhlenbeck fluctuations for the symmetric inclusion process in a weakly condensing regime, a result of independent interest.
\end{abstract}

\hspace{.4cm}{\small\textit{{Keywords.}} Inclusion Process, Weak KPZ
	Universality, Energy Solutions, Condensation.}

\hspace{.4cm}{\small\textit{{MSC classification.}} Primary: 60K35,
	%60J25,
	Secondary: 60H15, 82B21.}

\tableofcontents

\section{Introduction}

The Kardar--Parisi--Zhang universality class is expected to describe the
nonlinear large-scale fluctuations of one-dimensional driven conservative
systems, including strongly or totally asymmetric dynamics, see for example
\cite{Corwin12KPZ,Remenik22ICM} and the references therein.
The present paper studies
the weak-asymmetry, or weak-universality, route to this class.  In this regime
the asymmetry is sent to zero together with a scaling parameter.  The
limiting density fluctuation field is the stochastic Burgers equation,
equivalently the spatial derivative of the KPZ equation.
Rigorous derivations of this weak limit from microscopic dynamics go back to
Bertini and Giacomin's derivation from the weakly asymmetric simple exclusion
process \cite{BertiniGiacomin97}.  Beyond exclusion,
weak KPZ universality has been established for zero-range dynamics
\cite{GoncJaraSethuraman15}, continuous interacting diffusions \cite{DGP17},
and several variations of such systems.
In the stochastic PDE literature, analogous weak-universality results have
been proved for weakly asymmetric stochastic growth equations and related
interface fluctuation models, see for example \cite{HairerQuastel18}.
Motivated by recent predictions of KPZ-type fluctuations in asymmetric bosonic
and fermionic transport \cite{MinoguchiHuberGarbeGambassiRabl25}, we study here
the weak-asymmetry route for the simple inclusion process, an attractive
conservative particle system with occupation-enhanced jump rates.

\begin{figure}[H]
	\centering
	\includegraphics[width=14cm]{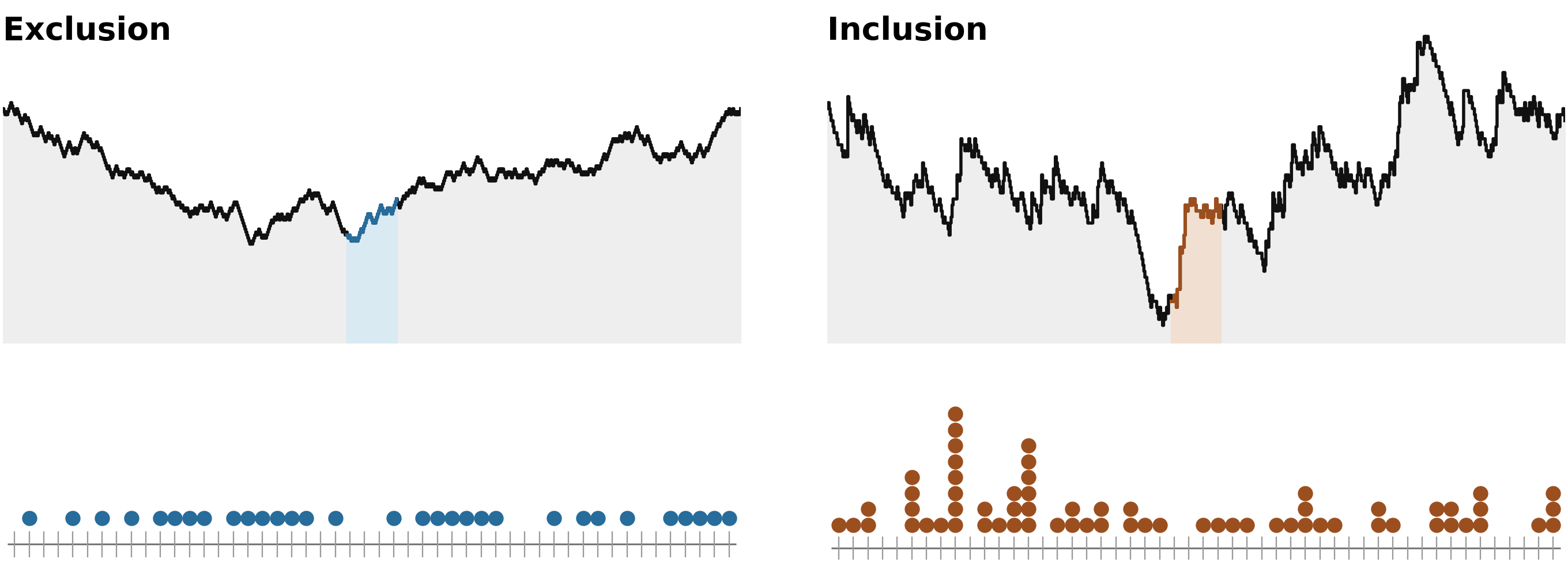}
	\caption{Plot of stationary fluctuation profiles
		$ x \mapsto \sum_{y=1}^{x} (\eta_{y}- \rho) $ with $x \leqslant 750$ and density $
			\rho_{\rm ex}=
			\tfrac{1}{2}$ for the exclusion process and density $ \rho_{\rm ip}=1$ for the
		inclusion process (with $ \alpha=1$).
		The coloured areas correspond to the microscopic particle configurations
		displayed at the bottom.}
		\label{fig:simulations}
\end{figure}

The inclusion process is an interacting particle system on $\ZZ$ with
unbounded occupation variables. 
If the particle configuration is $\eta \in \NN_{0}^{\ZZ}$, then the rate of a particle jumping from $x$ to $y$ is
\begin{equation*}
	c_L(x,y)\eta_x(\alpha+\eta_y).
\end{equation*}
%We study its diffusive scaling as the lattice spacing $L^{-1}$ vanishes when $L$ diverges.
Here $ \eta_{x}$ denotes the number of particles at location $ x \in \ZZ$ and  $c_L$
encodes the geometry of the underlying jumps: below we restrict to nearest-neighbour jumps with a preference to jump to the
right proportional to
$ \beta L^{-1/2}$, $ \beta>0$.
The parameter $L $ will denote the inverse diffusive scaling at which we approximate continuous space as $ L$ grows.
While the factor $\eta_x$ counts particles
available to jump, $\alpha+\eta_y$ enhances jumps into already
occupied sites.
This bosonic stimulation mechanism is responsible for the condensation phenomena
\cite{GRV13} of the
inclusion process and places the model outside the direct scope of the
standard microscopic KPZ derivations.

Our main result (Theorem~\ref{thm_main}) identifies the limit of the equilibrium density fluctuation field $ \sqrt{L} ( \eta_{x}(t)-
	\rho)_{t,x}$, when observed
in a moving frame and started from the invariant product measure with
particle density $\rho>0$.
The scaling limit is the unique stationary energy solution of the singular
stochastic Burgers equation
\begin{equation}\label{eq_burgers}
	\begin{aligned}
		\partial_{t} u
		= \frac{\alpha}{2} \partial_{x}^{2} u
		- \beta \partial_{x} (u^{2})
		+ \sqrt{\alpha\rho+\rho^{2}}\,\partial_{x}\xi\,,
	\end{aligned}
\end{equation}
where $ \xi$ denotes space--time white noise on $ \RR_{+}\times \RR $.\footnote{
We notice that the inclusion process fluctuations are stronger than those associated
with the weakly asymmetric exclusion process (with $ \alpha=1 $), where the noise
scales like $ \sqrt{\rho ( 1- \rho)} \partial_{x} \xi$ with $ \rho \in [0,1]$. See also Figure~\ref{fig:simulations}.}
The derivation of \eqref{eq_burgers} follows the energy-solution strategy initiated in
\cite{GoncJara14}, with new
inputs for the inclusion process.
%The local state space is unbounded, and the microscopic current has genuinely quadratic
%two-site growth.
The novel technical ingredient is therefore a
current-specific second-order Boltzmann--Gibbs principle for the inclusion
process.

Beyond the case of fixed $ \alpha$, we also analyse a weakly condensing regime
$\alpha=\alpha_L\to 0$ such that $\alpha_L L\to\infty$, which leads to diverging
particle clusters.
After suitable normalisation of the density field, we establish Gaussian Ornstein-Uhlenbeck
fluctuations for symmetric particle interactions ($
	\beta=0$),
see Theorem~\ref{thm_main_cond}.
We also discuss possible extensions and limitations to an asymmetric regime in which
similar Burgers-type fluctuations are expected.

\subsection{The inclusion process}\label{sec_ip}

The inclusion process is a conservative interacting particle system with an
attractive jump mechanism,
mimicking bosonic attraction in a discrete space of energy modes.
It was introduced in \cite{GKR07} as the dual of an energy transport model.
The process has a formal infinitesimal description given in terms of
\begin{equation}\label{eq_genIP}
	\begin{aligned}
		\gen_{L} f ( \eta)
		=
		L^{2}
		\sum_{x, y \in \ZZ}
		c_{L} (x,y) \eta_{x}( \alpha + \eta_{y})
		\big[ f ( \eta^{x,y}) - f ( \eta) \big]\,,
	\end{aligned}
\end{equation}
for some $ \alpha>0$ with state space $ \Omega :=
	\NN_{0}^{\ZZ}$ which we equip with the product topology.
Here $ \eta^{x,y}$ denotes the particle configuration $ \eta \in \Omega$ with one particle
moving from position $ x $ to position $ y$, i.e.
$ \eta^{x,y}_{z} = \eta_{z}+ \delta_{y,z} - \delta_{x,z}$.
The function $ c_{L} : \ZZ\times \ZZ \to [0, \infty)$ encodes the
geometry of the underlying lattice $ \ZZ$.
Throughout the paper, we will restrict ourselves to large $L$ and
\begin{enumerate}
	\item \emph{nearest neighbour interactions}, i.e. $ c_{L} (x, y ) = 0 $
	      unless $ |x-y|=1$,

	\item \emph{weakly asymmetric rates}\footnote{
		      To the best of our knowledge,
		      well--posedness of the asymmetric inclusion process in infinite volume on $ \ZZ$ is
		      an open problem and will not be addressed here.
		      Instead, we will consider \emph{the} stationary asymmetric inclusion process on $ \ZZ$ associated to the large
		      volume limit of periodic approximations. This is discussed in detail below.
			  % and in
		      %Appendix~\ref{sec_ip_exists}. Our result applies more generally to any stationary %process
		      %solving the forward and backward martingale problem associated to~$\gen_{L}$.}
	}, i.e. there exists a $ \beta >0$ such
	      that $ c_{L} $ has the form
	      \begin{equation}\label{eq_weakasy}
		      \begin{aligned}
			      c_{L}  (x , x+1) = p_{L} :=
			      \tfrac{1}{2} \big( 1+ L^{-1/2} \beta \big)
			      \quad \text{and}\quad
			      c_{L}(x, x-1 ) = q_{L}:=
			      \tfrac{1}{2} \big( 1-L^{-1/2} \beta\big)\,.
		      \end{aligned}
	      \end{equation}

\end{enumerate}
We will often drop the subscript of $ c = c_{L} $, leaving the dependence on $L$ implicit.
With the above assumption, $ c$ is spatially homogeneous in the sense that $ c (x,y) $
depends only on $ x-y$.

The jump rates in \eqref{eq_genIP} combine two competing dynamics.
On one hand, the dissipative effect by $ \alpha$ leads to diffusive behaviour of
the particles.
On the other hand, the rate $ \eta_{x} \eta_{y} $ encourages particles to accumulate, which can
possibly lead to local particle clusters.
This ``inclusion'' mechanism is the microscopic source of both features that are
central for this paper: the model has unbounded occupation variables with
quadratic two-site transfer rates, and it exhibits condensation phenomena if
the parameter $\alpha= \alpha_{L} $ vanishes in the large $L$ limit.
\par\smallskip\smallskip

Because the jump rates are unbounded, the infinite-volume description requires some care.
First, since all fluctuation fields in this paper are studied in equilibrium,
we will only use the stationary process associated with a one-parameter\footnote{Often
	the stationary measures are instead indexed by the
	fugacity $ \phi = \rho/ (\alpha+ \rho) $.} family of invariant product
measures.
More precisely, for any fixed particle density $ \rho>0 $, we define
the product measure
\begin{equation}\label{eq_nu}
	\begin{aligned}
		\nu_{\rho} (\ud \eta)
		:=
		\bigotimes_{x \in \ZZ} \frac{1}{Z_{\rho}} w (
		\eta_{x}) \Big( \frac{ \rho}{ \alpha + \rho} \Big)^{
		\eta_{x}}\,,\qquad \text{with }\ \quad
		w ( n ) := \frac{ \Gamma ( \alpha + n )}{ n! \Gamma(
			\alpha)}\,,
	\end{aligned}
\end{equation}
and $ Z_{\rho}$ a normalising constant \cite{GRV11}.
Notice that $
	\nu_{\rho}$ is the product of negative binomial distributions $
	\mathrm{NB}(\alpha, \tfrac{\alpha}{ \alpha+ \rho})$.
The measures $ \nu_{\rho}$ are called \emph{grand--canonical
	distributions} and
are supported on the whole configuration space $ \Omega$, without fixing
the total number of particles.
On the other hand, fixing the number of particles on a subset of the
domain $ \ZZ$ will lead to the so called \emph{canonical distributions},
which are introduced in Section~\ref{sec_equivens}.
We will write $ \EE_{\rho}$ for the expectation with respect to
$ \nu_{\rho}$.

Indexing $ \nu_{\rho}$ by the particle density is adequate since one can check that
$\EE_{\rho} [ \eta_{x}] = \rho$, and
\begin{equation}\label{eq_eta_mean_var}
	\begin{aligned}
		 				\var_{\rho}( \eta_{x}) = \rho + \tfrac{ \rho^{2}}{
			\alpha} =: \chi( \rho)\,,
		\quad
		\text{and} \quad
		\EE_{\rho}[ \eta_{x}^{q}] < \infty\,, \quad \forall q
		\geqslant 1
		\,.
	\end{aligned}
\end{equation}
Moreover, we notice the following quantitative moment bounds in
the small $ \alpha$ regime:
\begin{equation}\label{eq_quant_mom_bound}
	\begin{aligned}
		\EE_\rho[\eta_x^q]
		+
		\EE_\rho[|\eta_x-\rho|^q]
		\lesssim_{q,\rho}
		\alpha^{1-q}\,, \quad \forall q \geqslant 1\,,
	\end{aligned}
\end{equation}
and in particular $ \chi ( \rho) \lesssim_{\rho} \alpha^{-1}$.
This is a consequence of $ \nu_{\rho}$ being the product of negative binomial
distributions.

Let us return to the construction of inclusion dynamics on $ \ZZ$.
Throughout the remainder of the paper $ \eta = (\eta^{\rho,L}(t) )_{ t \geqslant 0}$ denotes
\emph{the} stationary process that is approximated by finite dimensional periodic
approximations with initial condition $ \nu_{\rho}$.
Its construction is presented in Proposition~\ref{prop_ip_exists} in the appendix.
We will write $ \mathbf P_{\rho,L}$ for the corresponding law and
$ \mathbf E_{\rho,L}$ for its expectation.
In particular, $ ( \eta (t))_{t \geqslant 0} $ solves the martingale problem
associated to $ \gen_{L}$, that is
\begin{equation}\label{eq_dynkin_mart_timeinde}
	\begin{aligned}
		M_t^{L}(f) :=
		f(\eta(t))-f(\eta(0))
		-
		\int_0^t \gen_{L} f(\eta(s)) \ud s\,,
	\end{aligned}
\end{equation}
is a martingale
for every local\footnote{$f$ only depends on finitely many positions $x$ of the input
	configuration $ (\eta_{x} )_{x \in \ZZ}$} and bounded  $ f: \Omega \to \RR$.
	We stress that only the martingale
formulation is used below, so our results apply without further changes to any sequence
of stationary processes $(\eta^{ \rho, L}(t))_{t \geqslant 0}$ for which \eqref{eq_dynkin_mart_timeinde} and the corresponding backward formulation hold.
Well-posedness beyond such finite-dimensional approximations is a delicate
issue for inclusion processes on the infinite lattice.
In the symmetric case,
%, finite-time flux from infinity is prevented by explicit cancellations.
 the reversible structure and self-duality allows to prove well-posedness for a
wide admissible class of
initial configurations, see for example \cite[Remark~2.1]{AyalaRedig}.
On the other hand,
monotonicity and weighted-mass estimates used for the infinite-volume asymmetric zero-range
process do not apply for the inclusion process.
\par\smallskip\smallskip

Lastly, because an arbitrary number of particles can occupy a single position in the
inclusion process, the attractive interaction term $ \eta_{x} \eta_{y}$ in
\eqref{eq_genIP} can lead to particles clustering into fragments of diverging
size as $ L$ grows and $ \alpha_{L} \to 0$ \cite{GRV13,CG14}.
This phenomenon is known as \emph{condensation}.
As an indication for this, we observe that for
fixed $\alpha>0$, the equilibrium variance $ \chi ( \rho)$ is finite, but diverges
as $ \alpha = \alpha_{L}$ vanishes.
Below we will consider a \emph{weakly condensing regime} where $ \alpha_{L}$ vanishes
slowly enough that $ \alpha_{L} L \to
	\infty$. This leads to typical particle clusters of size $\mO ( \alpha_{L}^{-1})$.
It stands in contrast to the regime $
	\alpha_{L} L = \mO (1)$ where particle clusters are of order $
	\mO(L)$ \cite{JCG19,CGG22}.

\subsection{Density fluctuation field and main results}

We study two equilibrium fluctuation regimes for the inclusion process:
the fixed-$\alpha$ weakly asymmetric regime, in which the
limit is the stochastic Burgers equation, and the weakly condensing regime $
	\alpha_{L} L \to \infty $ with symmetric interactions, in which
the diverging compressibility requires a different normalisation and leads to
Ornstein--Uhlenbeck fluctuations.

Regarding the first regime, our main objective is the analysis of the density fluctuation field $
	\sqrt{L}  (
	\eta_{x} (t) - \rho)_{t,x }$.
However, due to the asymmetry of the generator we cannot simply test the
fluctuation field against a fixed test function $ \varphi \in \mS ( \RR) $. Instead,
we have to consider a moving frame associated with the \emph{microscopic asymmetric current}
\begin{equation}\label{eq_def_current}
	\begin{aligned}
		j_{x , x+1} ( \eta) := \beta\Big( \frac{\alpha}{ 2} (
			\eta_{x}+ \eta_{x+1}) + \eta_{x} \eta_{x+1}\Big) \,.
	\end{aligned}
\end{equation}
The associated \emph{equilibrium flux} is
\begin{equation}\label{eq_def_j}
	\begin{aligned}
		j ( \rho)
		:=
		\EE_{\rho}[ j_{x , x+1}( \eta)] = \beta( \alpha \rho+
		\rho^{2})\,,
	\end{aligned}
\end{equation}
which we compensate by a transport shift.\footnote{For the weakly asymmetric
	simple exclusion process, see \cite{GoncJara14}, one can choose $ \rho$ such that
	$ j ' ( \rho) =0 $ and the transport term vanishes. However, this is not possible
	for the inclusion process, see \eqref{eq_def_j}.}
The correct observables  are therefore of the form
\begin{equation}\label{eq_def_u}
	\begin{aligned}
		u_{L}( t, \varphi; \eta (t))
		:=
		\frac{1}{ \sqrt{L}}  \sum_{x \in \ZZ} \varphi \big( L^{-1}x -
		L^{1/2}
		j' ( \rho) t   \big) \big(
		\eta_{x}(t) - \rho\big)\,, \quad \varphi \in \mS ( \RR) \,.
	\end{aligned}
\end{equation}
We will often drop the dependence on $ \eta$ and simply write
$u_{L}(t,\varphi)$.
After moving with the characteristic speed $j'(\rho)$, 
the leading order term of the asymmetric drift  is a quadratic term with coefficient $\frac{1}{2}j''(\rho)=\beta$, which determines the nonlinear coefficient in the Burgers equation.

Our first result establishes convergence of the density fluctuation field to the energy
solutions of the stochastic Burgers equation \eqref{eq_burgers}.
Energy solutions are introduced in
detail in Section~\ref{sec_energy_solutions}.

\begin{theorem}[Non--condensing weakly asymmetric IP]\label{thm_main}
	Let $ \rho> 0 $ and $ \alpha , \beta>0 $.
	The sequence of stationary density fluctuation fields $ ( u_{L})_{L \in \NN}$
	with weak asymmetry \eqref{eq_weakasy}
	converges weakly in $ D ([0,T] ,
		\mS ' ( \RR ))$ to the unique stationary energy solution of the stochastic Burgers equation
	\begin{equation}\label{eq_burgers_again}
		\begin{aligned}
			\partial_{t} u
			= \frac{\alpha}{2} \partial_{x}^{2} u
			- \beta \partial_{x} (u^{2})
			+ \sqrt{\alpha\rho+\rho^{2}}
			\,\partial_{x}\xi\,,
		\end{aligned}
	\end{equation}
	in the sense of Definition~\ref{def_energy_solution}.
\end{theorem}

The novelty of Theorem~\ref{thm_main} is that it applies the nonlinear
fluctuation theory to the classical weakly asymmetric inclusion process.
For exclusion-type models and zero-range or related mass-conservative systems, the
stochastic
Burgers limit has been obtained
by the use of G\"artner transformations \cite{BertiniGiacomin97,DemboTsai16},
several versions of a second-order Boltzmann--Gibbs principle
\cite{GoncJara14,GoncJaraSethuraman15,Yang23BG},
and
Regularity Structures \cite{HuangMatetskiWeber25}.
The inclusion process falls outside these frameworks because the occupation
variables are unbounded and the rate
$\eta_x(\alpha+\eta_y)$ has quadratic growth in the local particle numbers.
Our proof is therefore tailored to the inclusion process. It is based on a replacement
theorem through a second-order Boltzmann--Gibbs principle \cite{GoncJara14} and
the notion of Energy solutions \cite{GubPerk}.

We restrict ourselves to the stationary setting.  This avoids the separate
problem of constructing non-stationary infinite-volume asymmetric inclusion
dynamics and allows us to prove the second-order Boltzmann--Gibbs principle
by stationary $L^2(\nu_\rho)$-estimates.  Extending the argument beyond
equilibrium would require replacement estimates that do not rely on global
stationarity, see for example \cite{Yang18Burgers,Yang23BG}.
\par\smallskip\smallskip

Our quantitative analysis of the fluctuation behaviour that proves Theorem~\ref{thm_main}
extends to the weakly condensing regime
$ \alpha_{L}L \to \infty $ for symmetric
interactions.
% , since we carefully keep track of the explicit $ \alpha$-dependencies.
In this case, we consider the variance-normalised field
\begin{equation}\label{eq_cond_density_field}
	\begin{aligned}
		v_{L}(t, \varphi )
		:=
		\frac{1}{ \sqrt{  \chi_{L}( \rho) L}}
		\sum_{x \in \ZZ} \varphi ( L^{-1} x) ( \eta_{x}(t/ \alpha_{L}) - \rho)
		\,.
	\end{aligned}
\end{equation}
Here we added the subscript to $ \chi_{L}$ to make the dependence on $ \alpha_{L}$
explicit, in particular, $ \lim_{L \to \infty}  \chi_{L}( \rho) =\infty$ for every $ \rho
	>0 $.
%The time acceleration by $1/\alpha_L$ compensates the vanishing diffusive
%coefficient of the symmetric dynamics.

\begin{theorem}[Weakly condensing symmetric IP]\label{thm_main_cond}
	Let $ \rho>0 $, $ \beta=0 $, and let $ ( \alpha_{L})_{L \in \NN} $ be a convergent sequence such
	that $ \lim_{L \to \infty} \alpha_{L} \in [0, \infty) $ and $ \lim_{L \to \infty}\alpha_{L} L = \infty $.
	Then the sequence of symmetric density fluctuation fields $( v_{L})_{L \in \NN }$
	converges weakly in $ D ([0,T] ,
		\mS ' ( \RR ))$ to the stationary Ornstein–Uhlenbeck process, which solves
	\begin{equation*}
		\partial_{t} v
		= \frac{1}{2} \partial_{x}^{2} v
		+ \partial_{x}\xi\,.
	\end{equation*}
\end{theorem}

The theorem describes the Gaussian
fluctuation regime that remains when $\alpha_{L }$ vanishes slowly
enough that many mesoscopic clusters of size $ o(L)$ are still averaged by each macroscopic
test function.
For fixed $ \alpha$, duality methods yield
equilibrium and non-equilibrium Ornstein--Uhlenbeck limits for symmetric
inclusion dynamics, including long-jump variants
\cite{AyalaCarinciRedig21,AyalaZimmer25}.
Beyond the weakly condensing regime the condensation behaviour differs drastically and
Theorem~\ref{thm_main_cond} is not expected to extend to this regime.
There the cluster sizes are of order $ \mO (L)$ and their occurrence is much sparser across
the space, which does not lead to the CLT-like behaviour at macroscopic scales described above.
Here the natural
observables are cluster statistics or cluster locations rather than centred
spatial density fields \cite{AyalaRedig,CGG24,Gabriel25,LandimKim2026}.

The closest comparison in the weakly condensing regime is the size-biased
mean-field analysis of \cite{CGG24}.  There the inclusion process on the
complete graph is encoded by the measure-valued process which records the empirical distribution of rescaled cluster sizes.
In the regime $\alpha_L L\to\infty$, this process has a deterministic
hydrodynamic limit, whose profile is exponential.
This is consistent with the weakly condensing picture here: typical size-biased
clusters have size of order $\alpha_L^{-1}$, and
$\chi_L(\rho)\sim \rho^2/\alpha_L$.
In contrast to the size--biased analysis, Theorem~\ref{thm_main_cond} keeps the spatial
geometry.
%  and moreover proves  a spatial Gaussian fluctuation result.

The natural next question is whether one can combine the two mechanisms above,
namely weak asymmetry and weak condensation.  The scaling analysis in
Section~\ref{sec_asym_cond} suggests a nontrivial asymmetric regime, which we leave for future work.

\paragraph{Concerning AI usage.}
ChatGPT 5.5 Pro was used to review parts of the manuscript. During this review, it was observed that
Lemma~\ref{lem_quant_qv} yields the sharper quadratic-variation bound
\eqref{eq_qv_quant_cond}. This allowed us to strengthen 
Theorem~\ref{thm_main_cond} from
$\alpha_L^{2}L\to\infty$ to the full weakly condensing regime, see
Remark~\ref{rem_cond_improvement}.

\section{Preliminaries and proof strategy}

\subsection{Generator decomposition and martingales}

We begin by recording the algebraic identities for the generator that will be
used throughout the paper.  All adjoints are taken in $L^2(\nu_\rho)$, whose
inner product is
\begin{equation*}
	\begin{aligned}
		\langle f, g \rangle_{\rho}
		:= \EE_{\rho}[ f( \eta)  g ( \eta) ] \,.
	\end{aligned}
\end{equation*}
The scaling of the weak asymmetry separates the accelerated generator into a
diffusive symmetric part and a lower-order antisymmetric part.  Namely, with
$\gen_L^\ast$ denoting the adjoint of $\gen_L$, we set
\begin{equation}\label{eq_sym}
	\begin{aligned}
		\gens f ( \eta)
		:=
		\frac{ 1}{2 L^{2} } \big( \gen_{L} + \gen_{L}^{\ast}\big)
		=
		\frac{1}{2}
		\sum_{\substack{x,y  \in \ZZ \\ |x-y| = 1}}
		\eta_{x} ( \alpha+ \eta_{y}) \big[ f ( \eta^{x,y}) - f (
			\eta) \big]\,,
	\end{aligned}
\end{equation}
and
\begin{equation*}
	\begin{aligned}
		\gena f ( \eta)
		:= &
		L^{-3/2}
		\big( \gen_{L} - L^{2}\gens\big) f ( \eta) \\
		=  &
		\frac{\beta}{2}
		\sum_{x \in \ZZ}
		\Big(
		\eta_{x} ( \alpha+ \eta_{x+1})
		\big[ f ( \eta^{x,x+1}) - f (\eta) \big]
		-
		\eta_{x} ( \alpha+ \eta_{x-1})
		\big[ f ( \eta^{x,x-1}) - f (\eta) \big]
		\Big)\,.
	\end{aligned}
\end{equation*}
After this normalisation neither $\gens$ nor $\gena$ depends on $L$, and $ \gen_{L}=L^{2}\gens+L^{3/2}\gena$.
Moreover, 
\begin{equation}\label{eq_adjointness}
	\begin{aligned}
		\gens^{\ast} = \gens\,, \quad \gena^{\ast} = - \gena\,,
		\quad \text{and} \quad
		\gen_L^{\ast} = L^{2}\gens - L^{3/2}\gena \,, \quad
		\text{on }  L^{2} ( \nu_{\rho}) \,.
	\end{aligned}
\end{equation}
%which can be deduced by a direct change of variables with respect to $\nu_\rho$.

Thus $\gens$ is the reversible part of the dynamics, while $\gena$ is the
part responsible for the asymmetric current.  Indeed, with the current
definition \eqref{eq_def_current}, the action of $\gena$ on the coordinate
field is given in terms of
\begin{equation*}
	\begin{aligned}
		\gena \eta_{x} = j_{x-1 , x}( \eta) - j_{x , x+ 1}(
		\eta)\,.
	\end{aligned}
\end{equation*}
This identity is the microscopic origin of the transport correction in \eqref{eq_def_u} and of the
nonlinear drift appearing in the fluctuation field.

Throughout the paper we will apply Dynkin's formula for time-dependent observables that are not
bounded and local, such as the density fluctuation field.
Indeed, our construction of the inclusion process provides a solution to the martingale problem for bounded local
test functions, see Proposition~\ref{prop_ip_exists}. In Section~\ref{sec_dynkin}, we extend this to non-local polynomial fields of the
form
\begin{equation*}
	F(t,\eta)
	:=
	\sum_{x\in\ZZ} v_{t,x}P(\eta_x),
\end{equation*}
where $P$ is a polynomial and $v\in C^1([0,T];\ell^1(\ZZ))$. More precisely, for every
observable obtained through this extension, the process
\begin{equation}\label{eq_dynkin_mart}
	\begin{aligned}
		M_t^{L}(F) :=
		F (t, \eta(t)) - F(0, \eta(0))
		- \int_{0}^{t} ( \partial_{s} + \gen_{L}) F(s,
		\eta(s))\ud s
	\end{aligned}
\end{equation}
is a martingale.  Its predictable quadratic variation is given by the squared jump
increments:
\begin{equation}\label{eq_qv}
	\begin{aligned}
		\langle M^{L}(F) \rangle_{t}
		 & =
		L^{2}
		\int_{0}^{t}
		\sum_{x,y \in \ZZ}
		c(x,y) \eta_{x}( \alpha + \eta_{y})
		\big( F (s, \eta^{x,y}(s)) - F( s, \eta(s)) \big)
		^{2}
		\ud s\,.
	\end{aligned}
\end{equation}
Taking expectation under the stationary law and using the adjointness identities
\eqref{eq_adjointness} yields the identity
\begin{equation*}
	\begin{aligned}
		\mathbf E_{\rho,L}\big[\langle M^L(F)\rangle_t\big]
		=
		2L^{2}
		\int_0^t
		\langle F(s,\cdot),-\gens F(s,\cdot)\rangle_\rho
		\ud s\,.
	\end{aligned}
\end{equation*}
In particular, the extension in Section~\ref{sec_dynkin} applies to
$F(t,\eta):=u_L(t,\varphi;\eta)$, so the fluctuation field can be decomposed
using \eqref{eq_dynkin_mart} and its martingale bracket can be computed from
\eqref{eq_qv}.

\subsection{Canonical measures on finite boxes}

In Section~\ref{sec_ip}, we considered the grand canonical invariant measures $
	\nu_{\rho}$ for which the particle number was undetermined.
On the other hand, we may fix a box of length $ \ell \in \NN$ of the form
\begin{equation*}
	\begin{aligned}
		\mathbf{B}_{x_{0}, \ell}:=
		\{
		x_{0}, x_{0}+1, \ldots, x_{0}+ \ell -1
		\}\subset \ZZ\,, \quad x_{0} \in \ZZ\,
	\end{aligned}
\end{equation*}
and write $ N_{x_{0}, \ell} ( \eta)  = \sum_{x \in
	\mathbf{B}_{x_{0}, \ell}} \eta_{x}$ for the total
number of particles in the box $ \mathbf{B}_{x_{0} , \ell}$.
Moreover, we define the shift operators $ \tau_{x}$ that shift particle
configurations by $ x$, i.e. $ (\tau_{x} \eta)_{y}= \eta_{x+y}$.
The operators can be extended to act on functions
$ f : \Omega \to \RR$ by
\begin{equation*}
	\begin{aligned}
		(\tau_{x}f)( \eta) = f ( \tau_{x} \eta) \,.
	\end{aligned}
\end{equation*}
Hence, it will suffice to consider $ x_{0} =0 $ for which we simply
write $ \mathbf{B}_{\ell} $ for the corresponding box of length $ \ell$.
Accordingly, we fix the state space
\begin{equation*}
	\begin{aligned}
		\Omega_{\ell ,n}:= \Big\{ \eta \in \NN_{0}^{\mathbf{B}_{ \ell}}\, :\,
		N_{ \ell}( \eta) = n  \Big\}
	\end{aligned}
\end{equation*}
of particle configurations in the box with total particle number $n$.

Conditioning $ \nu_{\rho}$ onto $ \Omega_{\ell, n}$ yields the
\emph{canonical distribution} $\pi_{\ell, n }$, which is a
 Dirichlet–multinomial distribution with
probability mass function
\begin{equation}\label{eq_def_canonical}
	\begin{aligned}
		\pi_{\ell ,n} ( \eta) := \frac{1}{Z_{\ell, n}} \prod_{x \in
			\mathbf{B}_{\ell}} w( \eta_{x})\,, \quad \forall \eta \in \Omega_{\ell ,n}\,,
	\end{aligned}
\end{equation}
where $ Z_{\ell ,n} = \tfrac{ \Gamma (n + \alpha\ell)}{ n! \Gamma ( \alpha\ell ) }$ \cite{GRV11}. Recall the definition of $ w $ in \eqref{eq_nu}.
These canonical distributions are stationary with respect to the symmetric inclusion process on $ \mathbf{B}_{\ell}$
generated by
\begin{equation}\label{eq_Srefell}
	\begin{aligned}
		\gens_{\rm ref}^{(\ell)}f ( \eta)
		:=
		\frac{1}{2}
		\sum_{ \substack{x,y \in \mathbf{B}_{\ell} \\ |x-y| =1 }}
		\eta_{x}( \alpha+ \eta_{y} ) \big[ f ( \eta^{x,y}) - f ( \eta) \big]\,.
	\end{aligned}
\end{equation}
In fact, the measure $ \pi_{\ell, n}$ is the unique invariant measure with respect to $
	\gens_{\rm ref}^{(\ell)}$ and finite state space $ \Omega_{\ell ,n}$ since the underlying process is
irreducible.
We will write $ \EE_{\ell ,n} $ and $ \var_{\ell ,n}$ for expectation and variance with
respect to $ \pi_{\ell ,n}$, respectively.

\subsection{The quadratic current observable and its block
	replacement}\label{sec_intro_quadratic}

The central object of our analysis is the martingale formulation \eqref{eq_dynkin_mart} applied to the
density fluctuation field \eqref{eq_def_u}. Consequently, we will encounter the
anti-symmetric contribution of the generator $ \gen $ taking the form
\begin{equation}\label{eq_flux_A}
	\begin{aligned}
		\gena \eta_{x} = j_{x-1 , x}( \eta) - j_{x , x+ 1}(
		\eta)\,,
	\end{aligned}
\end{equation}
which is quadratic in the particle configuration.
To control this term, we expand it around the density $ \rho$:
\begin{equation}\label{eq_current_expansion}
	\begin{aligned}
		j_{x, x+1} ( \eta)
		= j ( \rho ) + \frac{1}{ 2} j' (\rho) \big(( \eta_{x}-
		\rho) + ( \eta_{x+1}- \rho) \big)
		+ \frac{1}{2} j''( \rho) W_{x}( \eta ) \,,
	\end{aligned}
\end{equation}
where
\begin{equation}\label{eq_def_Wandff}
	\begin{aligned}
		W_{x}( \eta) := (
		\eta_{x}- \rho)
		( \eta_{x+1}- \rho) = \tau_{x}\ff ( \eta)  \,,
		\quad
		\ff ( \eta) :=  ( \eta_{0} - \rho) (
		\eta_{1} - \rho) \,.
	\end{aligned}
\end{equation}
The identity \eqref{eq_current_expansion} can be checked by hand, after recalling \eqref{eq_def_current} and
\eqref{eq_def_j}.

It turns out that the constant $ j( \rho)$ part will not contribute since it cancels in \eqref{eq_flux_A}, while the linear
term with $ j' ( \rho)$ is precisely compensated by the moving frame in the
definition of $ u_{L}$ \eqref{eq_def_u}.
Therefore, the leading order contribution that remains is the centred quadratic local
observable $ W ( \eta)$, which is the part that will be identified as the Burgers
nonlinearity.

While the microscopic observable $W_x(\eta)$ is well defined, the difficulty lies in identifying the scaling limit of its space-time average. The density fields $u_L$ converge only as random distributions, which does not allow one to pass directly to their squares. 
 The second-order Boltzmann–Gibbs principle resolves this problem by replacing $W_x$, within the relevant space-time averages, by the centred square of a mesoscopic density average:
\begin{equation}\label{eq_def_etabar}
	\begin{aligned}
		\overline{\eta}^{\ell}_{0} :=
		\frac{1}{ \ell} \sum_{y \in \mathbf{B}_{ \ell}} \eta_{y}\,.
	\end{aligned}
\end{equation}
More precisely, the second order Boltzmann--Gibbs principle controls the error when
replacing $ W_{0} ( \eta) = \ff( \eta)  $ by the (Wick renormalised) centred field
\begin{equation}\label{eq_Q}
	\begin{aligned}
		\mQ_{\rho}( \ell , \eta)
		:=
		\big( \overline{\eta}_{0}^{\ell}- \rho \big)^{2} - \frac{\chi ( \rho)}{
			\ell} \,.
	\end{aligned}
\end{equation}

\begin{proposition}[Second order Boltzmann--Gibbs principle for the current]\label{prop_BG}
	Let $ A>0 $ and $ \rho>0 $. Then for every $\ell \geqslant 2$, $ t \geqslant 0 $,
	$ \alpha \leqslant A$, and any $ L\in\NN $:
	\begin{equation*}
		\begin{aligned}
			\mathbf E_{\rho,L}
			\bigg[
				\bigg(
				\int_{0}^{t}
				\sum_{x \in \ZZ} \tau_{x}
				\big(
				\ff ( \eta(s))
				-
				\mQ_{\rho}(\ell , \eta (s))
				\big) v_{s,x}
				\ud s
				\bigg)^2
				\bigg]
			\lesssim_{\rho,A}
			\frac{1}{ \alpha^{3} L^{2}} \Big(\frac{1}{ \alpha}
			+
			\frac{\ell^{2}}{  \alpha \ell +1 }
			+
			\frac{tL^{2}}{ \ell^2}
			\Big)
			%C_{\mathrm{BG}} \Big(
			%	\frac{\ell}{L^{2}}
			%		 + \frac{t}{\ell^{2}}
			%		\Big)
			\int_{0}^{t} \| v_{s,\cdot}
			\|^{2}_{\ell^{2}} \ud s\,,
		\end{aligned}
	\end{equation*}
	for any $ v \in C^{\infty}([0,T]; \ell^{1}( \ZZ))$.
\end{proposition}

The second order Boltzmann--Gibbs principle is classically stated for general
local functions. Proposition~\ref{prop_BG} is instead current-specific which is
sufficient to derive density fluctuations.  The proof of the proposition follows the
strategy of \cite{GoncJara14}.
The main work here is to make this strategy compatible
with the asymptotically quadratic inclusion rates $ \eta_x(\alpha+\eta_y)$.
We therefore require inclusion-specific input, namely a
spectral gap estimate uniform in the block particle number \cite{KimSau}
%, moment bounds for the negative-binomial invariant measures with explicit dependence on
%$\alpha$,
and a current-specific second-order equivalence-of-ensembles
expansion for $\ff(\eta)=(\eta_0-\rho)(\eta_1-\rho)$.

The case of unbounded rates was also considered in
\cite{GoncJaraSethuraman15} for zero-range-type dynamics with linearly
controlled rates.
We also mention the work \cite{GoncJaraSimon17}, which establishes a Boltzmann--Gibbs principle for polynomial local
functions and does not require a spectral-gap input. There elliptic and
bounded weights of the Dirichlet form are necessary, which do not seem to  extend trivially to the inclusion process.

\section{Current specific second--order Boltzmann--Gibbs principle}

The proof of the Boltzmann--Gibbs principle
(Proposition~\ref{prop_BG}) follows the one--block, two--block and
renormalisation scheme of \cite[Section~4.2]{GoncJara14}.  For this we combine the general
$H^{-1}$ estimate in Corollary~\ref{cor_kp_postHminus1} below with the current--specific equivalence--of--ensembles
input of Lemma~\ref{lem_equivensemble}.
First, we will establish these two inputs, before discussing the proof of the second
order Boltzmann--Gibbs principle in
Section~\ref{sec_proof_BG}.

%\textbf{$H^{-1}$ norm.} For $ f \in L^{2}( \nu_{\rho}) $, we define the $
%H^{-1, ( L)}$--norm of $ f$ through the variational formula
%\begin{equation}\label{eq_Hminus1_L}
%		\begin{aligned}
%				\| f \|_{-1, L}^{2} :=
%				\sup_{ g \text{ local}}
%				\big\{
%					2 \langle f, g \rangle_{\rho} - L^{2} \langle g,
%					- \gens g\rangle_{\rho}\,.
%				\big\}\,.
%			\end{aligned}
%		\end{equation}
%		In fact, it will be more convenient to work directly with
%		\begin{equation}\label{eq_Hminus1}
%			\begin{aligned}
%				\| f \|_{-1}^{2}:=
%				\sup_{ g \text{ local}}
%				\big\{
%					2 \langle f, g \rangle_{\rho} - \langle g,
%					- \gens g\rangle_{\rho}\,.
%				\big\}\,.
%			\end{aligned}
%		\end{equation}
%		Notice that
%		\begin{equation*}
%			\begin{aligned}
%				\| f \|_{-1, L}^{2}= L^{-2} \| f \|_{-1}^{2}\,,
%			\end{aligned}
%		\end{equation*}
%		since $ g $ in \eqref{eq_Hminus1_L} can simply be scaled with $
%		L^{-2}$
%		and reabsorbed into the supremum.

\subsection{Spectral gap inequality}\label{sec_sg}

Denote by $ ( \eta^{(K)}(t) )_{t \geqslant 0}$ the inclusion process on the periodic torus $
	\TT_{K} = \{ -K , -K+1, \ldots, K \}$ which approximates $ (\eta(t))_{t \geqslant 0 }$, see Appendix~\ref{app_asip}.
We write
$\gen_L^{(K)}$ for the corresponding periodic generator on $\TT_K$, and define
its symmetric part by
\begin{equation}\label{eq_def_sym_K}
	\gens^{(K)}
	:=
	\frac{1}{2L^2}
	\big(\gen_L^{(K)}+(\gen_L^{(K)})^*\big).
\end{equation}
Furthermore, we denote by $ \nu_{\rho}^{\TT_K}$ the product measure on $ \NN_0^{\TT_K}$.
Then, for $h\in L^2(\nu_\rho^{\TT_K})$, define the usual finite-volume variational norm
\begin{equation}\label{eq_def_hminus1_finitevol}
	\|h\|_{-1,\TT_K}^2
	:=
	\sup_g
	\big\{
	2\langle h,g\rangle_{\rho,\TT_K}
	-
	\langle g,-\gens^{(K)}g\rangle_{\rho,\TT_K}
	\big\},
\end{equation}
where the supremum is taken over functions $ g \in L^{2}( \nu_{\rho}^{\TT_K})$ for which $ \langle
	g,-\gens^{(K)}g\rangle_{\rho,\TT_K}< \infty$.

The Kipnis--Varadhan inequality  (Lemma~\ref{lem_local_kv_finite_volume})
provides a bound of additive functionals of the form
\begin{equation*}
	\begin{aligned}
		\mathbf E_{\rho,L}
		\left[
			\sup_{t \in [0,T]}
			\left|
			\int_0^t F(s,\eta(s))\ud s
			\right|^2
			\right]
		\leq
		C_{\mathrm{KV}} L^{-2}
		\liminf_{K\to\infty}
		\int_0^T
		\|F(s,\cdot)\|_{-1,\TT_K}^2\ud s\,.
	\end{aligned}
\end{equation*}
For this estimate to become useful, we need to
establish suitable bounds on the $ \| \cdot \|_{-1,\TT_K}$--norm. As in the
original work \cite{GoncJara14}, we use a spectral gap inequality (or Poincar\'e
inequality), which was recently established for the inclusion process~\cite{KimSau}.

\begin{lemma}[Spectral gap inequality]\label{lem_sg}
	For every $ \ell \geqslant 2 $, $ N \in \NN $ and $ f : \Omega_{\ell,N} \to \RR $, we have
	\begin{equation*}
		\begin{aligned}
			\var_{\ell, N} (f)
			\leqslant
			\frac{\ell^2}{4\alpha(1\wedge\alpha)}
			\langle f, - \gens_{\mathrm{ref}}^{( \ell )}
			f\rangle_{\ell , N}\,,
		\end{aligned}
	\end{equation*}
	where $ \gens_{\mathrm{ref}}^{( \ell )}$ generates the reflected 
	inclusion process on the segment $ \mathbf{B}_{\ell}$, see \eqref{eq_Srefell}.
\end{lemma}

\begin{proof}
	For the reflected (attempted jumps outside of $ \mathbf{B}_{\ell}$ are suppressed) inclusion process
	\cite[Theorem~1.1]{KimSau} establishes that
	\begin{equation}\label{eq_gap}
		\begin{aligned}
			\frac{\mathrm{gap}_{\mathrm{SIP}}( \ell, \alpha)}{
			\mathrm{gap}_{\mathrm{RW}}( \ell , \alpha)} \in [1 \wedge \alpha,
				1]\,,
		\end{aligned}
	\end{equation}
	with the spectral gap defined in terms of
	\begin{equation*}
		\begin{aligned}
			\mathrm{gap}_{\mathrm{SIP}}( \ell, \alpha)
			:= \inf_{N \in \NN} \inf_{\substack{ f: \Omega_{\ell,N} \to
			\RR \\ \text{not constant}}}
			\frac{\langle f, - \gens_{\mathrm{ref}}^{( \ell )} f\rangle_{\ell,N }}{\var_{\ell,N} (f)} \,,
		\end{aligned}
	\end{equation*}
	and $\mathrm{gap}_{\mathrm{RW}}$ defined analogously with $ \mS $ replaced by the
	 generator of a single random walker with speed $ \alpha$ and $ \pi_{\ell ,N}$ replaced by the
	uniform measure on the line segment $ \mathbf{B}_{\ell}$.

	In particular, \eqref{eq_gap} implies that for every $ N \in \NN$
	\begin{equation*}
		\begin{aligned}
			\var_{\ell, N} (f)
			\leqslant
			\frac{1}{ 1 \wedge \alpha}
			\frac{1}{\mathrm{gap}_{\mathrm{RW}}( \ell , \alpha)}
			\langle f, - \gens_{\mathrm{ref}}^{( \ell )} f\rangle_{\ell,N }\,.
		\end{aligned}
	\end{equation*}
	Next, denote by $P$ the discrete time transition matrix of a symmetric random
	walk on $ \mathbf{B}_{\ell}$ with holding probability $\tfrac{1}{2}$ at the
	boundaries. The eigenvalues of $P$ are explicitly
	given in terms of
	\begin{equation*}
		\begin{aligned}
			\cos{\big( \tfrac{ \pi i }{\ell} \big)} \,, \quad i =0, \ldots,
			\ell -1\,,
		\end{aligned}
	\end{equation*}
	see for example \cite[Example~12.11]{LPW17}.
	Now, because $\alpha (P- \mathrm{id})$
	generates the continuous time (reflected) random walk of speed $ \alpha$, we have
	\begin{equation*}
		\begin{aligned}
			\mathrm{gap}_{\mathrm{RW}}( \ell , \alpha)
			=
			\alpha \Big(
			1- \cos{\big( \tfrac{\pi}{\ell} \big)}
			\Big)
			\geqslant
			\frac{4 \alpha}{\ell^{2}}
			%\frac{ \alpha \pi^{2}}{2 \ell^{2}}
			\,, \quad \forall \ell \geqslant 2\,.
		\end{aligned}
	\end{equation*}
	%It is only left to argue that we can replace the generator $
	%\gens_{\mathrm{ref}}^{( \ell )}$ in \eqref{eq_supp1_sg} with $ \gens^{(K)}$, that
	%is because
	%opening up further jump rates along edges only increases the Dirichlet energy,
	%see e.g. \eqref{eq_qv}.
	%Hence, $ \langle f, - \gens_{\mathrm{ref}}^{( \ell )} f\rangle_{\ell,N }
	%\leqslant \langle f, - \gens^{(K)} f\rangle_{\ell,N }$,
	This concludes the proof of the lemma.
\end{proof}

The spectral gap inequality allows us to derive an upper bound on the
$ H^{-1}$-norm.
First, we define
the $ \sigma$--algebra
\begin{equation*}
	\begin{aligned}
		\mF_{x_{0},\ell}
		:=
		\sigma \big( N_{x_{0}, \ell} ( \eta) , \eta_{y} ; y \notin \mathbf{B}
		_{x_{0},\ell}\big)\,,
	\end{aligned}
\end{equation*}
which contains information about the particle configuration outside the box and
the total number of particles present within the box.
Whenever clear from context, we will use this notation for both the periodic process and the
infinite-volume process, without distinction.
Thus, for any local function  $ f$ supported on the box $
	\mathbf{B}_{x_{0} , \ell}$, we have
\begin{equation*}
	\begin{aligned}
		\EE_{\rho} [ f | \mF_{x_{0}, \ell}]
		=
		\EE_{ \ell , N_{x_{0}, \ell}}[ f ] \,,
	\end{aligned}
\end{equation*}
as well as $ \EE_{\rho}^{(K)} [ f | \mF_{x_{0}, \ell}]
	=
	\EE_{ \ell , N_{x_{0}, \ell}}[ f ]$ for $K$ large enough.
Here the expectation on the right--hand side is with respect to the
canonical distribution $ \pi_{\ell , N_{x_{0}, \ell}}$ on $
	\mathbf{B}_{x_{0}, \ell}$.

\begin{proposition}\label{prop_Hminus1_bound}
	Let $ m \in \NN $ and $ \mathbf{B}_{x_{i}, \ell_{i}} \subset \ZZ$ be
	disjoint
	blocks of length $ \ell_{i} \geqslant 2$, $ i =1, \ldots, m$.
	For every $ i =1, \ldots, m $, let
	$f_{i} \in L^{2}( \nu_{\rho}) $ be
	supported on $
		\mathbf{B}_{x_{i}, \ell_{i}}$ such that
	$\EE_{\rho}
		[ f_{i}| \mF_{x_{i}, \ell_{i}}] = 0$.
	Then for every $K \in \NN$ large enough
	\begin{equation*}
		\begin{aligned}
			\Big\| \sum_{i =1}^{m} f_{i}\Big\|_{-1,\TT_K}^{2}
			\leqslant
			\frac{1}{4\alpha(1\wedge\alpha)}
			\sum_{i =1}^{m} \ell_{i}^{2}
			\langle f_{i}, f_{i} \rangle_{\rho,\TT_K}\,.
		\end{aligned}
	\end{equation*}
\end{proposition}

The proof is almost analogous to \cite[Proposition~3.4]{GoncJara14}, however,
with a slightly adjusted argument since the particle number on a box is unbounded
(see also \cite[Lemma~4.3]{GoncJaraSethuraman15}).
We present the proof for completeness.

\begin{proof}
	Let $K$ be large enough such that the blocks embed into $ \TT_{K}$ without wrapping
	around the periodic boundary.
	First, consider the case $ m =1$ (for convenience we assume $
		x_{1}= 0 $ and write $f = f_{1}$ and $ \ell = \ell_{1}$).
	Then for every local function $ g : \NN_{0}^{\TT_{K}} \to \RR $
	\begin{equation*}
		\begin{aligned}
			\langle f, g \rangle_{\rho,\TT_{K}}
			=
			\EE_{\rho}^{(K)}[ f g ]
			=
			\EE_{\rho}^{(K)}\big[ f ( g - \EE_{\rho}^{(K)}[ g | \mF_{0,\ell}]
			)\big]\,,
		\end{aligned}
	\end{equation*}
	using the tower property and that  $ \EE_{\rho}^{(K)}[ f \EE_{\rho}^{(K)}[ g | \mF_{0 , \ell}] | \mF_{0 , \ell}]
		=0$ by the assumption on $ f$. Thus, the Cauchy--Schwarz inequality
	implies
	\begin{equation*}
		\begin{aligned}
			\big| \langle f, g \rangle_{\rho,\TT_{K}}\big|
			\leqslant
			\EE_{\rho}^{(K)}[ f^{2}]^{\frac{1}{ 2}}
			\EE_{\rho}^{(K)} [ \var_{\rho}^{(K)}( g | \mF_{0 ,
				\ell})]^{\frac{1}{2}} =
			\EE_{\rho}^{(K)}[ f^{2}]^{\frac{1}{ 2}}
			\EE_{\rho}^{(K)}[
			\var_{\ell , N_{0 , \ell}}(g)
			]^{\frac{1}{2}}\,.
		\end{aligned}
	\end{equation*}
	Now, applying first the spectral gap inequality from Lemma~\ref{lem_sg} and
	then Young's product inequality, we see that
	\begin{equation}\label{eq_supp_sg_new}
		\begin{aligned}
			2 \langle f, g \rangle_{\rho,\TT_{K}}
			 & \leqslant
			2
			\EE_{\rho}^{(K)}[ f^{2}]^{\frac{1}{ 2}}
			\bigg(
			\frac{\ell^2}{4\alpha(1\wedge\alpha)}
			\EE_{\rho}^{(K)}[\langle g, -
			\gens^{(\ell)}_{\mathrm{ref}} g \rangle_{\ell , N_{0 ,
						\ell}}]
			\bigg)^{\frac{1}{2}} \\
			 & \leqslant
			%\frac{1}{\lambda}
			\frac{\ell^2}{4\alpha(1\wedge\alpha)}
			\EE_{\rho}^{(K)}[ f^{2}]
			+
			%\lambda
			\langle g, -\gens^{(K)} g \rangle_{ \rho,\TT_{K}} \,,
		\end{aligned}
	\end{equation}
	where
	we also used that
	the finite--volume reflective Dirichlet form is bounded
	by the corresponding finite--volume periodic Dirichlet form:
	\begin{equation*}
		\begin{aligned}
			\EE_{\rho}^{(K)}\big[
			\langle g, -\gens^{(\ell)}_{\mathrm{ref}} g \rangle_{\ell , N_{0 , \ell}}\big]
			 & =
			\frac{1}{4}\EE_{\rho}^{(K)}\bigg[
			\sum_{\substack{x,y \in \mathbf{B}_{\ell}   \\ |x-y|=1}}
				\eta_x(\alpha+\eta_y)
				\big(g(\eta^{x,y})-g(\eta)\big)^2
			\bigg]                                      \\
			 & \leqslant
			\frac{1}{4}\EE_{\rho}^{(K)}\bigg[
			\sum_{\substack{x \in \TT_{K} \,, y \in \ZZ \\ |x-y|=1}}
				\eta_x(\alpha+\eta_y)
				\big(g(\eta^{x,y})-g(\eta)\big)^2
				\bigg]
			=
			\langle g, -\gens^{(K)} g \rangle_{ \rho, \TT_{K}}\,,
		\end{aligned}
	\end{equation*}
	where $y$ in the second sum is interpreted in the periodic sense in $
		\TT_{K}$.
	Consequently, by rearranging the inequality \eqref{eq_supp_sg_new} and taking
	the supremum with respect to $ g$, we have
	\begin{equation*}
		\begin{aligned}
			\| f \|_{-1, \TT_{K}}^{2}
			= \sup_{g}
			\big\{
			2 \langle f, g \rangle_{\rho,\TT_{K}}
			-
			\langle g, -\gens^{(K)} g \rangle_{ \rho,\TT_{K}}
			\big\}
			\leqslant
			\frac{\ell^2}{4\alpha(1\wedge\alpha)}
			\EE_{\rho}^{(K)}[ f^{2}]\,,
		\end{aligned}
	\end{equation*}
	which concludes the proof in the case $ m =1$.

	For general $ m \in \NN$, we apply the preceding one-block estimate to each
	$f_{i}$ individually
	\begin{equation*}
		\begin{aligned}
			2 \sum_{i =1}^{m}
			\langle f_{i} , g \rangle_{\rho,\TT_{K}}
			\leqslant
			\sum_{i =1}^{m}
			\frac{\ell_i^2}{4\alpha(1\wedge\alpha)}
			\EE_{\rho}^{(K)}[ f_{i}^{2}]
			+
			%\lambda
			\sum_{i =1}^{m}
			\EE_{\rho}^{(K)}[ \langle g, -\gens^{( \ell_{i})}_{\rm ref} g \rangle_{\ell_{i} ,
				N_{ x_{i} , \ell_{i}}}]
			\,,
		\end{aligned}
	\end{equation*}
	and
	use the fact that the blocks $
		\mathbf{B}_{x_{i}, \ell_{i}}$ are disjoint which yields together with
	the comparison estimate above
	\begin{equation*}
		\begin{aligned}
			\EE_{\rho}^{(K)}\bigg[
				\sum_{i =1}^{m}
				\langle g, -\gens^{(\ell_{i})}_{\rm ref} g \rangle_{\ell_{i} ,
					N_{ x_{i} , \ell_{i}}}\bigg]
			\leqslant  \langle g, - \gens^{(K)} g \rangle_{\rho,\TT_{K} }\,.
		\end{aligned}
	\end{equation*}
	Again, rearranging and taking the supremum with respect to $ g$ concludes
	the proof.
\end{proof}

Combining Lemma~\ref{lem_local_kv_finite_volume} and
Proposition~\ref{prop_Hminus1_bound},
we conclude the improved Kipnis--Varadhan inequality, representing the first (out
of two) relevant inputs for the Boltzmann--Gibbs principle.

\begin{corollary}\label{cor_kp_postHminus1}
	Let $ m \in \NN $ and $ \mathbf{B}_{x_{i}, \ell_{i}} \subset \ZZ$ disjoint
	blocks of length $ \ell_{i} \geqslant 2$, $ i =1, \ldots, m$.
	For every $ i =1, \ldots, m $, let $ F_{i} \in  C^{\infty}([0,T];
		L^{2}( \nu_{\rho}))$ be  supported on $ \mathbf{B}_{x_{i}, \ell_{i}}$ satisfying $
		\EE_{\rho} [ F_{i}(t, \eta) | \mF_{x_{i}, \ell_{i}}] = 0 $, for almost every $t \in [0,T]$.
	Then
	%for every $ t \in [0,T] $
	\begin{equation*}
		\begin{aligned}
			\mathbf E_{\rho,L}\bigg[
			\sup_{ t \in [0, T]}
			\bigg|
			\int_{0}^{t}	\sum_{i =1}^{m} F_{i}( s , \eta (s)) \ud
			s
			\bigg|^{2}
			\bigg]
			\leqslant
			C_{\mathrm{KV}} L^{-2}
			\sum_{i =1}^{m}
			\frac{\ell_i^2}{4\alpha(1\wedge\alpha)}
			\int_{0}^{ T}
			\langle F_{i}(s, \cdot), F_{i}(s, \cdot)
			\rangle_{\rho} \ud s\,.
		\end{aligned}
	\end{equation*}
\end{corollary}

\begin{proof}
	Let $K$ be large enough that the blocks are embedded in $\TT_K$ without
	wrapping and remain disjoint, we can apply
	Proposition~\ref{prop_Hminus1_bound} which
	gives
	\begin{equation*}
		\Big\|\sum_{i=1}^{m}F_i(s,\cdot)\Big\|_{-1,\TT_K}^{2}
		\leqslant
		\frac{1}{4\alpha(1\wedge\alpha)}
		\sum_{i=1}^{m}\ell_i^2
		\langle F_i(s,\cdot),F_i(s,\cdot)\rangle_{\rho,\TT_K}\,,
	\end{equation*}
	for all such $K$ and almost every $s$.
	Moreover, for $K$ large enough the $\nu_\rho^{\TT_K}$-marginal on the union of the
	blocks agrees with the infinite-volume $\nu_\rho$-marginal, so
	\begin{equation*}
		\langle F_i(s,\cdot),F_i(s,\cdot)\rangle_{\rho,\TT_K}
		=
		\langle F_i(s,\cdot),F_i(s,\cdot)\rangle_{\rho}\,.
	\end{equation*}
	It is only left to apply Lemma~\ref{lem_local_kv_finite_volume}, which we
	can do because
	for
	$K$ large enough
	\begin{equation*}
		\begin{aligned}
			\EE_{\rho}^{(K)}\bigg[
				F_{i}(s, \eta) \bigg| \sum_{x \in \TT_{K}} \eta_{x}
				\bigg]
			=
			\EE_{\rho}^{(K)}\bigg[
				\EE_{\rho}\big[
					F_{i}(s, \eta) \big| \mF_{x_{i} ,\ell_{i}}\big] \bigg| \sum_{x \in \TT_{K}} \eta_{x}
				\bigg]
			=0\,,
		\end{aligned}
	\end{equation*}
	for almost every $ s$, by the tower property.
	Moreover, we pick up an extra $ L^{-2}$-factor in front of the $H^{-1}$-norm which is an artefact of the speed-up in
	\eqref{eq_genIP}:
	\begin{equation*}
		\begin{aligned}
			\sup_g
			\big\{
			2\langle f,g\rangle_{\TT_K,N}
			-
			L^2\langle g,-\gens^{(K)}g\rangle_{\TT_K,N}
			\big\}
			 & =
			L^{-2}
			\sup_g
			\big\{
			2\langle f,g\rangle_{\TT_K,N}
			-
			\langle g,-\gens^{(K)}g\rangle_{\TT_K,N}
			\big\}\,.
		\end{aligned}
	\end{equation*}
	This finishes the proof.
\end{proof}

Note that Corollary~\ref{cor_kp_postHminus1} is the local-polynomial
version of the input used for the weakly asymmetric simple exclusion
process \cite[Corollary~3.5]{GoncJara14}.

\subsection{Equivalence of ensembles}\label{sec_equivens}

We restrict ourselves to the box $ \mathbf{B}_{
	\ell_{0}}$, $\ell_{0} \geq 2$, and the local function $ \ff (\eta) = ( \eta_{1}- \rho) ( \eta_{0}- \rho) $.
For every $ \ell \geqslant \ell_{0}$, we define the centred conditional expectation (which is independent of $ \ell_{0}$)
\begin{equation}\label{eq_psi}
	\begin{aligned}
		\psi_{\ff}( \ell ; \eta) := \EE_{\rho}[ \ff|
			\overline{\eta}^{\ell}_{0}]
		=
		\EE_{ \ell , N_{ 0 , \ell}}[ \ff]
		\,,
	\end{aligned}
\end{equation}
where $ \overline{\eta}^{\ell}_{0}$ was defined in \eqref{eq_def_etabar}.
The random variable $ \psi_{\ff}$ is the best possible guess of the mean behaviour of $
	\ff$
under $ \nu_{\rho}$ conditioned on the number of particles inside the box.

\begin{lemma}[Second order equivalence of ensembles for the quadratic current]\label{lem_equivensemble}
	Let $ A \in (0, \infty) $.
	For every $ \ell \geqslant 2$ and $ \alpha \leqslant A $
	\begin{equation*}
		\begin{aligned}
			\EE_{\rho}\Big[
			\big|
			\psi_{\ff}( \ell ; \eta) \big|^{2}
			\Big]
			\lesssim_{\rho,A}
			\frac{1}{ \alpha \ell ( \alpha \ell +1)}\,,
			%\red{\alpha^{-6}\ell^{-2} }
			\qquad
			\EE_{\rho}\Big[
			\big|
			\psi_{\ff}( 2\ell ; \eta)
			-  \psi_{\ff}( \ell ; \eta) \big|^{2}
			\Big]
			\lesssim_{\rho,A}
			\frac{1}{ \alpha \ell ( \alpha \ell +1)}
			%\red{\alpha^{-6}\ell^{-2}}
			\,,
		\end{aligned}
	\end{equation*}
	\begin{equation*}
		\begin{aligned}
			\text{and} \quad
			\EE_{\rho}\Big[
			\big|
			\psi_{\ff}( \ell ; \eta)
			-  \mQ_{\rho}( \ell , \eta)\big|^{2}
			\Big]
			\lesssim_{\rho, A}
			\frac{1}{ \alpha^{3} \ell^{3}}
			%\red{\alpha^{-6}\ell^{-3}}
			\,,
		\end{aligned}
	\end{equation*}
	where  $ \mQ_{\rho}$ was defined
	in \eqref{eq_Q}.
\end{lemma}

\begin{proof}
	First, we collect the estimates needed for the proof
	in the following display:
	\begin{equation}\label{eq_supp10_equivensembles}
		\begin{aligned}
			\EE_{\rho}\big[
				(\overline{\eta}^{\ell}_{0}-\rho)^{2}\big]
			= \frac{\chi(\rho)}{\ell} & \lesssim_{\rho}
			\alpha^{-1} \ell^{-1}\,,                                                                   \\
			\EE_{\rho}\big[
				(\overline{\eta}^{\ell}_{0}-\rho)^{4}\big]
			                          & \lesssim_{\rho} \alpha^{-2}\ell^{-2}+ \alpha^{-3}\ell^{-3}\,,  \\
			\EE_{\rho}\big[
				(\overline{\eta}^{\ell}_{0})^{q}\big]
			                          & \lesssim_{\rho,q}\alpha^{1-q}\,,\qquad \forall q\geqslant 1\,. 
		\end{aligned}
	\end{equation}
	The first two bounds use the product structure of $\nu_{\rho}$ and the
	one-site moments \eqref{eq_quant_mom_bound}, the third follows from Jensen's
	inequality.

	Next, using the  explicit weights and the partition
	function from \eqref{eq_nu} and \eqref{eq_def_canonical},
	we have
	\begin{equation*}
		\begin{aligned}
			\EE_{\ell , n} [ \eta_{0} \eta_{1}]
			 & = \EE_{\ell , n}\big[ \eta_{0} \EE_{\ell-1, n - \eta_{0}}[ \eta
					_{1}] \big]
			=
			\sum_{0 \leqslant k_{1}+ k_{2} \leqslant n}
			k_{1} w ( k_{1}) k_{2} w ( k_{2})
			\frac{Z_{\ell -2, n - k_{1}- k_{2}}}{Z_{\ell ,n}}
			= \frac{ \alpha n ( n -1)}{\ell ( \alpha\ell +1 )} \,,
		\end{aligned}
	\end{equation*}
	by explicit calculation.
	This identity allows us to rewrite $ \psi_{\ff}$ from \eqref{eq_psi} in terms
	of
	\begin{equation}\label{eq_supp_psi}
		\begin{aligned}
			\psi_{\ff}( \ell; \eta)
			 & =
			\frac{ \alpha N_{0,\ell} ( N_{0 ,\ell} -1)}{\ell ( \alpha\ell +1 )}
			- 2 \rho \frac{N_{0 , \ell}}{\ell} + \rho^{2}
			=
			\frac{ \alpha \ell}{ \alpha\ell +1}
			( \overline{\eta}^{\ell}_{0}- \rho)^{2}
			-
			\frac{\alpha + 2 \rho}{\alpha \ell +1}
			( \overline{\eta}^{\ell}_{0}- \rho)
			-
			\frac{\rho ( \alpha + \rho)}{\alpha \ell +1}
			%\red{=
			%( \overline{\eta}^{\ell}_{0}- \rho)^{2}
			%- \frac{ \chi( \overline{\eta}^{\ell}_{0})}{ \ell}
			%+ \overline{\eta}^{\ell}_{0}
			%\frac{\alpha + \overline{\eta}^{\ell}_{0}}{ \alpha \ell (
			%\alpha\ell +1)}}
			\,.
		\end{aligned}
	\end{equation}
	We derive the first estimate of the lemma by estimating the terms on the right
	separately
	\begin{equation*}
		\begin{aligned}
			\EE_{\rho}\Big[
			\big|
			\psi_{\ff}( \ell ; \eta) \big|^{2}
			\Big]
			\lesssim_{\rho}
			\frac{1}{( \alpha \ell +1)^{2}}
			\Big(
			\frac{1}{\alpha \ell}
			+1
			\Big)
			\lesssim_{\rho} \frac{1}{ \alpha \ell ( \alpha \ell +1)}
			\,.
		\end{aligned}
	\end{equation*}
	Likewise, the second bound follows by estimating each term in the difference separately.

	On the other hand, using the definition of $ \mQ_{\rho}$ \eqref{eq_Q}, we can write
	\begin{equation*}
		\begin{aligned}
			\psi_{\ff}( \ell; \eta)
			-
			\mQ_{\rho}( \ell , \eta)
			=
			-\frac{ 1}{ \alpha\ell +1}
			( \overline{\eta}^{\ell}_{0}- \rho)^{2}
			-
			\frac{\alpha + 2 \rho}{\alpha \ell +1}
			( \overline{\eta}^{\ell}_{0}- \rho)
			+
			\frac{ \rho ( \alpha + \rho)}{\alpha \ell (\alpha
				\ell +1)} \,.
		\end{aligned}
	\end{equation*}
	Thus, term-by-term estimates also yield the bound
	\begin{equation*}
		\begin{aligned}
			\EE_{\rho}\Big[
			\big|
			\psi_{\ff}( \ell ; \eta)
			-  \mQ_{\rho}( \ell , \eta)\big|^{2}
			\Big]
			\lesssim_{\rho}
			\frac{1}{ ( \alpha \ell +1 )^{2}}
			\Big(
			\frac{1}{ \alpha \ell} + \frac{1}{
				\alpha^{2}\ell^{2}}
			+
			\frac{1}{ \alpha^{3}\ell^{3}}
			\Big)
			\lesssim_{\rho} \frac{1}{ \alpha^{3} \ell^{3} }
			\,.
		\end{aligned}
	\end{equation*}
	This concludes the proof.
	%OLD VERY NONOPTIMAL VERSION
	%\gray{
	%		It suffices to estimate the second moment of each of the terms on the right--hand
	%		side of \eqref{eq_supp1_equiv}
	%		individually
	%		\begin{itemize}
	%			\item First, notice that
	%				\begin{equation*}
	%					\begin{aligned}
	%						\EE_{\rho}\big[ \big(
	%								\overline{\eta}^{\ell}_{0}- \rho
	%						\big)^{2} \big]
	%						=
	%						\frac{1}{ \ell^{2}}
	%						\sum_{x \in \mathbf{B}_{0 , \ell}} \EE[ (
	%						\eta_{x}- \rho)^{2}] =
	%						\frac{ \chi( \rho)}{ \ell} \,.
	%					\end{aligned}
	%				\end{equation*}
	%
	%			\item On the other hand, using again that $ \nu_{\rho}$ is a product
	%				measure,
	%				\begin{equation*}
	%					\begin{aligned}
	%						\EE_{\rho}\big[ (
	%								\overline{\eta}^{\ell}_{0}- \rho
	%						)^{4} \big]
	%						=
	%						&\frac{1}{ \ell^{4}}
	%						\sum_{x \in \mathbf{B}_{0 , \ell}}
	%						\EE[ (\eta_{x}- \rho)^{4}]
	%						+ \frac{3}{\ell^{4}}
	%						\sum_{\substack{x,y \in \mathbf{B}_{0 , \ell}\\ x \neq y}}
	%						\EE[ (\eta_{x}- \rho)^{2}]
	%						\EE[ (\eta_{y}- \rho)^{2}]
	%						\lesssim_\rho \alpha^{-3}\ell^{-2}\,,
	%					\end{aligned}
	%				\end{equation*}
	%				where we used the fourth moment of $\nu_\rho$.
	%
	%			\item Similarly, for the last term, we use that by Jensen's inequality
	%				(for $ q \geqslant 1$)
	%					\begin{equation*}
	%						\begin{aligned}
	%							\EE_{\rho}[
	%							(\overline{\eta}^{\ell}_{0})^{q}]
	%						\leqslant
	%						\frac{1}{\ell}
	%						\sum_{x \in \mathbf{B}_{0, \ell}} \EE_{\rho}[
	%						\eta_{x}^{q}]
	%							\lesssim_{\rho,q}\alpha^{1-q}\,.
	%						\end{aligned}
	%					\end{equation*}
	%			\end{itemize}
	%		}
\end{proof}

Furthermore, we also note the following moment bound of
$ \mQ_{\rho}$ which is a consequence of its definition \eqref{eq_Q} and
\eqref{eq_supp10_equivensembles}
\begin{equation}\label{eq_secondmoment_Q}
	\begin{aligned}
		\EE_{\rho}\Big[
		\big|
		\mQ_{\rho}( \ell , \eta) \big|^{2}
		\Big]
		\leqslant
		2
		\EE_{\rho}\big[\big( \overline{\eta}_{0}^{\ell}- \rho \big)^{4}\big] + 2\frac{\chi ( \rho)
			^{2}}{ \ell^{2}}
		\lesssim_{\rho} \frac{1}{\alpha^{2}\ell^{2}}
		+ \frac{1}{\alpha^{3} \ell^{3}}
		\,,
	\end{aligned}
\end{equation}
which is further crudely bounded by $ \alpha^{-3} \ell^{-2}$.

\subsection{Proof of Proposition~\ref{prop_BG}}\label{sec_proof_BG}

We now show how the preceding estimates imply Proposition~\ref{prop_BG}.  The proof uses two different kinds of inputs, only one of which is specific to the current.  The $H^{-1}$-estimate in Corollary~\ref{cor_kp_postHminus1} is a
general estimate for centred local functions supported on disjoint blocks, which is needed to control the auxiliary block functions $ \psi$.
On the other hand, the $L^2(\nu_\rho)$ control of those
auxiliary functions is specific to the current observable
$\ff(\eta)=(\eta_0-\rho)(\eta_1-\rho)$, see
Lemma~\ref{lem_equivensemble}.
Because the proof follows closely \cite{GoncJara14}, we will only sketch the
steps and guarantee that the unbounded rates of the inclusion process are
controlled.

Our goal is to control the error incurred by
\begin{equation*}
	\begin{aligned}
		\tau_{x}
		\big(
		\ff ( \eta(t))
		-
		\mQ_{\rho}(\ell , \eta (t))
		\big)
		 & =
		\tau_{x}
		\big(
		\ff ( \eta(t))
		-  \psi_{\ff}( \ell; \eta(t)) \big)
		+
		%& \quad +
		%\sum_{k=1}^{M}
		%\tau_{x}
		%\big(
		%		\psi_{\ff}( 2^{k}; \eta(t))
		%		-  \psi_{\ff}( 2^{k+1}; \eta(t)) \big)\\
		%			& \quad +
		\tau_{x}
		\big(
		\psi_{\ff}( \ell; \eta(t))
		-
		\mQ_{\rho}(\ell , \eta (t))
		\big) \,.
	\end{aligned}
\end{equation*}
We control the contribution of each of the two differences on the right--hand side separately.
In the following, all upper bounds with respect to $ \alpha$ are to be interpreted in the
regime where $ \alpha$ is uniformly bounded from above (these bounds will
only be of interest to us when $ \alpha$ is small).
\par\smallskip\smallskip

\textbf{Step 1.} \emph{One block estimate.}
First, we notice that $ \EE_{\rho}[ \tau_{x}
		\big(
		\ff ( \eta)
		-  \psi_{\ff}( \ell_{0}; \eta) \big)| \mF_{x, \ell_{0}}]=0 $ by
definition of $ \psi_{\ff}$.
Consequently, we can apply Corollary~\ref{cor_kp_postHminus1} (together
with the Cauchy-Schwarz inequality) which yields
\begin{equation}\label{eq_oneblock}
	\begin{aligned}
		 & \mathbf E_{\rho,L}\bigg[
		\sup_{ t \in [0, T]}\bigg|
		\int_{0}^{t}
		\sum_{x \in \ZZ} \tau_{x}
		\big(
		\ff ( \eta(s))
		-
		\psi_{\ff}( \ell_{0}; \eta(s))
		\big) v_{s,x}
		\ud s
		\bigg|^{2}\bigg]                \\
		 & \leqslant \ell_{0}
		\sum_{i =1}^{\ell_{0}}
		\mathbf E_{\rho,L}\bigg[
		\sup_{ t \in [0, T]}\bigg|
		\int_{0}^{t}
		\sum_{x \in \ZZ} \tau_{\ell_{0} x+i}
		\big(
		\ff ( \eta(s))
		-
		\psi_{\ff}( \ell_{0}; \eta(s))
		\big) v_{s,\ell_{0}x+i}
		\ud s
		\bigg|^{2}\bigg]
		\\
		 & \lesssim  L^{-2}
		\sum_{i =1}^{\ell_{0}}
		\alpha^{-2}
		\ell_{0}^{2}
		\sum_{x \in \ZZ}
		\int_{0}^{T}
		\EE_{\rho}\big[\big|
		\ff ( \eta)
		-
		\psi_{\ff}( \ell_{0}; \eta)\big|^{2}
		\big]
		v_{s,\ell_{ 0}x+i}^{2}
		\ud s                           
		 \lesssim_{\rho}  \alpha^{-2}
		\frac{\ell_{0}^{2}}{\alpha^{2} L^{2}}
		\int_{0}^{T} \| v_{s, \cdot}\|_{\ell^{2}}^{2} \ud s\,,
	\end{aligned}
\end{equation}
where, in the last inequality, we used that $\psi_{\ff}$ is an $ L^{2}$-projection of $ \ff$, thus,	
\begin{equation*}
	\begin{aligned}
		\EE_{\rho}\big[\big|
		\ff ( \eta)
		-
		\psi_{\ff}( \ell_{0}; \eta)\big|^{2}\big]
		\leqslant
		\EE_{\rho}\big[\big|
		\ff ( \eta)
		\big|^{2}\big] = \chi( \rho)^{ 2} \lesssim_{\rho} \alpha^{-2}\,.
	\end{aligned}
\end{equation*}

\textbf{Step 2.} \emph{Renormalisation step and two block estimate.}
Similarly to the one block estimate, we can control the replacement across dyadic scales:
\begin{equation*}
	\begin{aligned}
		 & \mathbf E_{\rho,L}\bigg[\bigg|
		\int_{0}^{t}
		\sum_{x \in \ZZ} \tau_{x}
		\big(
		\psi_{\ff}( 2^{k}; \eta(s))
		-
		\psi_{\ff}( 2^{k+1}; \eta(s))
		\big) v_{s,x}
		\ud s
		\bigg|^{2}\bigg]                  \\
		 & \lesssim 2^{k+1}
		\sum_{i =1}^{2^{k+1}}
		\frac{1}{\alpha 2^k ( 1+ \alpha 2^k)}
		L^{-2}
		\alpha^{-2}
		2^{2(k+1)}
		\sum_{x \in \ZZ}
		\int_{0}^{t}
		v_{s,2^{k+1}x+i}^{2}
		\ud s
		\lesssim_{\rho}
		\frac{2^{2k}}{\alpha^{3}( \alpha 2^{k} +1 ) L^{2}}
		\int_{0}^{t} \| v_{s, \cdot}\|_{\ell^{2}}^{2} \ud s\,,
	\end{aligned}
\end{equation*}
by application of Corollary~\ref{cor_kp_postHminus1} and Lemma~\ref{lem_equivensemble}.
Notice that the zero-mean assumption in Corollary~\ref{cor_kp_postHminus1} is satisfied
because $\mF_{x,2^{k+1}}\subset\mF_{x,2^{k} }$, thus, the tower property gives
\begin{equation}\label{eq_tower}
	\begin{aligned}
		\EE_\rho\big[\tau_x\psi_{\ff}(2^{k};\eta)\,\big|\,\mF_{x,2^{k+1}}\big]
		=
		\EE_\rho\big[\tau_x\ff(\eta)\,\big|\,\mF_{x,2^{k+1}}\big]
		=
		\tau_x\psi_{\ff}(2^{k +1}; \eta)\,.
	\end{aligned}
\end{equation}
Consequently, by telescoping and the triangle inequality
\begin{equation*}
	\begin{aligned}
		&\mathbf E_{\rho,L}\bigg[\bigg|
		\int_{0}^{t}
		\sum_{x \in \ZZ}
		\tau_{x}
		\big(
		\psi_{\ff}( \ell_{0}; \eta(s))
		-
		\psi_{\ff}( 2^{M}; \eta(s))
		\big) v_{s,x}
		\ud s
		\bigg|^{2}\bigg]^{1/2}\\
		 & \lesssim_{\rho}
		\left\{ \sum_{k =1}^{M-1} \frac{2^{k}}{\sqrt{ \alpha
				2^{k}+1}} \right\} \frac{1}{ \alpha^{3/2} L}
		\bigg(\int_{0}^{t} \| v_{s, \cdot}\|_{\ell^{2}}^{2} \ud s\bigg)
		^{1/2}             
		 \lesssim_{\rho}
		\frac{1}{ \alpha^{3/2} L}
		\frac{2^M}{\sqrt{\alpha 2^M +1}}
		\bigg(\int_{0}^{t} \| v_{s, \cdot}\|_{\ell^{2}}^{2} \ud s\bigg)
		^{1/2}
		\,,
	\end{aligned}
\end{equation*}
where $M$ is the smallest integer such that $\ell\leqslant 2^{M}<2\ell$.
Squaring both sides yields the desired two block estimate provided $ \ell = 2^{M}$.
If $\ell$ is not dyadic, we apply one last time Corollary~\ref{cor_kp_postHminus1} and
Lemma~\ref{lem_equivensemble} (jointly with the triangle inequality):
\begin{equation*}
	\begin{aligned}
		\mathbf E_{\rho,L}\bigg[
		\bigg|
		\int_0^t
		\sum_{x\in\ZZ}
		\tau_x\big( \psi_{\ff}(2^M;\eta(s))-\psi_{\ff}(\ell;\eta(s))\big)v_{s,x}\,\ud s
		\bigg|^2
		\bigg]
		\lesssim_\rho
		\frac{\ell^{2}}{\alpha^3 ( \alpha \ell +1 ) L^2}
		\int_0^t\|v_{s,\cdot}\|_{\ell^2}^2\ud s\,,
	\end{aligned}
\end{equation*}
where the zero mean assumption holds for the same reason as in \eqref{eq_tower}.

Together with \textbf{Step 1.}, we conclude
\begin{equation}\label{eq_one_block}
	\begin{aligned}
		\mathbf E_{\rho,L}\bigg[\bigg|
		\int_{0}^{t}
		\sum_{x \in \ZZ} \tau_{x}
		\big(
		\ff ( \eta(s))
		-
		\psi_{\ff}( \ell; \eta(s))
		\big) v_{s,x}
		\ud s
		\bigg|^{2}\bigg]
		\lesssim_{\rho , \ell_{0}}
		\frac{1}{ \alpha^{3} L^{2}} \Big(\frac{1}{ \alpha}
		+
		\frac{\ell^{2}}{  \alpha \ell +1 }
		\Big)
		\int_{0}^{t} \| v_{s, \cdot}\|_{\ell^{2}}^{2} \ud s\,.
	\end{aligned}
\end{equation}

\textbf{Step 3.} \emph{Final static replacement.}
Lastly, by the Cauchy-Schwarz inequality (applied to the time integral)
\begin{equation}\label{eq_final_block}
	\begin{aligned}
		 & \mathbf E_{\rho,L}\bigg[\bigg|
		\int_{0}^{t}
		\sum_{x \in \ZZ}
		\tau_{x}
		\big(
		\psi_{\ff}( \ell; \eta(s))
		-
		\mQ_{\rho}(\ell , \eta (s))
		\big) v_{s,x}
		\ud s
		\bigg|^{2}\bigg]                  \\
		 & \leqslant
		t
		\int_{ 0}^{t}
		\Big\|
		\sum_{x \in \ZZ}
		\tau_{x}
		\big(
		\psi_{\ff}( \ell; \eta)
		-
		\mQ_{\rho}(\ell , \eta )		\big) v_{s,x}
		\Big\|_{L^{2}( \nu_{\rho})}^{2}
		\ud s
		 \lesssim_{\rho}
		\frac{t}{ \alpha^{ 3} \ell^2}
		\int_0^t \|v_{s,\cdot}\|_{\ell^2}^2\,\ud s\,,
	\end{aligned}
\end{equation}
where we used Lemma~\ref{lem_equivensemble}, the fact that
$\tau_{x}
	\big(
	\psi_{\ff}( \ell; \eta)
	-
	\mQ_{\rho}(\ell , \eta )		\big)$ is correlated to at most
$ 2 \ell +1 $ other such blocks, and that $ \psi_{\ff}( \ell; \eta)
	-
	\mQ_{\rho}(\ell , \eta )	$ is centred under $
	\nu_{\rho}$.
\par\smallskip\smallskip

Combining both \eqref{eq_one_block} and \eqref{eq_final_block} proves Proposition~\ref{prop_BG}.

\begin{remark}\label{rem_l2}
	To justify the applications of Corollary~\ref{cor_kp_postHminus1} above,
	set $v^R_{s,x}:=v_{s,x}\mathbf 1_{\{|x|\leq R\}}$.
	Then $v^R$ has finite spatial support and
	$v^R\to v$ in $L^2([0,t];\ell^2(\ZZ))$.
	Applying the preceding estimates to $v^R-v^{R'}$ shows that the
	truncated additive functionals are Cauchy in
	$L^2(\mathbf P_{\rho,L})$.
	Passing to the limit then yields the bounds for $v$.
\end{remark}

\section{Tightness}\label{sec_tight}

Dynkin's formula yields the following decomposition of the density fluctuation field
\begin{equation}\label{eq_dynkin_tight}
	\begin{aligned}
		u_{L} ( t, \varphi)
		 & =
		u_{L}( 0, \varphi)
		+ \int_{0}^{t} ( \partial_{s}+ L^{2} \gens +
		L^{3/2}\gena ) u_{L}(s, \varphi)
		\ud s
		+
		M_t^L(\varphi) \\
		 & =
		u_{L}( 0, \varphi)
		+ \mathcal{D}_t^L(\varphi)
		+ \mZ_t^L(\varphi)
		+
		M_t^L(\varphi)
		\,,
	\end{aligned}
\end{equation}
where $ M^L(\varphi)$ is a martingale and we defined
the symmetric integral contribution
\begin{equation}\label{eq_def_mD}
	\begin{aligned}
		\mathcal{D}_t^L(\varphi)
		:=
		\int_{0}^{t}
		L^{2} \gens u_{L} (r, \varphi)
		\ud r
		\,,
	\end{aligned}
\end{equation}
as well as the antisymmetric integral contribution
\begin{equation}\label{eq_def_mZ}
	\begin{aligned}
		\mZ_t^L(\varphi) := \int_{0}^{t} (\partial_{s} + L^{3/2} \gena) u_{L}( s, \varphi; \eta(s))  \ud
		s\,.
	\end{aligned}
\end{equation}
% By Corollary~\ref{cor_growth} and the rapid decay of Schwartz functions, the
% spatial truncations converge almost surely in
% $\mS'(\RR)$, uniformly on $[0,T]$. Since the truncated fields are càdlàg and the drift terms are
% time integrals, 
Notice that
\begin{equation*}
	\begin{aligned}
		u_L,M^L\in D([0,T],\mS'(\RR)),
		\qquad
		\mD^L,\mZ^L\in C([0,T],\mS'(\RR))\,,
	\end{aligned}
\end{equation*}
since their spatially truncated analogues have càdlàg paths that converge almost surely in $ \mS ' ( \RR)$ uniformly in $[0,T]$, due to Corollary~\ref{cor_growth} and the rapid decay of Schwartz functions.

The goal of this section is to prove that each of the components on the right--hand side
of \eqref{eq_dynkin_tight} is tight for every $ \varphi \in \mS ( \RR ) $, 
which yields tightness in $ D ([0,T] , \mS' ( \RR ))$ by Mitoma's criterion.
% , namely $ ( u_{L}( 0, \varphi))_{L \in \NN }$, $ (
% 	\mathcal{D}_{\cdot}^{L}(\varphi))_{L \in \NN}$, $ ( \mZ^{L}( \varphi) )_{L \in \NN}$, and $ (M^{L}(
% 	\varphi))_{L \in \NN}$.
Hence, for every subsequence we will find a further subsequence along which all of
the terms converge.
%In particular, this yields tightness of each component in $ D ([0,T] , S' ( \RR ))$, by
%Mitoma's criterion.
%\cite{Mitoma83}.
We also establish that limit points are supported (almost surely) on $ C (
	[0,T] , \mS' (\RR ) )$.

First, we introduce some useful
notation.
For every $ \varphi \in \mS ( \RR ) $, we define the shifted test functions
\begin{equation}\label{eq_phi_timeshift}
	\begin{aligned}
		\varphi_{t}(\cdot)=	\varphi_{t}^{(L)} (\cdot)
		:=
		\varphi \big( \cdot -
		L^{1/2}
		j' ( \rho) t   \big)
	\end{aligned}
\end{equation}
and $ \varphi_{t,x}
	:= \varphi_{t} ( L^{-1}x) =
	\varphi \big( L^{-1}x -
	L^{1/2}
	j' ( \rho) t   \big)$.
With this notation
\begin{equation*}
	\begin{aligned}
		u_{L}(t, \varphi) =
		\frac{1}{ \sqrt{L}}  \sum_{x \in \ZZ}
		\varphi_{t, x} \big(
		\eta_{x} - \rho\big)\,.
	\end{aligned}
\end{equation*}
Furthermore, we write
\begin{equation}\label{eq_grad}
	\begin{aligned}
		\nabla_x^L \varphi := L \big(
		\varphi (
		L^{-1} (x+1)) - \varphi ( L^{-1}x )\big)\,,
	\end{aligned}
\end{equation}
for the discrete gradient, and
\begin{equation*}
	\begin{aligned}
		\Delta_x^L  \varphi :=
		L^{2}
		\big(
		\varphi ( L^{-1}(x+1) ) + \varphi ( L^{-1} ( x-1)) - 2
		\varphi(
		L^{-1}x)
		\big)\,,
	\end{aligned}
\end{equation*}
for the discrete Laplace operator.

The subsequent estimates are tedious but mostly standard, see for example
\cite{GoncJara14,DGP17}.
Because we are also interested in the regime $ \alpha \to 0 $, we take particular
interest in how the derived bounds depend on $ \alpha$.

\subsection*{Tightness of the field at fixed times}
The law of the input $ \eta(t)$ into $ u_{L}( t, \varphi) $ does not depend on $t$, due
to the invariance of $ \nu_{\rho}$.
By definition of $ u_{L}(t, \varphi)$, we have for every
$ t \in [0,T]$
\begin{equation}\label{eq_cond_tight_fixedfield}
	\begin{aligned}
		\sup_{L\in\NN}
		\mathbf E_{\rho,L}\big[
		|u_{L}(t,
		\varphi)|^{2}\big]
		=
		\chi( \rho)\,
		\sup_{L\in\NN}
		L^{-1}
		\sum_{x \in \ZZ} \varphi \big( L^{-1}x -
		L^{1/2}
		j' ( \rho) t   \big)^{2}
		\lesssim_{\varphi}\chi(\rho) \,,
	\end{aligned}
\end{equation}
where we used that the
Riemann sum converges to the integral $ \int
	\varphi(x)^{2} \ud x$, with the transport shift having no effect.
More explicitly, since $\chi(\rho)\lesssim_\rho\alpha^{-1}$, we have
\begin{equation*}
	\begin{aligned}
		\sup_{L \in \NN}
		\mathbf E_{\rho,L}[|u_L(t,\varphi)|^2]
		\lesssim_{\rho,\varphi}
		\alpha^{-1}\,.
	\end{aligned}
\end{equation*}

\subsection*{Tightness of the martingale term}

Tightness of the sequence of martingales $ ( M^L(\varphi))
	_{L\in\NN}$ follows from $C$--tightness of the (predictable) quadratic
variation processes $(
	\langle M^L(\varphi)\rangle )_{L\in\NN}$, see \cite[Theorem VI.4.13]{JS}.
%see also Whitt FCLT martingale survey Theorem 3.6
%Indeed, for
%the latter it suffices to establish the conditions of the Kolmogorov--Chentsov
%theorem, since by \eqref{eq_qv} the trajectories of $ (\langle M^L(\varphi)\rangle_{t})_{t \geqslant 0}$ are almost surely continuous.
In fact, we will prove a much stronger statement, namely a quantitative uniform control
of fluctuations of the predictable quadratic variation process:

\begin{lemma}\label{lem_quant_qv}
	For every $L \in \NN $, $ \varphi \in \mS ( \RR)$ and $ T>0$
	\begin{equation*}
		\begin{aligned}
			\mathbf E_{\rho,L}
			\Big[
			\sup_{t\leq T}
			\Big|
			\langle  M^L(\varphi)\rangle_t
			-
			t\, \alpha \chi(\rho) \| \partial_{x} \varphi\|_{L^{2} ( \RR )}^{2}
			\Big|^2
			\Big]
			\lesssim_{\rho, \varphi}
			\frac{T}{ \alpha^{4} L^{3}} + \frac{T^{ 2}}{\alpha L}
			+ T^{2} o_{L}(1)	\,,
		\end{aligned}
	\end{equation*}
	where $ o_{L}(1) $ vanishes as $ L \to \infty $ uniformly for small $ \alpha$.
\end{lemma}

The reader may wonder about the need for such an overly precise quantitative statement.
We will see in
Section~\ref{sec_cond} that the quantitative control is precisely the correct one to
control the martingale contribution in the weakly condensing regime.

\begin{proof}
	Using \eqref{eq_qv}, the quadratic variation is given in terms of
	\begin{equation}\label{eq_prelimit_bracket}
		\begin{aligned}
			\langle  M^L(\varphi)\rangle_t
			 & =
			L^{2}
			\int_0^t
			\sum_{x,y \in \ZZ}
			c(x,y) \eta_{x}(r)( \alpha + \eta_{y}(r))
			\big(
			u_{L} (r, \varphi; \eta^{x,y}(r))
			-
			u_{L}( r, \varphi; \eta(r))
			\big)^{2}
			\ud r \\
			 & =
			\int_0^t
			\frac{1}{L}
			\sum_{x\in\ZZ}
			a_x^{L}(\eta (r))
			(\nabla_x^L\varphi_{r})^2\,\ud r,
		\end{aligned}
	\end{equation}
	where we defined
	\begin{equation*}
		a_x^L(\eta):=
		\eta_x\eta_{x+1}+
		\alpha(p_L\eta_x+q_L\eta_{x+1})\,,
	\end{equation*}
	which has mean $ \EE_{\rho}[ a_{x}^{L}( \eta)]=
		\alpha \chi( \rho) $ and  satisfies the small-$\alpha$ $L^{2}(
		\nu_{\rho}) $-bound
	$					\EE_\rho[|a_x^L|^2]
		\lesssim_\rho \alpha^{-2}$.
	Moreover, in the second identity we used
	$ u_{L} (r, \varphi; \eta^{x,x+1}) -
		u_{L}( r, \varphi; \eta)
		=
		L^{-1/2}
		( \varphi_{r,x+1} - \varphi_{ r,x}) $ and performed a change of
	variables in the sum.
	Consequently,  by stationarity
	\begin{equation}\label{eq_mean_qv_conv}
		\begin{aligned}
			\sup_{t \in [0,T]}
			\left|
			\mathbf E_{\rho,L}[\langle M^L(\varphi)\rangle_t]
			-
			t\, \alpha \chi ( \rho) \|\partial_x\varphi\|_{L^2(\RR)}^2
			\right|
			 & \leqslant T \, \alpha \chi( \rho)
			\sup_{ r \in [0, T]} \Big|
			\frac{1}{L}
			\sum_{x\in\ZZ}
			(\nabla_x^L\varphi_{r})^2-
			\|\partial_x\varphi\|_{L^2(\RR)}^2
			\Big|                                          \\
			 & \lesssim_{\rho, \varphi}  T \,  o_{L}(1)\,,
		\end{aligned}
	\end{equation}
	which vanishes as $L \to \infty$ by Lemma~\ref{lem_usefulestimates_phi}\,(iii) and a Riemann
	sum convergence argument. This establishes $ \lim_{L \to \infty}  \mathbf E_{\rho,L}[\langle M^L(\varphi)\rangle_t]
		= t\, \alpha \chi(\rho) \| \partial_{x} \varphi\|_{L^{2} ( \RR )}^{2} $ uniformly on
	compact time intervals.
	\par\smallskip\smallskip

	Next, we show that fluctuations of the quadratic variation vanish uniformly in time.
	To this end, we use (recall $ \tau_{x}\ff (\eta)= (\eta_{x}- \rho) (\eta_{x+1}- \rho) $ and
	$ \psi_{\ff} ( \ell, \eta) = \EE_{\ell, N_{\ell}}[ \ff ( \eta) ]$)
	\begin{equation*}
		\begin{aligned}
			a_x^{L}(\eta) - \alpha \chi( \rho)
			 & =
			\tau_{x}\ff ( \eta) \pm \tau_x\psi_\ff(2,\eta)
			+ \rho( \eta_{x}+ \eta_{x+1}- 2 \rho)
			+
			\alpha(p_L\eta_x+q_L\eta_{x+1} - \rho)                        \\
			 & =: \tau_{x}\ff ( \eta) - \tau_x\psi_\ff(2,\eta) + \tau_{x}
			b_{\alpha, L} (\eta)
		\end{aligned}
	\end{equation*}
	to split
	\begin{equation}\label{eq_qv_splitting}
		\begin{aligned}
			\langle  M^L(\varphi)\rangle_t
			-
			\mathbf E_{\rho,L}[ \langle  M^L(\varphi)\rangle_t]
			 & =
			\int_0^t
			\frac{1}{L}
			\sum_{x\in\ZZ}
			\big(	\tau_{x}\ff ( \eta(r))- \tau_x\psi_\ff(2,\eta(r))\big)
			(\nabla_x^L\varphi_{r})^2\,\ud r \\
			 & + \int_0^t
			\frac{1}{L}
			\sum_{x\in\ZZ}
			\tau_{x}b_{\alpha, L} (\eta(r))        (\nabla_x^L\varphi_{r})^2\,\ud r\,.
		\end{aligned}
	\end{equation}

	First, we notice that $ \EE_{\rho}[ \tau_{x}\ff ( \eta)- \tau_x\psi_\ff(2,\eta)|
			\mF_{x,2} ] =0$ by definition. Hence, the one block estimate  \eqref{eq_oneblock} yields
	\begin{equation}\label{eq_tight_oneblock}
		\begin{aligned}
			 & \mathbf E_{\rho,L}\bigg[
			\sup_{ t \in [0,T]}
			\bigg|
			\int_{0}^{t}
			\frac{1}{L}
			\sum_{x\in\ZZ}
			\big(	\tau_{x}\ff ( \eta(r))- \tau_x\psi_\ff(2,\eta(r))\big)
			(\nabla_x^L\varphi_{r})^2\ud r
			\bigg|^{2}
			\bigg]
			\lesssim_{\rho, \varphi}
			\frac{T}{ \alpha^{4}L^{3}} \,,
		\end{aligned}
	\end{equation}
	where one power of $ L$ is lost to control $ \| \nabla_{x}^{L} \varphi
		\|_{\ell^{2}}^{2}$.

	On the other hand, using first the Cauchy--Schwarz inequality and stationarity, as well as finite range
	correlation of the centred variables $ \tau_{ x} b$, we have the bound
	%\simon{main point is that now expression is linear in $ \eta$}
	\begin{equation}\label{eq_b_int_est}
		\begin{aligned}
			\mathbf E_{\rho,L}\bigg[
			\sup_{ t \in [0,T]}
			\bigg|
			\int_{0}^{t}
			\frac{1}{L}
			\sum_{x\in\ZZ}
			\tau_{x}b_{\alpha, L} (\eta(r))
			(\nabla_x^L\varphi_{r})^2\ud r
			\bigg|^{2}
			\bigg]
			 & \lesssim \frac{T^{2}}{L^{2}}
			\EE_{\rho}[ b_{\alpha, L}( \eta)^{2}]
			\sup_{r \in [0,T]}
			\sum_{x \in \ZZ}
			|\nabla_{x}^{L}
			\varphi_{r}|^{4}
			\,.
		\end{aligned}
	\end{equation}
	By the explicit formula \eqref{eq_supp_psi}, we can write with $ N_{0,2}=
		\eta_{0}+ \eta_{1}$
	\begin{equation}\label{eq_phi2_supp}
		\begin{aligned}
			|\psi_{\ff}( 2; \eta)|
			 & \leqslant 
			\frac{ \alpha N_{0,2} ( N_{0 ,2} -1)}{2 (2 \alpha +1 )}
			+ \rho N_{0 , 2} + \rho^{2}
			\lesssim_{\rho} \alpha N_{0,2}^{2}+ N_{0,2} + 1\,.
		\end{aligned}
	\end{equation}
	Moreover, $\big|\rho( \eta_{0}+ \eta_{1}- 2 \rho)
		+
		\alpha(p_L\eta_0+q_L\eta_{1} - \rho)\big|
		\lesssim_{\rho} N_{0,2} + \alpha N_{0,2} + 1$.
		Hence, together with \eqref{eq_phi2_supp}
	\begin{equation*}
		\begin{aligned}
			\EE_{\rho}[ b_{\alpha, L}( \eta)^{2}]
			\lesssim_{\rho}
			\EE_{\rho}[ | \alpha N_{0,2}^{2} + N_{0,2} + 1|^{2}]
			\lesssim_{\rho} \alpha^{-1}\,.
		\end{aligned}
	\end{equation*}
	Combining the decomposition in \eqref{eq_qv_splitting} with \eqref{eq_tight_oneblock} and \eqref{eq_b_int_est} finally yields
	\begin{equation*}
		\begin{aligned}
			\mathbf E_{\rho,L}
			\Big[
				\sup_{t\leq T}
				\Big|
				\langle  M^L(\varphi)\rangle_t
				-
				\mathbf E_{\rho,L}[ \langle  M^L(\varphi)\rangle_t]
				\Big|^2
				\Big]
			\lesssim_{\rho, \varphi}
			\frac{T}{ \alpha^{4} L^{3}} + \frac{T^{ 2}}{\alpha L} \,,
		\end{aligned}
	\end{equation*}
	which together with \eqref{eq_mean_qv_conv} finishes the proof.
\end{proof}

To conclude that any limit point of $ (
	M^L(\varphi))_{L\in\NN}$ in $ D ([0,T] , \RR)$  has continuous trajectories almost surely, it
	remains only to show that the sizes of the jumps vanish as $ L\to\infty $.
Indeed, because discontinuities of $ M^L(\varphi)$ must come from
discontinuities of $u_{L}( t, \varphi) $ (since the integrated drift terms are continuous in time, see \eqref{eq_dynkin_tight}), it
suffices to control
\begin{equation*}
	\begin{aligned}
		u_{L}( t, \varphi; \eta (t) )
		- u_{L}( t, \varphi; \eta(t-))\,.
	\end{aligned}
\end{equation*}
Assume that at time $ t$ a single particle jumps from $ x $ to $ x+1$, then
\begin{equation*}
	\begin{aligned}
		u_{L}( t, \varphi; \eta (t) )
		- u_{L}( t, \varphi; \eta(t-))
		=
		L^{-1/2} ( \varphi_{t,x+1}- \varphi_{t,x})
		= L^{-3/2} \nabla_x^L \varphi_{t}\,.
	\end{aligned}
\end{equation*}
The same bound holds for a particle jumping from $ x $ to $ x-1$.
Hence, for every $ t \geqslant 0 $
\begin{equation}\label{eq_vanish_jumps}
	\begin{aligned}
		| M_t^L(\varphi) - M_{t-}^L(\varphi)|
		\leqslant
		L^{-3/2} |\nabla_x^L \varphi_{t}|
		\leqslant L^{-3/2} \| \varphi' \|_{\infty}\,.
	\end{aligned}
\end{equation}
% as a consequence of
% the mean-value theorem for integration: for some $ y \in \RR $
% \begin{equation*}
% 	\begin{aligned}
% 		\nabla_x^L \varphi_{t}
% 		%=
% 		%L
% 		%( \varphi_{x+1, t}- \varphi_{x,t})
% 		=
% 		L
% 		\big( \varphi ( L^{-1} (x+1) - L^{1/2} j'(\rho) t)
% 		- \varphi( L^{-1}x - L^{1/2} j'(\rho) t)\big)
% 		= \varphi' ( y ) \,.
% 	\end{aligned}
% \end{equation*}
Therefore, any subsequential limit of $ M^L(\varphi)$ is
supported on continuous trajectories.

%SOME OLD VERSION THAT WAS NOT SUFFICIENT, TRIED VIA KOLMOGOROV
%	\gray{
%		By the martingale property and the definition of predictable quadratic variation, we have
%		\begin{equation*}
%			\begin{aligned}
%				\mathbf E_{\rho,L}\big[
%					\big| M_t^L(\varphi) -
%				M_s^L(\varphi) \big|^{2}\big]
%			=
%			\mathbf E_{\rho,L}\big[
%					\big| M_t^L(\varphi)  \big|^{2}\big]
%				-
%				\mathbf E_{\rho,L}\big[
%					\big| M_s^L(\varphi)  \big|^{2}\big]
%				=
%				\mathbf E_{\rho,L}[ \langle M^L(\varphi)\rangle_{t}]
%				-
%				\mathbf E_{\rho,L}[ \langle M^L(\varphi)\rangle_{s}]
%			\end{aligned}
%		\end{equation*}
%		The expectation of the quadratic variation of the martingale takes the
%		form \eqref{eq_expect_qv}, thus, using \eqref{eq_supp_sym}
%		\begin{equation*}
%			\begin{aligned}
%				\mathbf E_{\rho,L}\big[
%					\big| M_t^L(\varphi) -
%				M_s^L(\varphi) \big|^{2}\big]
%				&= 2 L^{2}
%				\int_{s}^{t}  \langle u_{L}(r, \varphi) , -
%				\gens u_{L}(r, \varphi) \rangle_{\rho} \ud r\\
%				&=
%					\alpha \chi( \rho)  \int_{s}^{t} L^{-1}\sum_{x \in \ZZ}
%				( \nabla_x^L \varphi_{r})^{2}\ud r\,.
%		\end{aligned}
%		\end{equation*}
%		The integrand is uniformly bounded in $ L$ and $ r$, thus,
%		\begin{equation*}
%			\begin{aligned}
%				\mathbf E_{\rho,L}\big[
%					\big| M_t^L(\varphi) -
%			M_s^L(\varphi) \big|^{2}\big]
%				\leqslant {\color{orange}\alpha\chi(\rho)
%				\left(
%				\sup_{L\in\NN,r}L^{-1}\sum_x
%				(\nabla_x^L\varphi_r)^2
%				\right)} (t-s)\,.
%			\end{aligned}
%		\end{equation*}
%	}

\subsection*{Tightness of the symmetric drift}

%(Here we can use the fact that integral makes the process continuous).

Direct manipulation of the symmetric term of the
generator (cf. \eqref{eq_sym}) yields
(by explicit cancellations)
\begin{equation}\label{eq_supp_sym}
	\begin{aligned}
		L^{2} \gens u_{L} (t, \varphi)
		 & =
		L^{3/2}
		\frac{1}{2}
		\sum_{x \in \ZZ}
		\Big( \eta_{x} ( \alpha + \eta_{x+1}) \big(
			\varphi_{t,x+1} - \varphi_{t, x}\big)
		+
		\eta_{x} ( \alpha+ \eta_{x-1}) \big( \varphi_{t, x-1}-
			\varphi_{t,x}\big)
		\Big) \\
		 & =
		\frac{ \alpha }{ 2}
		L^{-1/2}
		\sum_{x \in \ZZ} \eta_{x} \Delta_x^L
		\varphi_{t}
		=
		\frac{\alpha}{2}
		u_{L} (t, \Delta^{L} \varphi)\,,
	\end{aligned}
\end{equation}
where in the last step we included the zero term (notice that
$ \sum_{x \in \ZZ} |\varphi_{t,x}| < \infty$ because $ \varphi \in \mS ( \RR)$)
\begin{equation*}
	\begin{aligned}
		L^{-1/2}
		\sum_{x \in \ZZ} \Delta_x^L
		\varphi_{t}
		= L^{3/2}
		\bigg(
		\sum_{x \in \ZZ} \varphi_{t, x+1}
		+
		\sum_{x \in \ZZ} \varphi_{t, x-1}
		-
		2
		\sum_{x \in \ZZ} \varphi_{t, x}
		\bigg)=0\,.
	\end{aligned}
\end{equation*}
Hence,  the symmetric drift has the form
\begin{equation*}
	\begin{aligned}
		\mathcal{D}_t^L(\varphi)
		=
		\frac{\alpha}{2}  L^{-1/2}
		\int_{0}^{t}
		\sum_{x \in \ZZ} \eta_{x}(r) \Delta_x^L
		\varphi_{r} \ud r \,,
	\end{aligned}
\end{equation*}
and therefore
\begin{equation*}
	\begin{aligned}
		\mathbf E_{\rho,L}
		\big[
		\big| \mathcal{D}_t^L(\varphi) - \mathcal{D}_s^L(\varphi) \big|^{2}\big ]
		=
		\frac{\alpha^{2}}{4}
		\mathbf E_{\rho,L}
		\bigg[ \bigg|
		\int_{s}^{t} L^{-1/2} \sum_{x \in \ZZ} \eta_{x}(r) \Delta_x^L \varphi_{r} \ud r
		\bigg|^{2}
		\bigg] \,.
	\end{aligned}
\end{equation*}
Next, we apply the Cauchy--Schwarz inequality (to the time integral)
\begin{equation*}
	\begin{aligned}
		\mathbf E_{\rho,L}
		\big[
		\big| \mathcal{D}_t^L(\varphi) - \mathcal{D}_s^L(\varphi) \big|^{2}\big ]
		 & \leqslant	\frac{\alpha^{2}}{4}  (t-s)
		\int_{s}^{t}
		\EE_{\rho }\Big[ \Big|
		L^{-1/2}
		\sum_{x \in \ZZ} (\eta_{x}- \rho) \Delta_x^L \varphi_{r}
		\Big|^{2}\Big]
		\ud r                                    \\
		 & =
		\frac{\alpha^{2}}{4}   \chi ( \rho)  (t-s )
		\int_{s}^{t}
		L^{-1}
		\sum_{x \in \ZZ}
		\big(\Delta_x^L \varphi_{r}\big)^{2}
		\ud r
		\,,
	\end{aligned}
\end{equation*}
with the last identity being a consequence of $ \nu_{\rho}$ being a
product measure.
Moreover, the integrand on the right-hand side is
uniformly bounded (in $ r$) by Lemma~\ref{lem_usefulestimates_phi}\,(ii).
Hence, using $\chi(\rho)\lesssim_\rho\alpha^{-1}$, we have
\begin{equation}\label{eq_tight_sym}
	\begin{aligned}
		\mathbf E_{\rho,L}
		\big[
		\big| \mathcal{D}_t^L(\varphi) - \mathcal{D}_s^L(\varphi) \big|^{2}\big ]
		\lesssim_{\rho,\varphi}
		\alpha
		(t-s)^{2}\,.
	\end{aligned}
\end{equation}
Thus, the Kolmogorov--Chentsov
theorem yields tightness of the symmetric drift $
	(\mathcal{D}_{\cdot}^{L}(\varphi))_{L \in \NN}$ in $ C ([0,T], \RR)$.

\subsection*{Tightness of the antisymmetric drift}

%(Here we can use the fact that integral makes the process continuous).
Finally, we establish tightness of the sequence $ ( \mZ^{L} ( \varphi))_{L \in \NN}$.
To this end, we start by showing that the integrand of the antisymmetric drift term is
well approximated using $ W_{x}( \eta)$. This is the key step and explains why the
quadratic current term appears.

\begin{lemma}\label{lem_antisym_rewrite}
	We have
	\begin{equation*}
		\begin{aligned}
			%\lim_{L\to\infty}
			\sup_{ t \in [0,T] }
			\EE_{\rho}\bigg[
			\Big|
			(\partial_{t} + L^{3/2} \gena) u_{L}( t, \varphi; \eta)
			-
			\beta \sum_{x \in \ZZ} W_{x}( \eta)
			\nabla_x^L \varphi_{t}\Big|^{2}
			\bigg]
			\lesssim_{\rho , \varphi} \beta^{2}
			\alpha^{-1} L^{-1}
			\,,
		\end{aligned}
	\end{equation*}
	where $ W_{x}$ was defined in \eqref{eq_def_Wandff}.
\end{lemma}
\begin{proof}
	First, we notice that
	\begin{equation*}
		\begin{aligned}
			(\partial_{t} + L^{3/2} \gena) u_{L}( t, \varphi; \eta)
			 & = -L^{1/2} j' ( \rho)  u_{L}(t,
			\varphi' ; \eta )
			+ L
			\sum_{x \in \ZZ} \varphi_{t,x}
			\big(j_{x-1 , x}( \eta) - j_{x , x+ 1}(
			\eta)\big)                         \\
			 & =
			- j' ( \rho)
			\sum_{x\in\ZZ}\varphi_t'(L^{-1}x)(\eta_x-\rho)
			+ \sum_{x \in \ZZ}
			j_{x, x+1}( \eta)
			\nabla_x^L\varphi_{t} \,,
		\end{aligned}
	\end{equation*}
	where we used the definition of the density fluctuation field and rearranged the sums in the last step.
	Now, the key step is to expand $j_{x, x+1}( \eta)$ around the mean
	current $ j ( \rho) $, using the identity \eqref{eq_current_expansion},
	which we recall here for convenience
	\begin{equation}\label{eq_restate_eqflux}
		\begin{aligned}
			j_{x, x+1} ( \eta)
			= j ( \rho ) + \frac{1}{ 2} j' (\rho) \big(( \eta_{x}-
			\rho) + ( \eta_{x+1}- \rho) \big)
			+ \frac{1}{2} j''( \rho) W_{x}( \eta ) \,.
		\end{aligned}
	\end{equation}
	Notice that $ j'' (\rho) = 2 \beta$.
	Moreover, the contribution of the constant $j(\rho)$--term is exactly zero, since
	$\sum_{x\in\ZZ}\nabla_x^L\varphi_t=0$ by telescoping and the
	Schwartz decay of $\varphi_t$.
	Hence, to conclude the statement of the lemma, it suffices to show that
	the transport term in \eqref{eq_def_u} compensates the linear term in \eqref{eq_restate_eqflux}.
	Namely, we prove
	\begin{equation}\label{eq_supp_anti}
				\begin{aligned}
			\sup_{t\in[0,T]}
			\EE_\rho\Big[
				\Big|
				\sum_{x\in\ZZ}\varphi_t'(L^{-1}x)(\eta_x-\rho)
				-
				\frac12\sum_{x\in\ZZ}
				\nabla_x^L\varphi_t
				\big((\eta_x-\rho)+(\eta_{x+1}-\rho)\big)
				\Big|^2
				\Big]
			\lesssim_{\rho,\varphi}\alpha^{-1}L^{-1}\,.
		\end{aligned}
	\end{equation}
	To this end, we reindex the sum associated to the linear part
	of the current, which yields
	\begin{align*}
		\frac12\sum_{x\in\ZZ}
		\nabla_x^L\varphi_t
		\big((\eta_x-\rho)+(\eta_{x+1}-\rho)\big)
		%&=
		%\frac12\sum_{x\in\ZZ}
		%\nabla_x^L\varphi_t\,(\eta_x-\rho)
		%+
		%\frac12\sum_{x\in\ZZ}
		%\nabla_{x-1}^L\varphi_t\,(\eta_x-\rho)
		%\\
		 & =
		\sum_{x\in\ZZ}
		\frac12\big(
		\nabla_x^L\varphi_t
		+
		\nabla_{x-1}^L\varphi_t
		\big)(\eta_x-\rho) \\
		 & =
		\sum_{x\in\ZZ}
		(\eta_x-\rho)
		\frac{1}{2L^{-1}}
		\int_{-L^{-1}}^{L^{-1}}
		\varphi_t'(L^{-1}x+z)\ud z\,
		\,,
	\end{align*}
	by the fundamental theorem
	of calculus.		%We use
	%\begin{equation*}
	%		\frac12\big(
	%			\nabla_x^L\varphi_t
	%			+
	%			\nabla_{x-1}^L\varphi_t
	%			\big)
	%			=
	%			\frac{
	%			\varphi_t(L^{-1}(x+1))
	%			-
	%			\varphi_t(L^{-1}(x-1))
	%			}{2L^{-1}}
	%			=
	%			\frac1{2L^{-1}}
	%			\int_{-L^{-1}}^{L^{-1}}
	%			\varphi_t'(L^{-1}x+z)\ud z\,.
	%		\end{equation*}
	%	Therefore
	%		\begin{equation*}
	%			\begin{aligned}
	%			&\sum_{x\in\ZZ}\varphi_t'(L^{-1}x)(\eta_x-\rho)
	%			-
	%			\frac12\sum_{x\in\ZZ}
	%			\nabla_x^L\varphi_t
	%			\big((\eta_x-\rho)+(\eta_{x+1}-\rho)\big)\\
	%			&=
	%			\sum_{x\in\ZZ}
	%			\Big(
	%			\varphi_t'(L^{-1}x)
	%			-
	%			\frac{1}{2L^{-1}}
	%			\int_{-L^{-1}}^{L^{-1}}
	%			\varphi_t'(L^{-1}x+z)\ud z \Big)
	%			(\eta_x-\rho)\,.
	%			\end{aligned}
	%		\end{equation*}
Moreover, 
	\begin{align*}
		\Big|
		\varphi_t'(L^{-1}x)
		-
		\frac{1}{2L^{-1}}
		\int_{-L^{-1}}^{L^{-1}}
		\varphi_t'(L^{-1}x+z)\ud z
		\Big|
		 & \leq
		\frac1{2L^{-1}}
		\int_{-L^{-1}}^{L^{-1}}
		\left|
		\varphi_t'(L^{-1}x)
		-
		\varphi_t'(L^{-1}x+z)
		\right|\ud z \\
		 & \leq
		L^{-1}
		\sup_{|y-L^{-1}x|\leq L^{-1}}
		|\varphi_t''(y)|,
	\end{align*}
	which implies that (since $\varphi_t$ is only a spatial translate of the Schwartz function
	$\varphi$)
	uniformly in $t\in[0,T]$
	\begin{equation*}
		\begin{aligned}
			 & \EE_{\rho}\bigg[
			\bigg|
			\sum_{x\in\ZZ}
			\Big(
			\varphi_t'(L^{-1}x)
			-
			\frac{1}{2L^{-1}}
			\int_{-L^{-1}}^{L^{-1}}
			\varphi_t'(L^{-1}x+z)\ud z \Big)
			(\eta_x-\rho)
			\bigg|^{2}\bigg]           \\
			 & \leqslant  \chi ( \rho)
			\sum_{x \in \ZZ}
			L^{-2}
			\sup_{|y-L^{-1}x|\leq L^{-1}}
			|\varphi_t''(y)|^2
			\lesssim_{\varphi} \chi( \rho) L^{-1}
			\lesssim_{\rho,\varphi}\alpha^{-1}L^{-1}\,,
		\end{aligned}
	\end{equation*}
	where the second-to-last inequality follows from a Riemann-sum estimate.
	This yields \eqref{eq_supp_anti}, which together with $j'(\rho)^{2}=\beta^{2}(\alpha+2\rho)^{2} \lesssim \beta^{2}$, concludes the proof.
\end{proof}

Lemma~\ref{lem_antisym_rewrite} allows us to reduce tightness of $ (\mZ^{L} (\varphi))_{L}$ to tightness of a drift with respect to the quadratic current $ W
$.
However, as explained in Section~\ref{sec_intro_quadratic}, the quadratic current $ W$ is singular,
which is why we further introduce an intermediate block approximation of the
drift. Namely, for $ \ell \geqslant \ell_{0}$  we define 
\begin{equation*}
	\begin{aligned}
		\mathcal{B}_{t}^{L,\ell}(\varphi)
		:= \beta \int_{0}^{t} \sum_{x \in \ZZ}
		\nabla_x^L \varphi_{s} \tau_{x}\mQ_{\rho}( \ell ,
		\eta(s))\ud s\,,
	\end{aligned}
\end{equation*}
which can be controlled using its static moments.
This allows us to split
$\mZ_t^L(\varphi)$ into two contributions,
one for the block approximation and one for the remainder term:
\begin{equation*}
	\begin{aligned}
		\mZ_t^L(\varphi)
		=
		\mathcal{B}_{t}^{L,\ell}(\varphi)
		+
		\mathcal{R}_{t}^{L,\ell}(\varphi)
		+o(1)\,,
	\end{aligned}
\end{equation*}
with the block remainder having the exact form
\begin{equation*}
	\mathcal{R}_{t}^{L,\ell}(\varphi)
	:=
	\beta\int_0^t
	\sum_{x\in\ZZ}
	\nabla_x^L\varphi_r\,
	\tau_x\big(W_0(\eta (r))-\mQ_\rho(\ell,\eta(r))\big)\ud r \,,
\end{equation*}
which can be controlled using the second--order Boltzmann--Gibbs
principle.
The vanishing error term $ o(1)$ is controlled in $ L^{2}( \mathbf{P}_{\rho, L})$ uniformly in time by
Lemma~\ref{lem_antisym_rewrite}, see also the proof of Corollary~\ref{cor_asym_increment_estimate} below.
We proceed by proving increment estimates for the block approximation
$ \mathcal{B}^{L,\ell}(\varphi)$ and the block remainder
$\mathcal{R}^{L,\ell}(\varphi)$ separately.
\par\smallskip\smallskip

\textbf{Step 1.}  \textit{Increment estimate for block approximation.}
Once more we proceed by establishing Kolmogorov's tightness criterion.
Indeed, by stationarity with respect to $ \nu_{\rho}$ and the triangle inequality with
respect to $ L^{2}( \nu_{\rho})$, we have
\begin{equation*}
	\begin{aligned}
		\mathbf E_{\rho,L}
		\big[
		\big|
		\mathcal{B}_{t}^{L,\ell}(\varphi)
		-\mathcal{B}_{s}^{L,\ell}(\varphi)\big|^{2}\big]^{1/2}
		\leqslant \beta
		\int_{s}^{t}
		\Big\|
		%\EE_{\rho}\Big[\Big|
		\sum_{x \in \ZZ} \nabla_x^L  \varphi_{r}
		\tau_{x}\mQ_{\rho}( \ell, \eta)
		%\Big|^{2}\Big]
		\Big\|_{L^{2}( \nu_{\rho})}
		\ud r\,.
	\end{aligned}
\end{equation*}
Notice that $ ( \tau_{x}\mQ_{\rho}( \ell, \eta))_{x \in \ZZ}$ only have a
finite range of dependency, given in terms of the block length $ \ell $.
Hence, for every $x $ there are at most $ 2\ell+1 $ block approximations
that $ \tau_{x}\mQ_{\rho}( \ell , \eta )$ depends on.
Moreover, we recall from \eqref{eq_secondmoment_Q} that
$ \EE_{\rho}[ \mQ_{\rho}( \ell , \eta )^{2}]
	\lesssim_\rho \alpha^{-2}\ell^{-2}+\alpha^{-3}\ell^{-3}$.
Therefore, using the simple inequality $ ab \leqslant 2( a^{2}+
	b^{2}) $, we have
\begin{equation}\label{eq_blockcrossvariation}
	\begin{aligned}
		\EE_{\rho}\Big[\Big|
		\sum_{x \in \ZZ} \nabla_x^L  \varphi_{r}
		\tau_{x}\mQ_{\rho}( \ell, \eta)\Big|^{2}\Big]
		 & \lesssim_\rho (\alpha^{-2}\ell^{-2} + \alpha^{-3}\ell^{-3})
		\sum_{x \in \ZZ}
		\sum_{ \substack{y \in \ZZ                                     \\ | y-x| \leqslant \ell}}
		\big| \nabla_x^L  \varphi_{r}
		\nabla_y^L  \varphi_{r} \big|                                  \\
		 & \lesssim_\rho(\alpha^{-2}\ell^{-2} + \alpha^{-3}\ell^{-3})
		\bigg(
		\sum_{x \in \ZZ}
		\sum_{ \substack{y \in \ZZ                                     \\ | y-x| \leqslant \ell}}
		\big| \nabla_x^L
		\varphi_{r}\big|^{2}
		+
		\sum_{y \in \ZZ}
		\sum_{ \substack{x \in \ZZ                                     \\ | y-x| \leqslant \ell}}
		\big| \nabla_y^L
		\varphi_{r}\big|^{2}
		\bigg)\,,
	\end{aligned}
\end{equation}
which implies, in combination with Lemma~\ref{lem_usefulestimates_phi}\,(i),
that\footnote{Notice that a more brutal application of the triangle
inequality would have yielded a factor $ \ell^{-2}L^{2}$ on the right--hand
side which is not sufficient in view of the choice
$ \ell = \lfloor L \sqrt{t-s}\rfloor$
below.}
\begin{equation*}
	\begin{aligned}
		\sup_{r \in [0,T]}
		\EE_{\rho}\Big[\Big|
		\sum_{x \in \ZZ} \nabla_x^L  \varphi_{r}
		\tau_{x}\mQ_{\rho}( \ell, \eta)\Big|^{2}\Big]
		\lesssim_{\rho,\varphi}
		(\alpha^{-2}\ell^{-1} + \alpha^{-3}\ell^{-2}) L \,.
	\end{aligned}
\end{equation*}
Overall, we conclude that for all $s,t \in [0,T]$
\begin{equation}\label{eq_tight_block}
	\begin{aligned}
		\mathbf E_{\rho,L}
		\big[
		\big|
		\mathcal{B}_{t}^{L,\ell}(\varphi)
		-\mathcal{B}_{s}^{L,\ell}(\varphi)\big|^{2}\big]
		\lesssim_{\rho,\varphi}
		\beta^2
		(\alpha^{-2}\ell^{-1} + \alpha^{-3}\ell^{-2}) L (t-s)^{2}\,,
	\end{aligned}
\end{equation}
which is sufficient for Kolmogorov's tightness criterion
when $ \ell = \lfloor L \sqrt{t-s}\rfloor \vee \ell_{0}$.
%provided
%$ \ell  = \Theta (L)$ and $ \ell^{-1}L \lesssim (t-s)^{-(1- \kappa)}  $
%for some $ \kappa> 0 $.
\par\smallskip\smallskip

\textbf{Step 2.}  \textit{Increment estimate for block remainder.}
On the other hand, to prove analogous estimates for the remainder process
$\mathcal{R}^{L,\ell}(\varphi)$, we will use the
second--order Boltzmann--Gibbs principle.
More precisely, we prove that
\begin{equation}\label{eq_asym_tightness_estimate}
	\begin{aligned}
		\mathbf E_{\rho,L}
		\big[
		\big|
		\mathcal{R}_{t}^{L,\ell}(\varphi)
		-\mathcal{R}_{s}^{L,\ell}(\varphi)\big|^{2}\big]
		\lesssim_{\rho,\varphi}
		\beta^2\alpha^{-4}
		(t-s)^{3/2}
		\,,
	\end{aligned}
\end{equation}
for $ \ell=\lfloor L\sqrt{t-s}\rfloor \vee \ell_{0}$.
We consider two cases:
\begin{enumerate}
	\item If $ \lfloor L\sqrt{t-s}\rfloor \geqslant \ell_{0} $, we
	      apply
	      Proposition~\ref{prop_BG} and see that
	      \begin{equation*}
		      \begin{aligned}
			      \mathbf E_{\rho,L}
			      \big[
			      \big|
			      \mathcal{R}_{t}^{L,\ell}(\varphi)
			      -\mathcal{R}_{s}^{L,\ell}(\varphi)\big|^{2}\big]
			       & \lesssim_\rho
			      \frac{\beta^2}{ \alpha^{3} L^{2}} \Big(\frac{1}{ \alpha}
			      +
			      \frac{\ell^{2}}{  \alpha \ell +1 }
			      +
			      \frac{(t-s)L^{2}}{ \ell^2}
			      \Big)
			      %\alpha^{-8}
			      %  \Big(
			      %L^{-2} \ell + \frac{t-s}{\ell^{2}}
			      %\Big)
			      \int_{s}^{t} \sum_{x \in \ZZ}|
			      \nabla_x^L\varphi_{r}|^{2} \ud r\,.
		      \end{aligned}
	      \end{equation*}
	      By Lemma~\ref{lem_usefulestimates_phi}\,(i), we have
	      \begin{equation}\label{eq_quant_blockremainder1}
		      \begin{aligned}
			      \mathbf E_{\rho,L}
			      \big[
			      \big|
			      \mathcal{R}_{t}^{L,\ell}(\varphi)
			      -\mathcal{R}_{s}^{L,\ell}(\varphi)\big|^{2}\big]
			      \lesssim_{\rho,\varphi}
			      \frac{\beta^2(t-s)}{ \alpha^{3} L} \Big(\frac{1}{ \alpha}
			      +
			      \frac{\ell^{2}}{  \alpha \ell +1 }
			      +
			      \frac{(t-s)L^{2}}{ \ell^2}
			      \Big)
			      %\beta^2\alpha^{-8}
			      %\left(
			      %		\frac{ \ell}{L} (t-s) + \frac{L(t-s)^{2}}{
			      %		\ell^{2}}\right)
			      \,.
		      \end{aligned}
	      \end{equation}
	      Finally, because
	      $ \ell = \lfloor L \sqrt{t-s}\rfloor \geqslant \ell_{0}$,
	      the previous display can be upper bounded by
	      \begin{equation*}
		      \begin{aligned}
			      \mathbf E_{\rho,L}
			      \big[
			      \big|
			      \mathcal{R}_{t}^{L,\lfloor L\sqrt{t-s}\rfloor}(\varphi)
			      -\mathcal{R}_{s}^{L,\lfloor L\sqrt{t-s}\rfloor}(\varphi)
			      \big|^{2}\big]
			      \lesssim_{\rho,\varphi}
			      \beta^2\alpha^{-4}
			      (t-s)^{3/2}
			      \,.
		      \end{aligned}
	      \end{equation*}

	\item If $ L\sqrt{t-s} < \ell_{0} $, then we directly estimate the expression using the triangle
	      inequality
	      \begin{equation*}
		      \begin{aligned}
			      \mathbf E_{\rho,L}
			      \big[
			      \big|
			      \mathcal{R}_{t}^{L,\ell}(\varphi)
			      -\mathcal{R}_{s}^{L,\ell}(\varphi)\big|^{2}\big]^{1/2}
			      \leqslant
			      \beta
			      \int_{s}^{t}
			      %\EE_{\rho}\Big[
			      \Big\|
			      \sum_{x \in \ZZ} \nabla_x^L  \varphi_{r}
			      \tau_{x}\big( W ( \eta) - \mQ_{\rho}( \ell,
			      \eta)\big)\Big\|_{L^{2}( \nu_{\rho})}
			      %\Big]
			      \ud r\,.
		      \end{aligned}
	      \end{equation*}
	      Now, we have (uniformly in $r$ and $ L$)
	      \begin{equation*}
		      \begin{aligned}
			       & \EE_{\rho}\Big[\Big|
			      \sum_{x \in \ZZ} \nabla_x^L  \varphi_{r}
			      \tau_{x}\big( W ( \eta) - \mQ_{\rho}( \ell,
			      \eta)\big)\Big|^{2}
			      \Big]                   \\
			       & 	\leqslant 2\,
			      \EE_{\rho}\Big[\Big|
			      \sum_{x \in \ZZ} \nabla_x^L  \varphi_{r}
			      W_{x} ( \eta) \Big|^{2}
			      \Big]
			      +
			      2\,
			      \EE_{\rho}\Big[\Big|
			      \sum_{x \in \ZZ} \nabla_x^L  \varphi_{r}
			      \tau_{x}\mQ_{\rho}( \ell,
			      \eta)\Big|^{2}
			      \Big]                   \\
			       & \lesssim_{\rho}
			      \Big( \ell_{0}\,
			      \alpha^{-2} +
			      \alpha^{-3}\ell^{-1}
			      \Big)
			      \sum_{x \in \ZZ} | \nabla_x^L
			      \varphi_{r}|^{2}	\,,
		      \end{aligned}
	      \end{equation*}
	      where in the last inequality we performed the same estimate as in
	      \eqref{eq_blockcrossvariation} which used that both $ W( \eta) $ and $ \mQ_{\rho} (\ell, \eta)$ are
	      supported on blocks of size $  \ell_{0} $ and $\ell $, respectively,
	      and therefore only $\ell$--many crossvariation terms remain.
	      We also recall \eqref{eq_secondmoment_Q} and that $\EE_{\rho}[ W ( \eta)^{2}] =
		      \chi ( \rho)^{2}\lesssim_{\rho} \alpha^{-2}$.
	      Thus, by Lemma~\ref{lem_usefulestimates_phi}\,(i), we conclude that
	      \begin{equation*}
		      \begin{aligned}
			      \mathbf E_{\rho,L}
			      \big[
			      \big|
			      \mathcal{R}_{t}^{L,\ell}(\varphi)
			      -\mathcal{R}_{s}^{L,\ell}(\varphi)\big|^{2}\big]
			      \lesssim_{\rho,\varphi}
			      \beta^2\alpha^{- 4}
			      (t-s)^{3/2}\,, \quad \forall \ell \geqslant
			      \ell_{0}\,.
		      \end{aligned}
	      \end{equation*}

\end{enumerate}

\textbf{Step 3.}  \textit{Summary of antisymmetric drift.}
For later reference, it will be convenient to summarise the obtained bound for the
antisymmetric drift, which is an immediate consequence of the bounds obtained for
$\mathcal{B}^{L,\ell}(\varphi)$, $\mathcal{R}^{L,\ell}(\varphi)$, and
Lemma~\ref{lem_antisym_rewrite}:

\begin{corollary}[Antisymmetric increment estimate]
	\label{cor_asym_increment_estimate}
	For every $\varphi\in\mS(\RR)$
	there exists a $ L_{0} \in \NN$ such that
	\begin{equation*}
		\sup_{L \geqslant L_{0}}
		\mathbf E_{\rho,L}
		\big[
			|\mZ_t^L(\varphi)-\mZ_s^L(\varphi)|^2
			\big]
		\lesssim_{\rho,\varphi,T}
		\beta^2\alpha^{-4}
		(t-s)^{3/2}\,, \quad \forall 0\leq s<t\leq T\,.
	\end{equation*}
\end{corollary}

\begin{proof}
	We claim that
	\begin{equation*}
		\mZ_t^L(\varphi)-\mZ_s^L(\varphi)
		=
		\big(\mathcal{B}_{t}^{L,\ell}(\varphi)
		-\mathcal{B}_{s}^{L,\ell}(\varphi)\big)
		+
		\big(\mathcal{R}_{t}^{L,\ell}(\varphi)
		-\mathcal{R}_{s}^{L,\ell}(\varphi)\big)
		+
		o_{L^{2}( \PP),s,t}(1)\,,
	\end{equation*}
	for every $ \ell \geqslant \ell_{ 0}$.
	The first two increments on the right are controlled by \eqref{eq_tight_block} and \eqref{eq_asym_tightness_estimate} with $ \ell = \lfloor L
	\sqrt{t-s}\rfloor \vee \ell_{0}$. It is only left to justify the stated control of the last term.
	Indeed,  the squared $L^2(\mathbf P_{\rho,L})$-norm of the last
	term is uniformly bounded (in $L$ and $0\leq s<t\leq T$) by
	\begin{equation*}
		\begin{aligned}
			\beta^{2} \alpha^{-1}
			C_{\rho,\varphi}
			L^{-1}(t-s)^2
			\lesssim_{\rho,\varphi,T}
			\beta^{2} \alpha^{-1}
			(t-s)^{3/2}\,,
		\end{aligned}
	\end{equation*}
	where we  applied Lemma~\ref{lem_antisym_rewrite} and the Cauchy--Schwarz
	inequality.
	% We use the estimate \eqref{eq_tight_block} (with $ \ell = \lfloor L
	% \sqrt{t-s}\rfloor$, provided it is larger than $ \ell_{0}$) for the block part, which gives
	% \begin{equation*}
	% 	\mathbf E_{\rho,L}
	% 	\big[
	% 		|\mathcal{B}_{t}^{L,\ell}(\varphi)
	% 		-\mathcal{B}_{s}^{L,\ell}(\varphi)|^2
	% 		\big]
	% 	\lesssim_{\rho,\varphi}
	% 	\beta^2\alpha^{-3}
	% 	(t-s)^{3/2}.
	% \end{equation*}
	% In the
	% regime $L\sqrt{t-s} < \ell_0$ one takes $\ell=\ell_0$ in \eqref{eq_tight_block}  which yields the same bound.
	% Together with the estimate \eqref{eq_asym_tightness_estimate} for the block remainder,
	% this concludes the claim.
\end{proof}

\subsection*{Conclusion of the tightness arguments}

Overall, we established tightness for each of the components in \eqref{eq_dynkin_tight}, when tested against a
Schwartz function $ \varphi$. By Mitoma's theorem, this yields tightness (with continuous limit
trajectories) in $
	D([0,T] , \mS' ( \RR ) )$.
% \cite{Mitoma83}.
In particular, $ u_{L}$ is tight in $D([0,T] , \mS' ( \RR ) )$, and
every limit point is supported on $ C([0,T] , \mS ' (\RR ) )$.

\section{Identification of limit points}

Energy solutions were introduced by Gon{\c c}alves and Jara
\cite{GoncJara14}, who proved existence of solutions to the stochastic Burgers equation through microscopic approximation. 
Later, uniqueness was also established in the stationary setting \cite{GubPerk}.
Energy solutions provide a
martingale formulation of the stochastic Burgers equation
\eqref{eq_burgers} in which the nonlinearity is defined by a renormalised
mollification procedure.
Thus, after tightness, the remaining task is to
identify each limit point of $ ( u_{L})_{L \in \NN}$ as an energy solution to the stochastic Burgers
equation.

\subsection{Energy solutions of the stochastic Burgers equation}\label{sec_energy_solutions}

First, we recall the definition of energy solutions to the stochastic Burgers equation in the controlled process formulation of \cite{GubPerk}. 
% There the field is controlled by the Ornstein--Uhlenbeck part of the
% equation, while the nonlinear drift is the limit of the regularised quadratic
% fields $\mN^\delta$ below.

\begin{definition}[Controlled process]\label{def_controlled_process}
	Let $(u, \mZ)$ be a pair of processes with trajectories in
	$C([0,T],\mS'(\RR))$ such that $\mZ_0=0$.
	We say the pair is controlled by the
	Ornstein--Uhlenbeck part of \eqref{eq_burgers} if the following conditions
	hold.
	\begin{enumerate}
		\item \emph{Stationarity.} For every $t \in [0,T]$, $u_{t}$ is a spatial
		      white noise of variance $\chi(\rho)$.
		      %, that is
		      %\begin{equation*}
		      %			\begin{aligned}
		      %					\EE\big[ u_{t}(\varphi)u_{t}(\psi)\big]
		      %						=
		      %						\chi(\rho) \int_{\RR}\varphi(x)\psi(x)\ud x\,,
		      %					\qquad \varphi,\psi \in \mS(\RR).
		      %				\end{aligned}
		      %			\end{equation*}

		\item \emph{Forward martingale problem.} For every $\varphi \in \mS(\RR)$,
		      the process
		      \begin{equation*}
			      \begin{aligned}
				      M_{t}(\varphi)
				      :=
				      u_{t}(\varphi)-u_{0}(\varphi)
				      - \frac{\alpha}{2}
				      \int_{0}^{t} u_{s}(\partial_{x}^{2}\varphi)\ud s
				      - \mZ_t(\varphi)
			      \end{aligned}
		      \end{equation*}
		      is a continuous martingale with respect to the natural filtration of
		      $(u,\mZ)$, with quadratic variation
		      \begin{equation}\label{eq_energy_qv}
			      \begin{aligned}
				      \langle M(\varphi)\rangle_{t}
				      =
				      t(\alpha\rho+\rho^{2})
				      \| \partial_{x}\varphi \|_{L^{2}(\RR)}^{2}\,.
			      \end{aligned}
		      \end{equation}

		\item \emph{Vanishing quadratic variation.}
		      For every $\varphi \in \mS(\RR)$,
		      the process $t\mapsto \mZ_t(\varphi)$ has zero quadratic
		      variation in the sense of Russo--Vallois
		      \cite{RussoVallois93}, that is
		      \begin{equation}\label{eq_energy_zero_qv}
			      \begin{aligned}
				      \frac{1}{\delta}\int_{0}^{T-\delta}
				      \big|
				      \mZ_{s+\delta}(\varphi)-\mZ_s(\varphi)
				      \big|^{2}
				      \ud s
				      \longrightarrow 0
			      \end{aligned}
		      \end{equation}
		      in probability, as $\delta \to 0$.

		\item \emph{Time reversal.} For every $\varphi \in \mS(\RR)$, the process
		      \begin{equation*}
			      \begin{aligned}
				      \overleftarrow{M}_{t}(\varphi)
				      :=
				      u_{T-t}(\varphi)-u_{T}(\varphi)
				      - \frac{\alpha}{2}
				      \int_{0}^{t} u_{T-s}(\partial_{x}^{2}\varphi)\ud s
				      +\big( \mZ_T(\varphi)-\mZ_{T-t}(\varphi)\big)
			      \end{aligned}
		      \end{equation*}
		      is a continuous martingale with quadratic variation \eqref{eq_energy_qv}, with respect to the backward filtration
		      \begin{equation*}
			      \overleftarrow{\mF}_t
			      :=
			      \sigma\big(
			      u_{T-s}(\varphi),
			      \mZ_T(\varphi)-\mZ_{T-s}(\varphi):
			      0\leqslant s\leqslant t,\
			      \varphi\in\mS(\RR)
			      \big)\,.
		      \end{equation*}
	\end{enumerate}
\end{definition}

Let $(u,\mZ)$ be a controlled process.
For every $ \delta >0 $ and $ x \in \RR $, we define the box kernel
\begin{equation*}
	i_{\delta,x}(y)
	:=
	\delta^{-1}\mathds{1}_{[x,x+\delta)}(y),
	\qquad \forall y\in\RR\, ,
\end{equation*}
and the associated box-regularised
quadratic drift
\begin{equation}\label{eq_energy_Adelta}
	\begin{aligned}
		\mN_{t}^{\delta}(\varphi)
		:=
		-
		\int_{0}^{t}\int_{\RR}
		u_{r}(i_{\delta,x})^{2}
		\partial_{x}\varphi(x)\ud x \ud r\,,
		\qquad
		t \in [0, T]
		\,.
	\end{aligned}
\end{equation}
Since the box kernel belongs to $L^{1}(\RR)\cap L^{2}(\RR)$ and has integral
one, the real-line energy estimate for controlled processes
\cite[Proposition~2.2]{GubPerk} gives, for every $\varphi\in\mS(\RR)$, the existence of the limit
\begin{equation}\label{eq_energy_A}
	\begin{aligned}
		\mN_{t}(\varphi)
		:=
		\lim_{\delta \to 0}
		\mN_{t}^{\delta}(\varphi)\,,
	\end{aligned}
\end{equation}
uniformly on compacts $[0,T]$ in probability\footnote{We write $ \mathbf P $ for the probability
	measure of the underlying probability space, and don't classify it any further.}.
This limit is independent of the choice of box kernel $ i_{\delta, x}$ we made
and represents $\partial_x(u^2)(\varphi)$ in the
distributional sense \cite[Proposition~2.2]{GubPerk}.

\begin{definition}[Stationary energy solution]\label{def_energy_solution}
	An $\mS'(\RR)$-valued process $u$ is called a stationary energy solution of the
	stochastic Burgers equation
	\eqref{eq_burgers} if there exists a process $\mZ$ such that
	$(u,\mZ)$ is controlled by the Ornstein--Uhlenbeck equation and
	\begin{equation*}
		\begin{aligned}
			\mZ_t(\varphi)=-\beta\mN_{t}(\varphi)\,,\   \text{almost
				surely, } \
			\qquad \forall  t\in[0,T],\ \varphi\in\mS(\RR).
		\end{aligned}
	\end{equation*}
\end{definition}

Uniqueness of such energy solutions was established in
\cite[Theorem~2.4]{GubPerk}:

\begin{theorem}[Uniqueness of energy solutions, \cite{GubPerk}]\label{thm_energy_uniqueness}
	For fixed $\alpha>0$, $\rho>0$ and $\beta\in \RR$, stationary energy
	solutions to the stochastic Burgers equation \eqref{eq_burgers} are unique in law.
\end{theorem}

Hence, the following proposition suffices to conclude the main result in
Theorem~\ref{thm_main}.

\begin{proposition}[Identification of subsequential limits]\label{prop_identification_limit_points}
	Any subsequential limit of $( u_{L} )_{L \in \NN}$ in
	$D([0,T],\mS'(\RR))$ is a
	stationary energy solution of the stochastic Burgers equation
	\eqref{eq_burgers}.
\end{proposition}

The proof of the proposition is the content of the next section.

\subsection{Proof of Proposition~\ref{prop_identification_limit_points}}
\label{sec_proof_identification_limit_points}

The central object in the proof of Proposition~\ref{prop_identification_limit_points} and
in the verification that any limit point is an energy solution is the prelimit martingale formulation from \eqref{eq_dynkin_tight} with
\begin{equation}\label{eq_dynkin_final}
	\begin{aligned}
		M_t^L(\varphi)
		%  & :=
		% u_{L}(t,\varphi)-u_{L}(0,\varphi)
		% -\int_0^tL^{2}\gens u_{L}(s,\varphi)\ud s
		% -\int_0^t(\partial_s+L^{3/2}\gena)
		% u_{L}(s,\varphi)\ud s \\
		 & :=
		u_{L}(t,\varphi)-u_{L}(0,\varphi)
		- \mD_{t}^{L}( \varphi)
		- \mZ_t^L( \varphi) \,.
	\end{aligned}
\end{equation}
% Besides establishing controlledness of any joint limit points
% (Lemma~\ref{lem_control}),
% the key steps are to identify the symmetric drift
% (Lemma~\ref{lem_blue_identification_symmetric_drift}) and the antisymmetric drift
% (Lemma~\ref{lem_blue_identification_nonlinear_drift}).
Throughout this section, we use without further mention the following consequence of
Section~\ref{sec_tight}: $(u_{L})_{L} $, $ (\mD^{L})_{L}$, $ (\mZ^{L})_{L}$, and
$ (M^{L})_{L}$ are tight in $ D ([0,T],\mS ' (\RR) )$ with limit points in
$ C ( [0,T], \mS ' ( \RR) )$.
\par\smallskip\smallskip

To establish energy solutions, we have to show controlledness and identify the antisymmetric drift. 
We summarise the output of these two key steps here, and then prove
Proposition~\ref{prop_identification_limit_points}.

\begin{lemma}[Controlledness of limit points]\label{lem_control}
	Any joint subsequential limit of $( u_{L}, \mZ^{L} )_{L \in \NN}$ in
	$D([0,T],\mS'(\RR) \times \mS ' (\RR) )$
	is a controlled process in the sense of
	Definition~\ref{def_controlled_process}.
\end{lemma}

Once we have established that the limit is a controlled process, it remains to identify $ \mZ$.

\begin{lemma}[Identification of the nonlinear drift]
	\label{lem_blue_identification_nonlinear_drift}
	Assume that $(u_{L_n}, \mZ^{L_{n}})\Rightarrow (u, \mZ )$
	%in $D([0,T],\mS'(\RR)\times \mS' (\RR ))$
	and that  $( u , \mZ )$ is a controlled process.
	Then, for every $t\in[0,T]$ and $\varphi\in\mS(\RR)$,
	\begin{equation*}
		\mZ_t(\varphi)=-\beta\mN_t(\varphi)
		\quad\text{almost surely}\,,
	\end{equation*}
	where
	%$\mZ^{L}$ is the antisymmetric drift \eqref{eq_def_mZ} and
	$\mN$ is the limiting box-kernel drift defined in \eqref{eq_energy_A}.
\end{lemma}

The proof of Lemma~\ref{lem_control} is the content of Section~\ref{sec_control} and the
proof of Lemma~\ref{lem_blue_identification_nonlinear_drift} is the content of
Section~\ref{sec_nonlinear}.
We now have all the necessary ingredients to state the
proof of  Proposition~\ref{prop_identification_limit_points}:

\begin{proof}[Proof of Proposition~\ref{prop_identification_limit_points}]
	Lemma~\ref{lem_control} establishes that any accumulation point is a controlled
	process.
	Hence the limiting box-kernel drift $ \mN$ in \eqref{eq_energy_A} is available by
	\cite[Proposition~2.2]{GubPerk}.
	On the other hand,
	Lemma~\ref{lem_blue_identification_nonlinear_drift} identifies the
	nonlinear drift on this same joint limiting law and gives, for every
	$t\in[0,T]$ and $\varphi\in\mS(\RR)$,
	\begin{equation*}
		\mZ_t(\varphi)=-\beta\mN_t(\varphi)
		\quad\text{almost surely}.
	\end{equation*}
	We conclude that the limit is a
	stationary energy solution of the stochastic Burgers equation
	\eqref{eq_burgers}, see Definition~\ref{def_energy_solution}.
\end{proof}

\subsubsection{Controlledness of limit points}\label{sec_control}

First, we
identify the limit of the term associated to the symmetric drift $ \mD_{t}^{L}( \varphi) $ from \eqref{eq_def_mD}.
\begin{lemma}[Identification of the symmetric drift]
	\label{lem_blue_identification_symmetric_drift}
	Let $\varphi\in\mS(\RR)$ and assume that $u_{L_n}\Rightarrow u$ in
	$D([0,T],\mS'(\RR))$. Then weakly in $C([0,T],\RR)$
	\begin{equation*}
		\big(
		\mathcal{D}_t^{L_{n}}(\varphi)
		\big)_{ t \in [0,T]}
		\Rightarrow
		\bigg(
		\frac{\alpha}{2}
		\int_0^t u_r(\partial_x^2\varphi)\ud r
		\bigg)_{ t \in [0,T] }\,.
	\end{equation*}
\end{lemma}

\begin{proof}
	By \eqref{eq_supp_sym}, we have
	%\begin{equation*}
	%	\int_0^t L^{2}\gens u_L(r,\varphi)\ud r
	%	=
	%	\frac{\alpha}{2}
	%	\int_0^t u_L(r,\Delta^L\varphi)\ud r .
	%\end{equation*}
	%Thus
	\begin{equation*}
		\mathcal{D}_t^L(\varphi)
		-
		\frac{\alpha}{2}
		\int_0^t u_L(r,\partial_x^2\varphi)\ud r
		=
		\frac{\alpha}{2}
		\int_0^t
		\frac{1}{ \sqrt{L}}
		\sum_{x\in\ZZ}
		\big(
		\Delta_x^L\varphi_r
		-
		\partial_x^2\varphi_r(L^{-1}x)
		\big)(\eta_x(r)-\rho)\ud r .
	\end{equation*}
	Now, using Cauchy--Schwarz in time, stationarity, and the product structure of
	$\nu_\rho$ yields
	\begin{equation*}
		\begin{aligned}
			 & \mathbf E_{\rho,L}
			\left[
				\sup_{t\leq T}
				\left|
				\int_0^t
				\frac{1}{ \sqrt{L}}
				\sum_{x\in\ZZ}
				\big(
				\Delta_x^L\varphi_r
				-
				\partial_x^2\varphi_r(L^{-1}x)
				\big)(\eta_x(r)-\rho)\ud r
				\right|^2
			\right]               \\
			 & \leqslant
			T\int_0^T
			\chi( \rho) \frac{1}{L} \sum_{x\in\ZZ}
			\big(
			\Delta_x^L\varphi_r
			-
			\partial_x^2\varphi_r(L^{-1}x)
			\big)^2
			\ud r \,.
		\end{aligned}
	\end{equation*}
	The right--hand side vanishes by Lemma~\ref{lem_usefulestimates_phi}\,(iv),
	thus, the replacement holds in $L^2 ( \mathbf P_{\rho,L})$. The remaining convergence
	follows by the continuous mapping theorem from $u_{L_n}\Rightarrow u$ and the continuity of the integral, since $
		u \in C ([0,T] , \mS ' ( \RR ))$.
\end{proof}

We are now ready to prove that any subsequential limit is a controlled process.

\begin{proof}[Proof of Lemma~\ref{lem_control}]
	Fix a subsequence
	$(L_n)_{n\geq1}\subset\NN$ with $L_n\to\infty$ along which $(u_{L_n},
		\mZ^{L_{n}})$ converges in law
	to $(u, \mZ )$ in $D([0,T],\mS'(\RR) \times \mS ' ( \RR ) )$. 
	The tightness estimates in the previous
	section allow us, along a further subsequence (which we do not relabel), to
	assume joint weak convergence of the triple $ ( u_{L_n},
		\mZ^{L_n} , M^{L_n} ) $ to $ ( u, \mZ , M)$.
	The limit has continuous trajectories.
	For brevity, we will write $ L$ instead of $
		L_n$.
	We verify each of the points in Definition~\ref{def_controlled_process}:
	\par\smallskip\smallskip

	\emph{Stationarity.} Fix $t\in[0,T]$ and $ \varphi \in\mS (\RR) $. Then under
	$ \nu_{\rho}$ the field $ u_{L}(t, \varphi; \eta)$ is the weighted sum of centred iid random
	variables with limiting second moment
	\begin{equation*}
		\begin{aligned}
			\EE_{\rho}[ | u_{L}(t, \varphi; \eta) |^{2}]
			=
			\frac{\chi ( \rho ) }{L} \sum_{x \in \ZZ}
			\varphi_{t,x}^{2}
			\to
			\chi( \rho) \| \varphi\|_{L^{2}( \RR)}^{2}\,.
		\end{aligned}
	\end{equation*}
	Next, we verify Lindeberg's condition.
	For every $\delta>0$, using \eqref{eq_quant_mom_bound}, we have
	\begin{equation}\label{eq_supp2_tightness_field}
		\begin{aligned}
			\frac{1}{L}
			\sum_{x\in\ZZ}
			\varphi_{t,x}^{2}
			\EE_{\rho}\big[( \eta_{x}- \rho)^2
				\mathbf 1_{\{ |\varphi_{t,x}| | \eta_{x}- \rho|>\delta \sqrt{ L}\}}\big]
			 & \leq
			\frac{\EE_{\rho}[| \eta_{0}- \rho|^4] }{ \delta^{2}  L^{2}}
			\sum_{x\in\ZZ}
			\varphi_{t,x}^{4}
			\lesssim_{\rho,\varphi,\delta}
			\frac{1}{ \alpha^{3}  L}
			\,,
		\end{aligned}
	\end{equation}
	which vanishes as $ L \to \infty$.
	Since $\nu_{\rho}$ is invariant, the argument applies at every time $t$. Hence,
	every one-time marginal
	of $u$ is spatial white noise with variance $\chi(\rho)$.
	\par\smallskip\smallskip

	\emph{Forward martingale problem.}
	Combining the joint convergence of $ ( u_{L},
		\mZ^{L}, M^L)$ with
	Lemma~\ref{lem_blue_identification_symmetric_drift} and \eqref{eq_dynkin_final} gives
	\begin{equation*}
		\begin{aligned}
			M_t(\varphi)
			:=
			u_t(\varphi)-u_0(\varphi)
			-\frac{\alpha}{2}\int_0^t u_s(\partial_x^2\varphi)\ud s
			-\mZ_t(\varphi).
		\end{aligned}
	\end{equation*}
	We see that $M(\varphi)$ is a martingale, because
	for $0\leq s<t\leq T$ we have
	\begin{equation}\label{eq_supp_new_control}
		\begin{aligned}
			\mathbf E
			\big[
				(M_t(\varphi)-M_s(\varphi))F
				\big]
			=
			\lim_{L \to \infty}
			\mathbf E_{\rho,L}
			\big[(M_t^L(\varphi)-M_s^L(\varphi))F_L\big]
			=0\,,
		\end{aligned}
	\end{equation}
	for every bounded continuous functional $F$ of
	finitely many marginals of $(u_r,\mZ_r)_{r\leq s}$.
	Here $ F_{L}$ is the analogous function using the same time marginals but of $ (
		u_{L}, \mZ^{L})$. This can then be extended to the whole natural filtration up to
	time $ s$ using the monotone class
	theorem, see e.g.
	\cite[Appendix~Corollary~4.4]{EthierKurtz86}.
	Indeed, we can pass the limit in \eqref{eq_supp_new_control} inside the expectation by weak convergence and uniform integrability, since
	by Lemma~\ref{lem_quant_qv}
	\begin{equation*}
		\begin{aligned}
			\sup_{L \in \NN }
			\mathbf E_{\rho,L}
			\big[
				|M_t^L(\varphi)-M_s^L(\varphi)|^2
				\big]
			=
			\sup_{L \in \NN }
			\mathbf E_{\rho,L}
			\big[
				\langle M^L(\varphi) \rangle_{t}- \langle M ^L(\varphi) \rangle_{s}
				\big]\lesssim_{\rho, \alpha, \varphi, T}1\,.
		\end{aligned}
	\end{equation*}

	\emph{Quadratic variation of $M(\varphi)$.}
	Let $ \varphi \in \mS ( \RR ) $. Then Lemma~\ref{lem_quant_qv} gives
	\begin{equation*}
		\lim_{ L \to \infty} \langle M^{L}( \varphi) \rangle_{t} = t \, (\alpha\rho+\rho^2) \| \partial_{x} \varphi \|_{L^{ 2}( \RR ) }^{2}\,,
	\end{equation*}
	uniformly over $[0,T]$.
	Thus, we only need to show that indeed  $ \langle M ( \varphi) \rangle_{t}= 	t
		(\alpha\rho+\rho^2) \| \partial_{x} \varphi \|_{L^{2}( \RR)}^{2}$.
	To this end, we verify that $ M_{t} ( \varphi)^{2} - 	t
		(\alpha\rho+\rho^2) \| \partial_{x} \varphi \|_{L^{2}( \RR)}^{2}$ is a
	martingale, which follows the same lines as the martingale identification in the previous point.
	For the necessary uniform integrability, we use the fourth--moment bound
	\cite[Lemma~VII.3.34]{JS}:
	\begin{equation}\label{eq_fourth_mom_mart}
		\begin{aligned}
			 & \mathbf E_{\rho,L}\big[ \sup_{t \in [0, T]} | M_{t}^{L}( \varphi)
			|^{4}\big]\\
			&\lesssim
			\sup_{t \in [0,T]}
			\big\| M_{t}^{L}( \varphi ) - M_{t -}^{L} (
			\varphi) \big\|^{2}_{L^{\infty}( \mathbf P_{\rho,L})}
			\mathbf E_{\rho,L}\big[ \langle M^{L}( \varphi)
				\rangle_{T}^{2} \big]^{1/2}
			+
			\mathbf E_{\rho,L}\big[ \langle M^{L}( \varphi)
			\rangle_{T}^{2} \big]                                                                                 \\
			 & \lesssim_{ \rho,\varphi} L^{-3} \Big( T^{2}+	\frac{T}{ \alpha^{4} L^{3}} + \frac{T^{ 2}}{\alpha L}
			+ T^{2} o_{L}(1)	\Big)^{1/2}+
			\Big(T^{2}+
			\frac{T}{ \alpha^{4} L^{3}} + \frac{T^{ 2}}{\alpha L}
			+ T^{2} o_{L}(1)
			\Big)
			\,,
		\end{aligned}
	\end{equation}
	where we used \eqref{eq_vanish_jumps} and Lemma~\ref{lem_quant_qv} (and
	once more that $ \alpha \chi ( \rho) \lesssim_{\rho} 1$).

	Hence, for fixed $ \alpha$ this yields
	$ \sup_{L \in \NN} \mathbf E_{\rho,L}\big[ \sup_{t \in [0, T]} | M_{t}^{L}( \varphi)
		|^{4}\big]\lesssim_{\rho, \varphi, \alpha} T+  T^{2}$.
	\par\smallskip\smallskip

	\emph{Vanishing quadratic variation of $ \mZ$.}
	Corollary~\ref{cor_asym_increment_estimate}, Fatou's lemma, and the
	convergence of $\mZ^L$ to $\mZ$ imply that
	for $0\leq s<t\leq T$,
	\begin{equation*}
		\EE\big[|\mZ_t(\varphi)-\mZ_s(\varphi)|^{2}\big]
		\lesssim_{\rho,\varphi,T}
		\beta^2\alpha^{-4}
		(t-s)^{3/2}.
	\end{equation*}
	Therefore, for $0<\delta<T$, the Russo--Vallois regularised
	quadratic variation in \eqref{eq_energy_zero_qv} satisfies
	\begin{equation*}
		\EE\bigg[
		\frac{1}{\delta}\int_0^{T-\delta}
		|\mZ_{s+\delta}(\varphi)-\mZ_s(\varphi)|^{2}
		\ud s
		\bigg]
		\lesssim_{\rho,\varphi}
		\beta^2\alpha^{-4}
		\frac{T}{\delta}\int_0^{T-\delta}\delta^{3/2}\ud s
		\lesssim_{\rho,\varphi,T}
		\beta^2 \alpha^{-4}\delta^{1/2}\,,
	\end{equation*}
	which vanishes when $ \delta \to 0 $.
	\par\smallskip\smallskip

	\emph{Time reversal.}
	Fix $T>0$ and let $\overleftarrow{\eta}^{L,T}$ denote the càdlàg reversal
	of $\eta^{L}$ on $[0,T]$, with the left-limit and endpoint conventions of
	Proposition~\ref{prop_ip_exists}. By that proposition,
	$\overleftarrow{\eta}^{L,T}$ solves the martingale problem associated to
	the adjoint generator $\gen_{L}^{\ast}$ \eqref{eq_adjointness} with respect to its reversed
	natural filtration.
	We then define
	the reversed field
	$
		\overleftarrow{u}_{L}(t,\varphi)
		:=
		u_L(T-t,\varphi;\overleftarrow{\eta}^{L,T}(t)).
	$
	Dynkin's formula applied to the reversed process yields a martingale after subtracting
	the drift
	\begin{equation*}
		\begin{aligned}
			\int_0^t
			(\partial_s+L^{2}\gens-L^{3/2}\gena)
			\overleftarrow{u}_{L}(s,\varphi)\ud s
			=
			\int_{T-t}^T
			(-\partial_r+L^{2}\gens-L^{3/2}\gena)
			u_L(r,\varphi)\ud r
			\,.
		\end{aligned}
	\end{equation*}
	The symmetric term is unchanged, and the same replacement estimate as in
	Lemma~\ref{lem_blue_identification_symmetric_drift} holds.
	On the other hand, the antisymmetric drift satisfies
	\begin{equation*}
		\int_{T-t}^{T}
		(-\partial_r-L^{3/2}\gena)
		u_L(r,\varphi)\ud r
		=
		-\big(
		\mZ_T^L(\varphi)-\mZ_{T-t}^L(\varphi)
		\big)\,,
	\end{equation*}
	which converges to
	$-(\mZ_T(\varphi)-\mZ_{T-t}(\varphi))$.
	Lastly,
	the martingale
	properties pass to the limit by the same argument as for the forward
	martingale, now tested
	against bounded continuous functionals $F$ of the backward filtration.
	% The reversed quadratic variation is unchanged, since the adjoint
	% dynamics only exchanges $p_L$ and $q_L$.
	\par\smallskip\smallskip

	\noindent Overall, any subsequential limit $ (u, \mZ )$ is a controlled
	process in the sense of Definition~\ref{def_controlled_process}.
\end{proof}

\subsubsection{Identification of the nonlinear drift}\label{sec_nonlinear}

It is only left to establish Lemma~\ref{lem_blue_identification_nonlinear_drift}.
To this end, we first define the microscopic block approximation
\begin{equation}\label{eq_def_N_L_delta}
	\begin{aligned}
		\mN_{t}^{L,\ell}(\varphi)
		:=
		-
		\int_0^t
		\sum_{x\in\ZZ}
		\nabla_x^L\varphi_r\,
		\tau_x\mQ_\rho(\ell,\eta(r))\ud r \quad \forall \ell \geqslant 2\,.
	\end{aligned}
\end{equation}

\begin{lemma}\label{lem_blue_block_to_box_mollifier}
	Assume that
$		(u_{L_n},\mZ^{L_n})\Rightarrow(u,\mZ)$
		in
		$D([0,T],\mS'(\RR)\times\mS'(\RR))$,
	and that $(u,\mZ)$ is a controlled process.
	Let $\delta>0$ and set
	$\ell_n=\lfloor\delta L_n\rfloor$.
	%, with $i_{\delta,x}$ and $\mN^\delta$ as in \eqref{eq_energy_Adelta}.
	Then, for every $\varphi\in\mS(\RR)$ and $t \in [0,T]$
	\begin{equation*}
		\big(
		u_{L_n},\mZ^{L_n},
		\mN_t^{L_n,\ell_n}(\varphi)
		\big)
		\Rightarrow
		\big(
		u,\mZ,\mN_t^\delta(\varphi)
		\big)\,,
	\end{equation*}
	where $ \mN_{t}^{\delta}$ was defined in \eqref{eq_energy_Adelta}.
\end{lemma}

The proof is based on the fact that $\mQ_\rho(\ell_n,\eta)$ can be written as the fluctuation field tested against
a box kernel of width $\delta$, shifted by the moving frame.
This allows us to reduce the weak convergence in the statement to that of $
	u_{L_{n}}( i_{\delta})^{2}$.
Because $i_{\delta}$ is not smooth, the proof requires several tedious approximations which are not particularly insightful. We therefore defer it until after the proof of
Lemma~\ref{lem_blue_identification_nonlinear_drift},
where we identify the nonlinear drift
term.

\begin{proof}[Proof of Lemma~\ref{lem_blue_identification_nonlinear_drift}]
	First, we apply Lemma~\ref{lem_antisym_rewrite} and Proposition~\ref{prop_BG} to derive
	\begin{equation}\label{eq_nonlinear_supp2}
		\begin{aligned}
			\mathbf E_{\rho,L}
			\left[
				\left|
				\mZ^{L}_{t}( \varphi)
				+
				\beta
				\mN_{t}^{L,\lfloor\delta L\rfloor}(\varphi)
				\right|^2
				\right]
			\lesssim_{\rho , \varphi}
			\frac{t^{2}\beta^{2}}{ \alpha L} +
			\frac{t \beta^{2}}{ \alpha^{4}}
			\Big(
			\frac{\lfloor\delta L\rfloor }{L}
			+ \frac{t L }{\lfloor\delta L\rfloor^{2}}
			\Big) \,,
		\end{aligned}
	\end{equation}
	where we used $ \| \nabla_{x}^{L} \varphi\|_{\ell^{2}}^{2}\lesssim L $ (see 
	Lemma~\ref{lem_usefulestimates_phi}~(i)).
	% Notice that the right--hand side
	% vanishes if we  first take the limit $L \to \infty$ and then $ \delta \to 0 $.
	Thus, by the Portmanteau theorem and~\eqref{eq_nonlinear_supp2}
	\begin{equation*}
		\begin{aligned}
			\mathbf E\left[
				\left|
				\mZ_t(\varphi)+\beta\mN_t^\delta(\varphi)
				\right|^2
			\right]
			&\leqslant
			\liminf_{n\to\infty}
			\mathbf E_{\rho,L_n}\left[
				\left|
				\mZ_t^{L_n}(\varphi)
				+\beta\mN_t^{L_n,\lfloor\delta L_n\rfloor}(\varphi)
				\right|^2
			\right]
			\lesssim_{\rho,\alpha,\beta,\varphi,t}\delta \,,
		\end{aligned}
	\end{equation*}
	for every fixed $ \delta >0$.
	Here we used $\big(
		\mZ^{L_n},
		\mN_t^{L_n,\lfloor\delta L_n\rfloor}(\varphi)
		\big)
		\Rightarrow
		\big(
		\mZ,\mN_t^\delta(\varphi)
		\big)$ from  
	Lemma~\ref{lem_blue_block_to_box_mollifier}.

	Together with the fact that $\mN_t^\delta(\varphi)\to\mN_t(\varphi)$ in probability as
	$\delta\downarrow0$ (see \eqref{eq_energy_A}), we conclude  that
		$\mZ_t(\varphi)=-\beta\mN_t(\varphi)$
	almost surely
	for every fixed $t$ and $\varphi$.
\end{proof}

\begin{proof}[Proof of Lemma~\ref{lem_blue_block_to_box_mollifier}]
	In the following, we suppress the $n$-dependence by writing $ L=L_{n}$ and also abuse notation by writing $ \delta L $ for the integer part $ \lfloor \delta L
		\rfloor $.
	Furthermore, $\mN^{L,\delta L }$ can be represented alternatively by 
	\begin{equation}\label{eq_blockboxN_equiv}
		\begin{aligned}
			\mN_{t}^{L,\delta L}(\varphi)
			=
			-
			\int_0^t
			\frac{1}{L} \sum_{x\in\ZZ}
			u_{L}
			(r,i_{ \delta, d_{x,r}})^2
			\nabla_{L d_{x,r}}^{L} \varphi \ud r \,,
		\end{aligned}
	\end{equation}
	with $ d_{x,r}:=  L^{-1}x -\sqrt{L} j'(\rho)r $.
	To see this, we first notice 
	\begin{equation*}
		\begin{aligned}
			u_{L}(r,i_{ \delta,  d_{x,r}})
			=
			\sqrt{L}
			\Big(
			\frac{1}{ \delta L }\sum_{z=x}^{x+ \delta L -1} (\eta_z(r)-\rho)
			\Big)
			= 
			\sqrt{L}( \overline{\eta}_{x}^{\delta L }- \rho) \,,
		\end{aligned}
	\end{equation*}
	where $ d_{x,r} $ is precisely undoing the moving time frame, and therefore
	\begin{equation}\label{eq_supp1_nonlinear}
		\begin{aligned}
			\tau_x\mQ_\rho(\delta L ,\eta(r))
			=
			\frac{1}{L}
			\Big(
			u_{L}(r,i_{\delta,  d_{x,r}})^2
			-
			\frac{\chi( \rho)}{\delta}
			\Big)\,.
		\end{aligned}
	\end{equation}
	Now, recalling \eqref{eq_def_N_L_delta} we can disregard the constant term on the right--hand side of \eqref{eq_supp1_nonlinear} since the sum over discrete gradients vanishes.
	This implies \eqref{eq_blockboxN_equiv} once we notice that
	\begin{equation*}
		\begin{aligned}
			\nabla_x^{L}\varphi_r
			=
			L \big(
			\varphi_{r}(
			L^{-1} (x+1)) - \varphi_{r} ( L^{-1}x )\big)
			= L \big(
			\varphi ( d_{x, r} + L^{-1})
			-
			\varphi ( d_{x, r} )
			\big)
			= \nabla_{ L d_{x,r}}^{L} \varphi\,.
		\end{aligned}
	\end{equation*}
	We proceed by introducing several auxiliary terms and control each of the differences
	separately:

	\textbf{Step 1.} \emph{Continuous approximation.} We start by replacing the Riemann sum with the corresponding continuous integral.
	Namely, we prove that
	\begin{equation}\label{eq_replace_continuum_deriv}
		\begin{aligned}
			 & %\lim_{L \to \infty} 
			 \mathbf E_{\rho,L}\bigg[
			\bigg|
			\int_0^t\bigg(
			\frac{1}{L} \sum_{x\in\ZZ}
			u_{L}
			(r,i_{ \delta, d_{x,r}})^2
			\nabla_{L d_{x,r}}^{L} \varphi
			-
			%\frac{1}{L}
			\int_{\RR}
			u_{L}(r, i_{\delta, x} )^{2}
			\partial_{x} \varphi (x)
			\ud x\bigg)
			\ud r
			\bigg|^2
			\bigg]
			\lesssim_{\rho, \varphi}  \frac{t^{2}}{
				\delta^{3}
				\alpha^{3}L} \,.
		\end{aligned}
	\end{equation}
	It suffices to control the integrand.
	To this end, we establish some useful estimates on $ u_{L}(r, i_{\delta,z})^{2} $:

	\begin{enumerate}
		\item
		      The moment estimates \eqref{eq_quant_mom_bound}
		      imply
		      \begin{equation}\label{eq_u_phi_estimate}
			      \begin{aligned}
				      \EE_\rho\big[
					      u_L(r, \varphi; \eta)^{4}
					      \big]
				       & \lesssim_{\rho}
				      \frac{1}{\alpha^{3} L^{2}} \sum_{x \in \ZZ}|
				      \varphi_{r,x}|^{4}
				      +
				      \Big(
				      \frac{ 1}{ \alpha L} \sum_{x \in \ZZ}
				      | \varphi_{r,x}|^{2}
				      \Big)^{2} \lesssim_{\rho} \alpha^{-3}
				      L^{-2}\|
				      \varphi_{r, \cdot}\|^{4}_{\ell^{2}}\,.
			      \end{aligned}
		      \end{equation}
		      In particular, for the box kernel
		      \begin{equation}\label{eq_u_fourth_moment}
			      \begin{aligned}
				      \sup_{z \in \RR}
				      \EE_\rho[u_L(r,i_{\delta,z}; \eta)^4]
				      \lesssim_{\rho} \alpha^{-3}
				      L^{-2}\|i_{\delta,z} (\cdot/L)\|^{4}_{\ell^{2}}
				      \lesssim_{\rho} \alpha^{-3} \delta^{-2}\,,
			      \end{aligned}
		      \end{equation}
			  for all large enough $L$, 
		      where we used $ \|i_{\delta,z} (\cdot/L)\|^{2}_{\ell^{2}}\lesssim \delta^{-1} L$.

		\item The same fourth-moment estimate as in 1. gives, for
		      $y\in[x,x+1)$,
		      \begin{equation}\label{eq_diff_u_L}
			      \begin{aligned}
				      \left\|
				      u_L(r,i_{\delta,d_{x,r}})^2
				      -u_L(r,i_{\delta,d_{y,r}})^2
				      \right\|_{L^2(\nu_\rho)}
				      \lesssim_{\rho}
				      \frac{1}{ \delta^{3/2} \alpha^{3/2} \sqrt{L} }\,.
			      \end{aligned}
		      \end{equation}
		      Here we first used \eqref{eq_u_fourth_moment} together with
		      \begin{equation}\label{eq_sq_diff_trick}
			      \begin{aligned}
				       & \left\|
				      u_L(r, \varphi)^2
				      -u_L(r, \psi)^2
				      \right\|_{L^2(\nu_\rho)} \\
				       & \leqslant
				      \big(
				      \|
				      u_L(r, \varphi)\|_{L^{4} ( \nu_{\rho})}
				      +
				      \|
				      u_L(r, \psi)
				      \|_{L^4(\nu_\rho)}
				      \big)
				      \left\|
				      u_L(r, \varphi- \psi)
				      \right\|_{L^4(\nu_\rho)}\,,
			      \end{aligned}
		      \end{equation}
		      and then, using \eqref{eq_u_phi_estimate},  that for $y\in[x,x+1)$
		      \begin{equation}\label{eq_l2_diff_u}
			      \begin{aligned}
				      \| u_L(r,i_{\delta,d_{x,r}}-i_{\delta,d_{y,r}})
				      \|_{L^{4}( \nu_{\rho} )}
				      \lesssim_{\rho}
				      \frac{1}{ \alpha^{3/4} \sqrt{L}}
				      \|i_{\delta,d_{x,r}} ( \cdot / L ) -
				      i_{\delta,d_{y,r}}( \cdot/L) \|_{\ell^{2}}
				      \lesssim_{\rho}
				      \frac{1}{ \delta \alpha^{3/4} \sqrt{L}}
				      \,,
			      \end{aligned}
		      \end{equation}
		     since  the two sampled indicators differ at no more than two lattice points.
	\end{enumerate}

	Next, we perform three approximation steps on the integrand.
	First, we replace the discrete gradient by showing that
	\begin{equation*}
		\begin{aligned}
			A_{L}^{(1)} (r, \varphi) & :=
			\Big\|
			\frac{1}{L} \sum_{x\in\ZZ}
			u_{L}
			(r,i_{ \delta, d_{x,r}}; \eta)^2
			\big(
			\nabla_{L d_{x,r}}^{L} \varphi
			-
			\partial_{x} \varphi ( d_{x, r})
			\big)
			\Big\|_{L^{2}( \nu_{\rho} )}               \\
			                         & \lesssim_{\rho}
			\frac{1}{ \delta \alpha^{3/2}}
			\frac{1}{L} \sum_{x\in\ZZ}
			\big| 	\nabla_{L d_{x,r}}^{L} \varphi
			-
			\partial_{x} \varphi ( d_{x, r}) \big|
			\lesssim_{\rho, \varphi} \frac{1}{ \delta \alpha^{3/2} L} \,,
		\end{aligned}
	\end{equation*}
	where we first used \eqref{eq_u_fourth_moment} and then
	\eqref{eq_grad_approx} with the fact that $ \varphi \in \mS ( \RR) $.
	Likewise, we bound
	\begin{equation*}
		\begin{aligned}
			A_{L}^{(2)} (r, \varphi) & :=	\bigg\|
			\frac{1}{L} \sum_{x\in\ZZ}
			u_{L}
			(r,i_{ \delta, d_{x,r}}; \eta)^2
			\int_{x}^{x+1}
			\big(
			\partial_{x} \varphi ( d_{x, r})
			-
			\partial_{x} \varphi (d_{y,r})\big)
			\ud y
			\bigg\|_{L^{2}( \nu_{\rho})}          \\
			                         &
			\lesssim_{\rho} \frac{1}{ \delta \alpha^{3/2}}
			\frac{1}{L} \sum_{x \in \ZZ}
			\sup_{y \in [x, x+1]} \big|
			\partial_{x} \varphi ( d_{x, r})
			-
			\partial_{x} \varphi (d_{y,r})\big|
			\lesssim_{\rho, \varphi}
			\frac{1}{ \delta \alpha^{3/2}L}\,,
		\end{aligned}
	\end{equation*}
	 where the last inequality holds by the mean-value theorem because $ \varphi \in \mS ( \RR ) $ and  Lemma~\ref{lem_usefulestimates_phi}.
	Lastly, using \eqref{eq_diff_u_L}
	\begin{equation*}
		\begin{aligned}
			A_{L}^{(3)} ( r,\varphi) & :=
			\bigg\|
			\frac{1}{L} \sum_{x\in\ZZ}
			\int_{x}^{x+1}
			\big(
			u_{L}(r,i_{ \delta, d_{x,r}}; \eta)^2
			-
			u_{L}(r, i_{\delta, d_{y,r}}; \eta )^{2}\big)
			\partial_{x} \varphi (d_{y,r})
			\ud y
			\bigg\|_{L^{2}( \nu_{\rho})}               \\
			                         & \leqslant
			\frac{1}{L} \sum_{x\in\ZZ}
			\int_{x}^{x+1}
			\big\|
			u_{L}(r,i_{ \delta, d_{x,r}}; \eta)^2
			-
			u_{L}(r, i_{\delta, d_{y,r}}; \eta )^{2}\big\|_{L^{ 2}( \nu_{\rho})}
			|\partial_{x} \varphi (d_{y,r})|
			\ud y                                      \\
			                         & \lesssim_{\rho}
			\frac{1}{ \delta^{3/2} \alpha^{3/2}}
			\frac{1}{L^{3/2}}
			\sum_{x\in\ZZ}
			\sup_{y \in [x, x+1]}
			|\partial_{x} \varphi (d_{y,r})|
			\lesssim_{\rho, \varphi} \frac{1}{ \delta^{3/2}
				\alpha^{3/2} \sqrt{L} }
			\,.
		\end{aligned}
	\end{equation*}
	Combining the above estimates for $ A^{(i)}_{L}$, $i=1,2,3$, 
	jointly with the Cauchy--Schwarz inequality (for
	the time integral), yields \eqref{eq_replace_continuum_deriv} after performing the change of
	variable $ d_{y,r} \mapsto x$ which cancels the
	factor $ L^{-1}$ in front of the integrals above.
	% \begin{equation*}
	% 	\begin{aligned}
	% 		 & \mathbf E_{\rho,L}\bigg[
	% 		\bigg|
	% 		\int_0^t\bigg(
	% 		\frac{1}{L} \sum_{x\in\ZZ}
	% 		u_{L}
	% 		(r,i_{ \delta, d_{x,r}})^2
	% 		\nabla_{L d_{x,r}}^{L} \varphi
	% 		-
	% 		\frac{1}{L}
	% 		\int_{\RR}
	% 		u_{L}(r, i_{\delta, d_{y,r}} )^{2}
	% 		\partial_{x} \varphi (d_{y,r})
	% 		\ud y\bigg)
	% 		\ud r
	% 		\bigg|^2
	% 		\bigg]                      \\
	% 		 & \lesssim
	% 		t
	% 		\int_0^t
	% 		\sum_{i=1}^{3} A_{L}^{(i)}( r, \varphi)^{2} \ud r
	% 		\lesssim_{\rho, \varphi} t^{2} \frac{1}{
	% 			\delta^{3}
	% 			\alpha^{3}L} \,.
	% 	\end{aligned}
	% \end{equation*}
	% This establishes \eqref{eq_replace_continuum_deriv} after performing a change of
	% variables: replace $ d_{y,r}$ with $ y$ which will cancel the
	% factor $ L^{-1}$ in front of the second integral.
	\par\smallskip\smallskip

	\textbf{Step 2.} \emph{Mollification and weak convergence.}
	Next, we would like to use $ u_{L_{n}} \Rightarrow u $ to
	conclude weak convergence of the integral, however, $i_{\delta , y}$ is not a
	Schwartz function.
	Instead, we consider compactly supported smooth
	approximations $i_{\delta,0}^{\kappa}$ of $i_{\delta,0}$ through mollification with a compactly supported function at scale $ \kappa $ (extended to $
		i_{\delta, x}^{\kappa}$ by shifting), such that
	\begin{equation*}
		\begin{aligned}
			\lim_{ \kappa \to 0}   \| i_{\delta, 0}- i_{\delta, 0}^{\kappa}\|
			_{L^{2}( \RR )}
			=0\,.
		\end{aligned}
	\end{equation*}
	Again, we introduce several auxiliary terms to proceed:
	First, we show that we can replace $i_{\delta, x}$ by its mollified
	version, because
	\begin{equation}\label{eq_supp_A4}
		\begin{aligned}
			A_{L, \kappa }^{(4)} ( r, \varphi) & :=
			\bigg\|
			\int_{\RR}
			\big(
			u_{L}(r, i_{\delta, x} )^{2}
			-
			u_{L}(r, i_{\delta, x}^{\kappa} )^{2}
			\big)
			\partial_{x} \varphi (x)
			\ud x
			\bigg\|_{L^{2}( \nu_{\rho})}                            \\
			                                   & \lesssim_{\varphi}
			\sup_{x \in \RR} \big\|u_{L}(r, i_{\delta, x} )^{2}
			-
			u_{L}(r, i_{\delta, x}^{\kappa} )^{2}
			\big\|_{L^{2}( \nu_{\rho})}
			\lesssim_{ \rho,\varphi}
			\frac{1}{ \delta^{1/2} \alpha^{3/2}  }
			\Big(
			\|i_{\delta,0} -
			i_{\delta,0}^{\kappa} \|_{L^{2}( \RR ) }
			+
			\frac{1}{\delta\sqrt{L}} \Big)\,.
		\end{aligned}
	\end{equation}
	In the last step we used the analogue of \eqref{eq_diff_u_L}
	with \eqref{eq_l2_diff_u} replaced by
	\begin{equation*}
		\begin{aligned}
			\| u_L(r,i_{\delta,d_{x,r}}-i_{\delta,d_{x,r}}^{\kappa})
			\|_{L^{4}( \nu_{\rho} )}
			 & \lesssim_{\rho}
			\frac{1}{ \alpha^{3/4} \sqrt{L}}
			\|(i_{\delta,d_{x,r}} -
			i_{\delta,d_{x,r}}^{\kappa}) ( \cdot/L) \|_{\ell^{2}} \\
			% & \lesssim_{\rho}
			%\frac{1}{ \alpha^{3/4}}
			%\bigg(
			%\frac{1}{L}
			%\sum_{y \in \ZZ}
			%| ( i_{\delta,d_{x,r}} -
			%i_{\delta,d_{x,r}}^{\kappa}) ( y/L)|^{2}
			%\bigg)^{1/2}
			%\\
			 & \lesssim_{\rho }
			\frac{1}{ \alpha^{3/4} }\Big(
			\|i_{\delta,0} -
			i_{\delta,0}^{\kappa} \|_{L^{2}( \RR ) }
			+
			\frac{1}{\delta\sqrt{L}} \Big)
			\,,
		\end{aligned}
	\end{equation*}
	where we performed another Riemann sum approximation in the last
	inequality.

	%using with $h = i_{\delta,d_{x,r}} -
	%				i_{\delta,d_{x,r}}^{\kappa}$
	%		\begin{equation*}
	%			\begin{aligned}
	%				\frac{1}{L}
	%					\sum_{y \in \ZZ}
	%					\int_{y}^{y+1}
	%					| h( y/L)|^{2}
	%					\ud z
	%					\lesssim
	%					\int_{\RR} |h (y)|^{2} \ud y
	%					+
	%					\frac{1}{L}
	%					\sum_{y \in \ZZ}
	%					\int_{y}^{y+1}
	%					| h( y/L) -h(z/L)|^{2}
	%					\ud z
	%					 \leqslant \| h \|_{L^{2}( \RR)}^{2}
	%					+
	%					\frac{1}{L}
	%					\delta^{-2}
	%			\end{aligned}
	%		\end{equation*}
	%		where the last inequality is a consequence of a total variation bound.

	Finally, the continuous mapping theorem applies such that $(u_{L},\mZ^{L}) \Rightarrow (u, \mZ)$ yields
	\begin{equation}\label{eq_conv_u_Ln_joint}
		\begin{aligned}
			\left(
			u_{L},\mZ^{L},
			\int_0^t\int_{\RR}
			u_{L}(r,i_{\delta,x}^{\kappa})^2
			\partial_x\varphi(x)\ud x\ud r
			\right)
			\Rightarrow
			\left(
			u,\mZ,
			\int_0^t\int_{\RR}
			u_r(i_{\delta,x}^{\kappa})^2
			\partial_x\varphi(x)\ud x\ud r
			\right)\,,
		\end{aligned}
	\end{equation}
	as $L \to \infty$, for every $ \kappa>0$.
	Lastly, we have to undo the mollification for the limiting field $ u$
	\begin{equation}\label{eq_supp_A5}
		\begin{aligned}
			A_{ \kappa }^{(5)} ( r, \varphi) & :=
			\bigg\|
			\int_{\RR}
			\Big(
				u_r(i_{\delta,x}^{\kappa})^2
			-
				u_r(i_{\delta,x})^2
			\Big)
			\partial_x\varphi(x)\ud x
			\bigg\|_{L^{2}( \PP)}                                     \\
			                                   & \lesssim_{ \varphi}
			\sup_{x \in \RR} \big\|u_r( i_{\delta,
					x}^{\kappa} )^{2}
			-
			u_r( i_{\delta, x} )^{2}
			\big\|_{L^{2}( \PP)}                                      \\
			                                   &
			\lesssim_{\varphi} \chi( \rho)
			\delta^{-1/2} \| i_{\delta, 0}^{\kappa} - i_{\delta , 0} \|_{L^{2}( \RR)}
			\lesssim_{\rho, \varphi} \alpha^{-1} \delta^{-1/2} \| i_{\delta, 0}^{\kappa} - i_{\delta , 0} \|_{L^{2}( \RR)}
			\,,
		\end{aligned}
	\end{equation}
	where we used again the trick \eqref{eq_sq_diff_trick} and the fact that $ u_r$ has the
	law of a spatial white noise with variance $ \chi ( \rho)$ (which allows us to use
	hypercontractivity).
	\par\smallskip\smallskip

	\textbf{Step 3.} \emph{Summary of approximations.}
	To conclude the statement of the lemma, we recall the identity
	\eqref{eq_blockboxN_equiv}.
	All approximation steps above are shown in a $ L^{2}$-sense:
	First, we replace the Riemann sum with
	the integral \eqref{eq_replace_continuum_deriv} and then combine \eqref{eq_supp_A4} and \eqref{eq_supp_A5}.
	Hence, \eqref{eq_conv_u_Ln_joint} implies that $\mN_{t}^{L,\delta L }(\varphi)$ converges to
	$\mN_{t}^{\delta}(\varphi)$ when first taking the limit $ L \to \infty$ and then
	$ \kappa \to 0 $.
\end{proof}

\section{The weakly condensing regime}
\label{sec_cond}

Having considered the inclusion process with fixed diffusivity $ \alpha>0$, we now turn to the weak condensation regime. 
To this end, we fix $ \rho>0$ and let $ \alpha_{L} \to 0 $ with $ \alpha_{L} L \to \infty$, such that 
\begin{equation*}
	\chi_L(\rho):=\rho+\frac{\rho^2}{\alpha_L}\to\infty\,.
\end{equation*}
In the symmetric setting, we prove Gaussian fluctuations throughout this
regime (Theorem~\ref{thm_main_cond}), which is the maximal regime in which
such a result is expected to hold. In Section~\ref{sec_asym_cond}, we also
discuss weak asymmetry and propose a candidate limiting SPDE.

\subsection{Symmetric weakly condensing fluctuations}\label{sec_sym}

In this section we prove Theorem~\ref{thm_main_cond}. Namely, we prove convergence of the density fluctuation field
for the symmetric
inclusion process in the weakly condensing regime
to the
Ornstein-Uhlenbeck SPDE
\begin{equation}\label{eq_ou_spde}
	\begin{aligned}
		\partial_{t}v = \frac{1}{2} \partial_{x}^{2} v + \partial_{x} \xi\,.
	\end{aligned}
\end{equation}
The proof follows the martingale-problem strategy of the equilibrium
fluctuation theory for reversible conservative systems, see for example
\cite[Chapter~11]{KipnisLandim99}.
Because the compressibility diverges, the classical reversible argument does
not apply verbatim. Instead each estimate must be revisited with explicit control of
its dependence on $\alpha_L$.

To avoid repeating estimates already proved for the fixed-$\alpha$ field, we
will repeatedly compare $v_L$ with the density fluctuation field from the
non-condensing regime.  In the symmetric case $\beta=0$, the only changes are
the time acceleration and the variance normalisation:
\begin{equation*}
	v_L(t,\varphi)
	=
	\frac{1}{\sqrt{\chi_L(\rho)}}
	u_L(t/\alpha_L,\varphi)\big|_{\beta=0}.
\end{equation*}
Consequently, several tightness and martingale estimates from the previous
sections can be reused, provided we
have suitable control of divergences that are caused by $ \lim_{L \to \infty}\alpha_{L}=0$.

\begin{proof}[Proof of Theorem~\ref{thm_main_cond}]
	We prove that $ ( v_{L})_{L}$ is tight and that any limit point $v$ is the solution of
	the martingale problem associated to \eqref{eq_ou_spde}.
	First, for every $L$ by Dynkin's formula
	\eqref{eq_dynkin_mart}
	\begin{equation*}
		v_L(t,\varphi)
		=
		v_L(0,\varphi)
		+
		\int_0^{t / \alpha_{L}} L^{2}\mS  v_L( \alpha_{L} s,\varphi)\,\ud s
		+
		\widetilde M_t^L(\varphi),
	\end{equation*}
	where $\widetilde M^L(\varphi)$ is a martingale.
	As for the fluctuation field, we have
	\begin{equation}\label{eq_tildeM_id}
		\begin{aligned}
			\widetilde M_t^L(\varphi) =
			\frac{1}{ \sqrt{ \chi_{L} ( \rho)}} M_{t/ \alpha_{L}}^{L} ( \varphi)
			\Big|_{ \beta =0} \,.
		\end{aligned}
	\end{equation}

	\textbf{Step 1.} \emph{Stationary white noise law.} We establish that $ v_{L} (t, \cdot)$ converges in law to a unit white noise for every $ t \geqslant 0$, provided $\alpha_{L} L \to \infty$.
	Indeed, we can follow the same steps as in the non--condensing regime.
	Tightness holds by \eqref{eq_cond_tight_fixedfield}, where the extra normalisation $
		\sqrt{\chi_{L} ( \rho)}$ in $
		v_{L}$ precisely
	compensates the variance blow up of the density field $ u_{L}$.
	Moreover,
	$ \EE_{\rho}[| \eta_{0}- \rho|^4] \lesssim \alpha_{L}^{-3}$.
	Hence, the estimate in
	\eqref{eq_supp2_tightness_field} with the replacement $ \delta
		\mapsto \delta \sqrt
		{\chi_{L}} $ and the extra prefactor $
		\chi_{L}^{-1}\sim \alpha_{L} $ yields
	\begin{equation*}
		\begin{aligned}
			\frac{1}{ \chi_{L} L}
			\sum_{x\in\ZZ}
			\varphi_{t,x}^{2}
			\EE_{\rho}\big[( \eta_{x}- \rho)^2
				\mathbf 1_{\{ |\varphi_{t,x}|| \eta_{x}- \rho|> \delta \sqrt{ \chi_{L}\,  L}\}}\big]
			\lesssim_{\rho,\varphi,\delta}
			\frac{1}{ \alpha_{L}  L}\,.
		\end{aligned}
	\end{equation*}

	\textbf{Step 2.} \emph{Symmetric drift.}
	By \eqref{eq_supp_sym}, we have $ L^{2}\mS v_{L}( t, \varphi) = \frac{\alpha_{L}}{2} v_{L}( t,
		\Delta^{L} \varphi)$.
	We define 
	\begin{equation*}
	\begin{aligned}
	\widetilde{\mathcal D}_t^L(\varphi)
			:=
			\int_{0}^{ t/ \alpha_{L}} \frac{\alpha_{L}}{2} v_{L}(
			\alpha_{L} s,
			\Delta^{L} \varphi) \ud s
			=\chi_L^{-1/2}\mathcal D_{t/\alpha_L}^L(\varphi)\,,
	\end{aligned}
	\end{equation*}
	where $ \mD^{L}$ was introduced in \eqref{eq_def_mD}.	
	Then, following the proof of Lemma~\ref{lem_blue_identification_symmetric_drift}, we see that
	\begin{equation*}
		\begin{aligned}
			\big(\widetilde{\mathcal D}_t^L(\varphi)\big)_{t \in [0,T]}
			\Rightarrow \bigg(\frac{1}{2}
			\int_0^t v_r(\partial_x^2\varphi)\ud r\bigg)_{t \in [0,T]}
			\,,
		\end{aligned}
	\end{equation*}
	by taking into account the extra variance
	normalisation $ \chi_{L}$.
	Indeed, 
	\eqref{eq_tight_sym} yields
	\begin{equation*}
		\mathbf E_{\rho,L}
		\big[
			|\widetilde{\mathcal D}_t^L(\varphi)
			-\widetilde{\mathcal D}_s^L(\varphi)|^2
			\big]
		\lesssim_{\rho,\varphi}
		\chi_L^{-1}\alpha_L
		\left(\frac{t-s}{\alpha_L}\right)^2
		\lesssim_{\rho,\varphi}
		(t-s)^2\,,
	\end{equation*}
	such that the tightness estimates of Section~\ref{sec_tight}
	remain uniform in the rescaled setting.

	\textbf{Step 3.} \emph{Martingale identification.}
	First, tightness of $ (\widetilde M^{L} ( \varphi) )_{L}$ follows again from tightness of
	the quadratic variation \cite[Theorem VI.4.13]{JS}:
	\eqref{eq_tildeM_id} together with
	Lemma~\ref{lem_quant_qv} implies
	\begin{equation}\label{eq_qv_quant_cond}
		\begin{aligned}
			\mathbf E_{\rho,L}
			\Big[
			\sup_{t\leq T}
			\Big|
			\langle \widetilde M^L(\varphi)\rangle_t
			-
			t\,  \| \partial_{x} \varphi\|_{L^{2} ( \RR )}^{2}
			\Big|^2
			\Big]
			 & \lesssim_{\rho, \varphi}
			\chi_{L}^{-2}
			\Big(
			\frac{T}{ \alpha^{5}_{L} L^{3}} + \frac{T^{ 2}}{\alpha^{3}_{L} L}
			+
			\frac{T^{2}}{\alpha_{L}^{2} } o_{L}(1)\Big) \\
			 & \lesssim_{\rho, \varphi}
			\frac{T}{ \alpha_{L}^{3}L^{3}} + \frac{T^{2}}{\alpha_{L} L} + T^{2} o_{L}(1)\,,
		\end{aligned}
	\end{equation}
	where in the last step we used again that $ \chi_{L} \sim \alpha^{-1}_{L}$. Notice also that \eqref{eq_vanish_jumps} still yields vanishing jumps,
	hence, any limit point of $ (\widetilde M^{L} ( \varphi) )_{L}$ has continuous
	trajectories.

	It is only left to show that $ \widetilde M ( \varphi) $ is indeed a martingale and has
	quadratic variation $ \langle \widetilde M ( \varphi) \rangle_{t} = t \|
		\partial_{x} \varphi\|_{L^{2} ( \RR )}^{2}$. To this end we follow the same steps as in
	the proof of Lemma~\ref{lem_control}: First, we notice that
	\begin{equation*}
		\begin{aligned}
			\mathbf E
			\big[
				(\widetilde M_t(\varphi)-\widetilde M_s(\varphi))F
				\big]
			=
			\lim_{L \to \infty}
			\mathbf E_{\rho,L}
			\big[(\widetilde M_t^L(\varphi)-\widetilde M_s^L(\varphi))F_L\big]
			=0\,,
		\end{aligned}
	\end{equation*}
	for any bounded continuous functional $ F$ of finitely many marginals of $(v_r, \widetilde
		M_{r})_{r \leqslant s}$, since by  \eqref{eq_fourth_mom_mart} and \eqref{eq_tildeM_id}
	\begin{equation*}
		\begin{aligned}
			&\mathbf E_{\rho,L}
			\big[ \sup_{t \in [0, T]}
				| \widetilde M_t^L(\varphi)|^4
				\big]\\
			&\lesssim \chi_{L}^{-1}
			L^{-3} \Big( T^{2}+	\frac{T}{ \alpha_{L}^{3} L^{3}} + \frac{T^{ 2}}{\alpha_{L} L}
			+ T^{2} o_{L}(1)	\Big)^{1/2}+
			\Big(T^{2}+
			\frac{T}{ \alpha^{3}_{L} L^{3}} + \frac{T^{ 2}}{\alpha_{L} L}
			+ T^{2} o_{L}(1)
			\Big)
			\lesssim_{\rho, \varphi, T} 1
			\,.
		\end{aligned}
	\end{equation*}
	The same argument yields that $ \widetilde M^{2}_{t}( \varphi) - t \| \partial
		_{x} \varphi\|_{L^{2}( \RR )}^{2}$ is a martingale, 
		which identifies the quadratic variation.
		\par\smallskip\smallskip

	\textbf{Step 4.} \emph{Identification of the limit.}
	Combining Steps~1-3 above establishes tightness of  $ ( v_{L} ( \cdot, \varphi), \widetilde M^{L}(
		\varphi) )_{L \in \NN }$ in $ D
		([0,T] , \RR \times \RR ) $ with limit points in $ C( [0,T] , \RR \times \RR) $.
	Thus, by Mitoma's theorem, the fields $ (v_{L} )_{L}$ are tight in $ D ([0,T] , \mS ' (
		\RR))$. Moreover, we established that any weak limit point $ v$ satisfies
	\begin{equation*}
		\begin{aligned}
			v_t(\varphi) = v_0(\varphi) + \frac{1}{2} \int_{0}^{t} v_r(
			\partial_{x}^{2} \varphi) \ud r+ \widetilde M_{t}( \varphi) \,, \quad \forall
			\varphi \in \mS ( \RR) \,,
		\end{aligned}
	\end{equation*}
	where $ \widetilde M ( \varphi) $ is a martingale.
	This uniquely identifies $ v$ as the generalized Ornstein--Uhlenbeck process
	solving \eqref{eq_ou_spde}, see for example
	\cite[Section~11.2]{KipnisLandim99}.
			\par\smallskip\smallskip

	Lastly, the proof remains true in the setting $ \lim_{L \to \infty} \alpha_{L} >0$, in which case the necessary bounds of Sections~\ref{sec_tight} and~\ref{sec_energy_solutions} hold verbatim up to an $ \alpha$- and $ \rho$-dependent constant.
\end{proof}

\begin{remark}\label{rem_cond_improvement}
	The above proof
	uses the one-block replacement estimate in Lemma~\ref{lem_quant_qv}
	which underlies the second-order Boltzmann--Gibbs principle. This yields the improved
	error \eqref{eq_qv_quant_cond},
	which vanishes in the full weakly condensing regime $\alpha_LL\to\infty$.
	On the other hand, applying a  na\"ive estimate to  the bracket representation
	\eqref{eq_prelimit_bracket} using  stationarity and the Cauchy--Schwarz inequality is not
	sufficient to cover the full regime. Indeed, using the
	bound $\var_\rho(a_x^L)\lesssim_\rho\alpha_L^{-2}$ and then applying the rescaling
	\eqref{eq_tildeM_id}, would only give an error bound
	\begin{equation*}
	\begin{aligned}
	\mathbf E_{\rho,L}\left[
		\sup_{t\leq T}
		\left|
		\langle\widetilde M^L(\varphi)\rangle_t
		-
		t \,\| \partial_{x} \varphi\|_{L^{2} ( \RR )}^{2}
		%\mathbf E_{\rho,L}
		%[\langle\widetilde M^L(\varphi)\rangle_t]
		\right|^2
		\right]
		\lesssim_{\rho,\varphi,T}
		\chi_L^{-2}
		\left(\frac{T}{\alpha_L}\right)^2
		\frac{\alpha_L^{-2}}{L}
		+o_L(1)
		\sim_\rho
		\frac{1}{\alpha_L^2L}
		+o_L(1)\,,
	\end{aligned}
	\end{equation*}
	where $o_L(1)$ is a deterministic Riemann-sum error in the expected bracket.
	Hence, this direct estimate would prove convergence of the bracket
	only under the stronger condition $\alpha_L^2L\to\infty$.
\end{remark}

\subsection{Discussion of a weakly asymmetric regime}\label{sec_asym_cond}

In view of the quantitative estimates derived throughout the non--condensing
regime, it is natural to ask whether weakly asymmetric fluctuations can also be
derived in the weakly condensing regime.  
% While the scaling points to a
% non-trivial Burgers-type limit, our bounds from the fixed-$\alpha$
% case are not strong enough after the variance normalisation and time acceleration.
To investigate this, we define
\begin{equation*}
	\begin{aligned}
		w_{L}( t, \varphi):=
		\frac{1}{ \sqrt{\chi_{L} L }}
		\sum_{x \in \ZZ} \varphi \big( L^{-1}x -
		L^{1/2}
		j' ( \rho) t / \alpha_{L}   \big) \big(
		\eta_{x}(t/ \alpha_{L}) - \rho\big)
		=
		\chi_{L}^{-1/2} u_{L}( t/ \alpha_{L} , \varphi) \,,
	\end{aligned}
\end{equation*}
which now takes both the shifted time frame and the variance blowup of \eqref{eq_def_u} and
\eqref{eq_cond_density_field} into account. Formally, we expect $ w_{L}$ to approximately solve the effective equation
\begin{equation*}
	\begin{aligned}
		\partial_{t} w_{L} = \frac{1}{ 2}  \partial_{x}^{2} w_{L} - \frac{\beta \sqrt{
				\chi_{L}} }{\alpha_{L}  } \partial_{x} ( w_{L}^{2}) + \partial_{x} \xi\,.
	\end{aligned}
\end{equation*}
Thus, in order to see a non-vanishing nonlinearity, it is necessary to choose a faster
vanishing asymmetry, namely $ \beta=
	\beta_{L} = \hat{\beta} \alpha_{L}^{3/2}$, for some $ \hat{\beta} >0$. This yields (at least formally) the following candidate SPDE for the scaling limit $w$:
\begin{equation*}
	\begin{aligned}
		\partial_{t} w = \frac{1}{2} \partial_{x}^{2} w - \rho \hat{\beta} \,
		\partial_{x} (w^{2}) + \partial_{x} \xi\,.
	\end{aligned}
\end{equation*}
Notice that the $L$-dependent choice $ \beta_{L}$ will also affect the transport shift
in the definition of $ w$ since it depends on $ j ' ( \rho)$.

However, our fixed-$\alpha$ bounds from the asymmetric case do not extend uniformly to $ \alpha_{L} \to 0$. After normalisation and time acceleration, the block and remainder estimates impose incompatible requirements on the block scale. Hence, using the present bounds, we are unable to establish tightness or identify the nonlinear drift in the weakly condensing regime. We leave the asymmetric limit in the weakly condensing regime for future work.

\appendix
\addtocontents{toc}{\protect\setcounter{tocdepth}{1}}

\section{Infinite-volume asymmetric simple inclusion process}\label{app_asip}

As discussed in the introduction, a general infinite-volume well-posedness statement for the
asymmetric inclusion process is delicate.  The formal generator is
\begin{equation*}
	\begin{aligned}
		\gen f(\eta)
		 & :=
		\sum_{x,y\in\ZZ}
		c(y-x)\eta_x(\alpha+\eta_y)
		\bigl[f(\eta^{x,y})-f(\eta)\bigr],
	\end{aligned}
\end{equation*}
where $c$ is the asymmetric nearest-neighbour kernel, $c(1)=p$,
$c(-1)=q$, $p+q=1$, and $c(z)=0$ otherwise.
The associated carr\'e du champ operator is $\Gamma f:=\gen(f^2)-2f\,\gen f $.

%Our equilibrium fluctuation arguments are carried out for the canonical
%infinite-volume stationary process with initial law $\nu_\rho$ obtained
%as the limit of stationary periodic finite-volume systems. 
Below we construct a candidate infinite-volume process by periodic approximation.
This construction
 provides both the necessary forward $\gen$-martingale problem and the adjoint
martingale problem for its time reversal. While we do not assert uniqueness
among all infinite-volume martingale solutions or pathwise well-posedness
from arbitrary deterministic initial configurations, this canonical process
is sufficient for the scope of the paper.

\subsection{Existence by periodic approximations}

We approximate the infinite volume inclusion dynamics by periodic dynamics on discrete
tori $\TT_K:= \{-K , \cdots, K-1,K \} \subset \ZZ$. To this end, denote by $\nu_\rho^{\TT_K}$
the product measure on $ \NN_{0}^{\TT_K}$
and define
\begin{equation*}
	\begin{aligned}
		\gen^{(K)} f ( \eta)
		:=
		\sum_{x \in \TT_{K}, y \in \ZZ}
		c (y-x) \eta_{x}( \alpha + \eta_{y})
		\big[ f ( \eta^{x,y}) - f ( \eta) \big]\,,
	\end{aligned}
\end{equation*}
where $\eta_{y}$ in the expression is to be read as its projection onto $ \TT_{K}$ (while in $ c(y-x)$ its meaning is indeed $x+1$ or $x-1$ without projecting).
The total number of particles is $\nu_{\rho}^{\TT_K}$-almost surely finite and is conserved
by the dynamics. We denote the resulting continuous--time Markov chain on $\TT_K$ by
$(\eta^{(K)}(t))_{t\geqslant 0}$ and its law by $\mathbf P_{\rho}^{(K)}$. The measure
$\nu_{\rho}^{\TT_K}$ is invariant under $\gen^{(K)}$, see \cite[Theorem~2.1]{GRV11}, which
is why we use periodic boundary conditions for the asymmetric dynamics. By periodic
extension, we regard $\eta^{(K)}$ as a process with values in $\NN_0^{\ZZ}$.

\begin{proposition}[Unique periodic approximation]\label{prop_ip_exists}
	Let $ \rho>0 $.
	There exists a probability measure $ \mathbf P_{\rho}$ such that
	\begin{equation*}
		\begin{aligned}
			\mathbf P_{\rho}^{(K)} \Rightarrow
			\mathbf P_{\rho}\,, \quad \text{in } \
			D ([ 0, \infty) ; \NN_{0}^{\ZZ})\,.
		\end{aligned}
	\end{equation*}
	The coordinate process $ ( \eta (t))_{t \geqslant 0}$ associated to the limiting measure $
		\mathbf P_{\rho}$ is stationary with respect to $ \nu_{\rho}$, and
	for every local and bounded
	function $ f : \NN_{0}^{\ZZ} \to \RR $ the process
	\begin{equation*}
		\begin{aligned}
			t \mapsto
			f ( \eta (t)) - f ( \eta (0)) - \int_{0}^{t} \gen f ( \eta(s)) \ud
			s 
		\end{aligned}
	\end{equation*}
	is a martingale with predictable quadratic variation $ \int_{0}^{t}
		\Gamma f ( \eta (s))\ud s$.
	Moreover, for every $T>0$, the càdlàg time reversal of the limiting process
	on $[0,T]$ solves the martingale problem associated to $\gen^{\ast}$ on
	bounded local functions.
	% , with predictable quadratic variation determined
	% by $\Gamma^{\ast}f:=\gen^{\ast}(f^{2})-2f\gen^{\ast}f$.
\end{proposition}

We first collect uniform moment bounds and prove a Cauchy property of the approximating
sequence through an explicit coupling.

\begin{lemma}\label{lem_sup_moment}
	For every $ R \in \NN $, $x \in [-R, R] \cap \ZZ$, and $ r \geqslant 1 $, we have
	\begin{equation*}
		\begin{aligned}
			\sup_{ K \geqslant R }
			\mathbf E_{\rho}^{(K)}\big[
			\big|\sup_{ t \in [0, T]}  \eta
			_{x}^{(K)}(t) \big|^{r}\big]< \infty\,.
		\end{aligned}
	\end{equation*}
\end{lemma}

\begin{proof}
	Let $N_x^{(K)}(t)$ denote the
	number of jumps that occurred up to time $t$ and involved a change in occupation of $ x \in \ZZ$.
	$N_{x}^{(K)}$ is a counting process with rates given in terms of the jump
	intensity at $x$, namely
	\begin{equation*}
		\begin{aligned}
			R_{x}( \eta)
			:= \sum_{y \in \ZZ} \Big(c(y-x) \eta_{x} ( \alpha + \eta_{y}) + c(x-y)
			\eta_{y}( \alpha + \eta_{x}) \Big)\,,
		\end{aligned}
	\end{equation*}
	which has uniform moment bounds under $ \nu_{\rho}^{\TT_K}$, cf. \eqref{eq_eta_mean_var}.
	Therefore, by Dynkin's formula
	\begin{equation*}
		\begin{aligned}
			\mathbf E_{\rho}^{(K)}[ |N_{x}^{(K)} (t)|^{r} ]
			=
			\mathbf E_{\rho}^{(K)}\bigg[
				\int_{0}^{t}
				R_{x} ( \eta^{(K)}(s))
				\big(
				( N_{x}^{(K)} (s) +1 )^{r}
				- N_{x}^{(K)}(s)^{r}
				\big)
				\ud s
				\bigg]\,.
		\end{aligned}
	\end{equation*}
	Using stationarity with respect to $ \nu_{\rho}^{\TT_K}$ and the moment bounds of $
		R_{x}( \eta) $, the previous display implies
	\begin{equation*}
		\begin{aligned}
			\mathbf E_{\rho}^{(K)}[ |N_{x}^{(K)} (t)|^{r} ]
			\lesssim_{\alpha,p,q, r}
			\int_{0}^{t} 1 + \mathbf E_{\rho}^{(K)}[ |N_{x}^{(K)} (s)|^{r}]^{(r-1)/r}   \ud s\,,
		\end{aligned}
	\end{equation*}
	where we applied H\"older's inequality to the integrand and furthermore used $
		(N+1)^{r} - N^{r} \lesssim_{r} (1 + N^{r-1})$.
	Hence, by Gr\"onwall's lemma, $ \sup_{ K \geqslant R}  \mathbf E_{\rho}^{(K)}[ |N_{x}^{(K)}
		(T)|^{r} ] < \infty$.
	Lastly, we observe the crude bound
	\begin{equation*}
		\begin{aligned}
			\sup_{ t \in [0, T]}  \eta
			_{x}^{(K)}(t) \
			\leq \eta_{x}^{(K)} (0 ) + N^{(K)}_{x} (T)\,,
		\end{aligned}
	\end{equation*}
	which finishes the proof.
\end{proof}

\begin{lemma}[Finite volume coupling]\label{lem_cauchy_coordinate}
	Let $ \Lambda \subset \ZZ$ be a finite interval and $ T \in (0, \infty)$.
	For every $ \varepsilon \in (0,1)$
	there exists a constant $ R \in \NN$ such that for all $ M, K \geqslant R +1$
	we find a coupling $ \mathbf P_{\rho}^{(K,M)} $ between the finite volume
	approximations $ \mathbf P_{\rho}^{(K)} $ and $ \mathbf P_{\rho}^{(M)} $ defined above, such that
	\begin{equation*}
		\begin{aligned}
			\mathbf P_{\rho}^{(K,M)} \big(
			\eta^{(K)}_{x} (t) = \eta^{(M)}_{x}(t)\,,\ \forall x \in
			\Lambda\,, \ t \in [0,T]
			\big)
			\geqslant 1- \varepsilon\,.
		\end{aligned}
	\end{equation*}
\end{lemma}

The proof follows Harris' idea of constructing
finite-range interacting particle systems \cite{Harris72}.
After localisation in a large box, jump rates are bounded and
boundary discrepancies can reach an observed window only along
a long oriented path, whose probability is small.

\begin{proof}
	With a slight abuse of notation, let us write $ [-R, R ] $ instead of $ [ -R, R ]
		\cap \ZZ$ and take $ R$ to be large enough such that $ \Lambda \subset [ -R, R]$.

	To construct the coupling between $ \mathbf P_{\rho}^{(K)} $ and $
		\mathbf P_{\rho}^{(M)} $ (let us consider $ M \geqslant K > R$),
	we first couple the initial laws $ \nu_{\rho}^{\TT_K}$ and $ \nu_{\rho}^{\TT_M}$, such
	that they agree on the domain of $ \TT_{K} $.
	Next, define $ \tau_{R}$ to be the first time at which one of the processes
	accumulates at least $ R^{1/4}$ particles on a site in $[-R, R]$, i.e.
	\begin{equation*}
		\begin{aligned}
			\tau_{R} :=
			\inf \big\{ t \geqslant 0 \,: \, \exists
			\text{$x \in [-R, R]$, $ i \in
					\{K,M\}$ such that $ \eta_{x}^{(i)}(t)
					\geqslant R^{1/4} $ }\big\}\,.
		\end{aligned}
	\end{equation*}
	Consequently, for any $ t < \tau_{R}$ any jump within the interval $ [-R,
				R]$ has a rate bounded by
	\begin{equation*}
		\begin{aligned}
			\overline{c}_{R} := ( p + q ) R^{1/4} ( \alpha+ R^{1/4})
			\lesssim_{p,q,  \alpha} \sqrt{R} \,.
		\end{aligned}
	\end{equation*}
	Hence, up to time $ \tau_{ R}$ we may couple the two finite-volume Markov chains
	as follows: for every nearest-neighbour pair $(x,y)$ in $ [-R,R]$ we consider an exponential
	clock of rate $ \overline{c}_{R }$. Upon ringing at time $t$, the jump in the
	chain $ \eta^{(i)}$ is
	performed provided
	\begin{equation*}
		\begin{aligned}
			U \leqslant
			\frac{c(y-x)}{ \overline{c}_{R }} \eta^{(i)}_{x}( t-) ( \alpha + \eta
			^{(i)}_{y}(t-))\,, \quad i \in \{K,M \} \,, \quad x ,y \in [-R,R]\,,
		\end{aligned}
	\end{equation*}
	where $ U $ is a uniform random variable on $ [0,1] $, which is sampled independently
	whenever one of the clocks rings (but the same for $i=K,M$).
	In particular, whenever the local configurations at $ x$ and $y$ agree, the same decision is made in both chains.
	For jumps outside of $[-R, R] $ both chains may evolve independently.
	Thus, only jumps near the boundary $ \{-R, R\}$ can possibly lead to any
	disagreement between the processes inside $ \Lambda$.
	Moreover, any such disagreement must then propagate over a distance of length $ d_{R}:= \mathrm{dist}(
		\Lambda, \{ -R, R \}) $ to reach the interval $
		\Lambda$.

	On the event $ \{ \tau_{R} > T \}$, any disagreement from the boundary $\{-R, R \}$ can
	at most travel at rate $ \overline{c}_{R} $ across the distance $ d_{R}$ to lead to a
	disagreement within $ \Lambda$. Thus,
	the probability of disagreement propagation can be bounded by the probability that a totally asymmetric, 
	continuous time random walk (started at the origin) of rate $ 2\overline{c}_{R}$ reaches $ d_{R}$ before time $
		T$. In particular, as $ R \to \infty$ we have
	\begin{equation*}
		\begin{aligned}
			 & \sup_{K,M>R}\mathbf P_{\rho}^{(K,M)} ( \text{disagreement reaches $ \Lambda $ from $\{-R, R\}$}  ,
			\tau_{R}> T )
			\leqslant
			2
			\mathbf P \Big( N_{T} \geqslant d_{R} \Big)
			%2 \sum_{k = d_{R}}^{\infty} \frac{( 2 \overline{c}_{R} T )^{k}}{k !}
			%\lesssim
			%\frac{(2e \overline{c}_{R} T )^{d_{R}}}{d_{R}^{d_{R}}} e^{- 2
			%\overline{c}_{R} T}
			\to 0 \,,
		\end{aligned}
	\end{equation*}
	where $ N_{T}\sim \mathrm{Poisson}(2 \overline{c}_{R}T)$ denotes the number of random
	walk jumps performed up to time $T$.
	The limit indeed vanishes, since $ \sqrt{R} \sim \overline{c}_{R} \ll d_{R} \sim R $.

	On the other hand, to control the event $ \{ \tau_{R} \leqslant T \}$, we use
	\begin{equation*}
		\begin{aligned}
			\mathbf P_{\rho}^{(K,M)} ( \tau_{R} \leqslant T )
			\leqslant
			\sum_{i \in \{K,M\}}
			\sum_{x= - R}^{R} \mathbf P_{\rho}^{(i)} \big( \sup_{ t \in [0, T]}  \eta
			_{x}^{(i)}(t) \geqslant  R^{1/4}\big)
			\lesssim  R^{1-r/4}
			\sum_{i \in \{K,M\}}
			\mathbf E_{\rho}^{(i)}\big[
			\big|\sup_{ t \in [0, T]}  \eta
			_{0}^{(i)}(t) \big|^{r}\big]\,.
		\end{aligned}
	\end{equation*}
	Because for any $ r >4 $ the moment on the right-hand side is uniformly bounded in $i$
	(see Lemma~\ref{lem_sup_moment}),
	the above display yields $ \lim_{ R \to \infty}  \sup_{M,K>R}\mathbf P_{\rho}^{(K,M)} ( \tau_{R} \leqslant T )=0$.
\end{proof}

\begin{proof}[Proof of Proposition~\ref{prop_ip_exists}]
	For every finite $\Lambda \subset \ZZ$, Lemma~\ref{lem_cauchy_coordinate} implies tightness of the
	coordinate restricted processes $ ( \eta_{x}^{(K)} (
		t))_{t \geqslant 0, x \in \Lambda}$ in $ D ([ 0, \infty), \NN_{ 0}^{\Lambda} ) $.
	Because $ \NN_{0}^{\ZZ}$ is equipped with the product topology, this implies
	tightness of the probability measures $ ( \mathbf P_{\rho}^{(K)})$ in $ D ([ 0, \infty),
		\NN_{0}^{\ZZ}
		) $ by a diagonal argument.
	Moreover, Lemma~\ref{lem_cauchy_coordinate} uniquely identifies all finite--dimensional
	projections of any
	accumulation point, hence, the limit $ \mathbf P_{\rho}$ is determined uniquely
	\cite[Corollary~3.9.1]{EthierKurtz86}.
		Since each $\mathbf P_{\rho}^{(K)}$ is stationary and
$\mathbf P_{\rho}^{(K)}\Rightarrow\mathbf P_{\rho}$, also the limiting
process is stationary. Moreover, its one-time marginal is $\nu_{\rho}$,
because $\nu_{\rho}^{\TT_K}\Rightarrow \nu_{\rho}^{\otimes\ZZ} =  \nu_{\rho}$.

	It remains to verify that the limiting measure solves the martingale problem.  Let $f : \Omega \to \RR $ be a bounded local
	function and let $K$ be large enough such that $ \gen^{( K)} f = \gen f $.
	Moreover, we define the functional
	\begin{equation*}
		\begin{aligned}
			\eta= ( \eta(s))_{s \in [0,t]} \mapsto	M_t (f)[ \eta]
			:=
			f(\eta(t))-f(\eta(0))
			-
			\int_0^t\gen f(\eta(s)) \ud s\,.
		\end{aligned}
	\end{equation*}
	Thus, by Dynkin's formula $ M_{t}^{K}(f):= M_{t}(f)[ \eta^{(K)}]$
	defines a martingale for every $K$, with predictable quadratic variation of the form $ \int_{0}^{t}
		\Gamma f ( \eta^{(K)} (s))\ud s$.
	Hence, it is only left to show that $ M_{t} (f) [ \eta]$ 
	is a martingale under $ \mathbf P_{\rho}$.

	To this end, we fix $0 \leqslant s < t \leqslant T $ and a bounded measurable cylinder function $H$ which depends only on finitely many marginals of the form $ \eta_{x_i}( r_{i})$ with $r_{i} \in [0,s]$.
	Then, using weak convergence of $ \mathbf P_{\rho}^{(K)}$ to $ \mathbf P_{\rho}$ we claim
	\begin{equation}\label{eq_supp1_exIP}
		\begin{aligned}
			\mathbf E_{\rho}
				[( M_{t} (f) - M_{s}(f)) H( \eta) ]
			=
			\lim_{K \to \infty}
			\mathbf E_{\rho}^{(K)}
			[( M_{t}^{K} (f) - M_{s}^{K}(f)) H ( \eta^{(K)})]=0\,,
		\end{aligned}
	\end{equation}
	which can be extended to the whole natural filtration of $ \eta$ up to
	time $ s$ \cite[Appendix~Corollary~4.4]{EthierKurtz86}.
	Indeed, the first equality in \eqref{eq_supp1_exIP} is justified because $ f ,H$ are bounded  and
	\begin{equation*}
		\begin{aligned}
			|\gen f ( \eta ) |
			\lesssim_{p,q, \alpha ,f }
			1+
			\sum_{x \in \Lambda}
			\eta_{x}^{2} \,,
		\end{aligned}
	\end{equation*}
	belongs to $ L^{q}( \nu_{\rho})$, $ q\geqslant 1$, by \eqref{eq_eta_mean_var},
	where $ \Lambda \subset \ZZ $ denotes the finite subset that $ f$ and $H$ actually depend on (extended by the range of $ c$). Moreover, the path functionals induced by $f$ and $H$ are almost surely continuous under $\mathbf P_{\rho}$, and also the map $ \omega \mapsto \int_{0}^{t} \gen f ( \omega (s)) \ud s$  is continuous in $D([0,T],\mathbb N_0^{\mathbb Z}) $. This allows to exchange limit and expectation \eqref{eq_supp1_exIP} by weak convergence.
	The same argument can be applied to show that the functional $ t \mapsto  M_{t}(f) [
				\eta]^{2}- \int_{0}^{t} \Gamma f ( \eta(s)) \ud s$ is a martingale, thus, verifying the
	explicit form of the predictable quadratic variation.

	It is only left to verify the backward martingale problem.
	Fix $T>0$ and let $\overleftarrow{\eta}^{(K),T}$ and
	$\overleftarrow{\eta}^{T}$ denote the càdlàg reversals of
	 $\eta^{(K)}$ and $\eta$, respectively. 
	More precisely,
	 \begin{equation*}
		\overleftarrow{\eta}^{T}(t)
		:=
		\begin{cases}
			\eta((T-t)-), & 0\leqslant t<T, \\
			\eta(0),      & t=T,
		\end{cases}
	\end{equation*}
	and analogously for $\eta^{(K)}$.
	%Conditionally on the conserved total particle number, the finite-volume
	%process is a finite-state stationary Markov chain. 
	This reversed process
	 solves the martingale problem associated
	to $(\gen^{(K)})^{\ast}$, with carr\'e du champ
	$\Gamma^{(K),\ast}f
		:=(\gen^{(K)})^{\ast}(f^{2})-2f(\gen^{(K)})^{\ast}f$,
	with respect to its full reversed natural filtration.
	For every bounded local $f$, one has
	$(\gen^{(K)})^{\ast}f=\gen^{\ast}f$ and
	$\Gamma^{(K),\ast}f=\Gamma^{\ast}f$ for all sufficiently large $K$.
	The passage $K\to\infty$ for the reversed martingale and its compensated
	square is then identical to the forward argument above.
\end{proof}

We also record a useful bound which is a consequence of Proposition~\ref{prop_ip_exists} and Lemma~\ref{lem_sup_moment}.

\begin{corollary}\label{cor_growth}
	For every $T>0$, there exists an almost surely finite random variable $C_T$
	such that
	\begin{equation*}
	\begin{aligned}
	\sup_{t\in[0,T]}\eta_x(t)
		\leq C_T(1+|x|),
		\qquad \forall x\in\ZZ\,.
	\end{aligned}
	\end{equation*}
\end{corollary}

\begin{proof}
	By Proposition~\ref{prop_ip_exists}, Lemma~\ref{lem_sup_moment}, and the Portmanteau theorem
	we see that
	\begin{equation*}
	\begin{aligned}
	\sup_{x\in\ZZ}
		\mathbf E_{\rho}\Big[
			\sup_{t\in[0,T]}\eta_x(t)^2
			\Big]
		\leqslant
		\liminf_{K\to\infty}
		\mathbf E_\rho^{(K)}\Big[
			\sup_{t\leq T}\eta_{0}^{(K)}(t)^2
			\Big]
		\leq C_T'\,,
	\end{aligned}
	\end{equation*}
	for some deterministic $C_T'<\infty$, where spatial translation invariance
	was used to obtain uniformity in $x$. Hence, Markov's inequality yields
	\begin{equation*}
	\begin{aligned}
	\sum_{x\in\ZZ}
		\mathbf P_{\rho}\Big(
		\sup_{t\in[0,T]}\eta_x(t)>1+|x|
		\Big)
		\leq
		C_T'\sum_{x\in\ZZ}(1+|x|)^{-2}
		<\infty\,.
	\end{aligned}
	\end{equation*}
	The Borel--Cantelli lemma therefore shows that only finitely many $x$ violate
	the stated bound with $C_T=1$. Absorbing these finitely many exceptions into
	$C_T$ proves the claim.
\end{proof}

\subsection{A local Kipnis--Varadhan inequality}

The celebrated Kipnis--Varadhan inequality provides a general
estimate of fluctuations of additive functionals in terms of a
$H^{-1}$ bound.
Because its proof requires knowledge of the infinitesimal generator associated to $(
	\eta(t))_{t \geqslant 0} $ (which we are not addressing in this work), we will derive the
necessary bound from the finite--dimensional
approximations $ \gen^{(K)}$ instead.
The symmetric part of these approximations is given in terms of
\begin{equation*}
	\gens^{(K)}
	=
	\frac{1}{2}
	\big(\gen^{(K)}+(\gen^{(K)})^*\big)\,,
\end{equation*}
which agrees with \eqref{eq_def_sym_K}.\footnote{When applying Lemma~\ref{lem_local_kv_finite_volume} in
 	Section~\ref{sec_sg}, we pick up an extra factor $ L^{-2}$ due to the speed-up in
 	\eqref{eq_genIP}, i.e. $ \gen^{(K)}_{L}= L^{2}\gen^{(K)}$.}
We restrict ourselves to local functions $ F :[0,T] \times
	\Omega \to \RR$, by which we
mean that there exists a finite subset $ \Lambda \subset \ZZ$ such that
for every $ t \in [0,T]$ the function $ F ( t, \cdot) $ is local and
only depends on the configuration in $ \Lambda$.

\begin{lemma}[Kipnis--Varadhan input]
	\label{lem_local_kv_finite_volume}
	Let $ F \in C^{\infty}([0,T]; L^{2}( \nu_{\rho}))$
	be local
	such that for every~$K$ large enough
	\begin{equation*}
		\begin{aligned}
			\EE_{\rho}^{(K)}\bigg[ F ( t, \eta) \bigg| \sum_{x \in
					\TT_{K}} \eta_{x} \bigg] =0
			\,,
		\end{aligned}
	\end{equation*}
	for almost every $ t \in [0,T]$.
	Then, for the stationary infinite-volume process constructed in Proposition~\ref{prop_ip_exists},
	\begin{equation*}
		\mathbf E_{\rho}
		\left[
			\sup_{t \in [0,T]}
			\left|
			\int_0^t F(s,\eta(s))\ud s
			\right|^2
			\right]
		\leq
		C_{\mathrm{KV}}
		\liminf_{K\to\infty}
		\int_0^T
		\|F(s,\cdot)\|_{-1,\TT_K}^2\ud s\,,
	\end{equation*}
	with a universal constant $C_{\mathrm{KV}}$.
	The norm $ \| \cdot \|_{-1, \TT_K}$ was defined in \eqref{eq_def_hminus1_finitevol}.
\end{lemma}

\begin{proof}
	For every $K \in \NN $
	and every $N \in \NN_{0}$ let $\langle\cdot,\cdot\rangle_{\TT_K,N}$ denote the inner
	product under the canonical measure $\nu_\rho^{\TT_K}(\ud \eta\mid
		\sum_{x \in \TT_{K}} \eta_{x}=N)$, which is stationary with respect to $
		\gen^{(K)}$ and $ \gens^{(K)}$. Moreover, set
	\begin{equation*}
		\|h\|_{-1,\TT_K,N}^2
		:=
		\sup_g
		\big\{
		2\langle h,g\rangle_{\TT_K,N}
		-
		\langle g,-\gens^{(K)}g\rangle_{\TT_K,N}
		\big\},
	\end{equation*}
	where the supremum is over functions on the finite sector configuration
	space with $N$ particles.
	The
	time-dependent  Kipnis--Varadhan inequality for finite state spaces, obtained by the
	forward--backward martingale decomposition on each sector (see for
	example \cite[Lemma~4.3]{CLO01}), gives
	\begin{equation*}
		\begin{aligned}
			 & \mathbf E_{\rho}^{(K)}
			\left[
				\sup_{t\leq T}
				\left|
				\int_0^t F(s,\eta^{(K)}(s))\ud s
				\right|^2
				\,\middle|\, \sum_{x \in \TT_{K}} \eta_{x}(0)=N
				\right]
			\leq
			C_{\mathrm{KV}}
			\int_0^T
			\|F(s,\cdot)\|_{-1,\TT_K,N}^2
			\ud s \,.
		\end{aligned}
	\end{equation*}
	Averaging over the total number of particles $N$ yields
	\begin{equation}\label{eq_kv_supp2}
		\begin{aligned}
			 & \mathbf E_{\rho}^{(K)}
			\left[
				\sup_{t\leq T}
				\left|
				\int_0^t F(s,\eta^{(K)}(s))\ud s
				\right|^2
				\right]
			\leq
			C_{\mathrm{KV}}
			\int_0^T
			\| F ( s, \cdot) \|_{-1, \TT_K}^{2}
			\ud s \,.
		\end{aligned}
	\end{equation}
	Here we used that for every $ h \in L^{2}( \nu_{\rho }^{\TT_K})$ such that 
	 $ \EE_{\rho}^{(K)}\big[ h \big| \sum_{x \in \TT_{K}} \eta_{x}\big]=0$, we have
	\begin{equation}\label{eq_kv_supp3}
		\begin{aligned}
			\EE_{\rho}^{(K)}
			\big[
			\|h\|_{-1,\TT_K, \sum_{x \in \TT_{K}} \eta
			_{x}}^2
			\big]
			= \| h\|_{-1, \TT_K}^{2}\,.
		\end{aligned}
	\end{equation}
	To see \eqref{eq_kv_supp3}, let $ g_{N}$ be such that it attains the
	supremum in the variational expression of $ \| h \|_{-1,\TT_K,N}$.
	For every $M \in \NN$, define the global test function
	\begin{equation*}
	\begin{aligned}
	g^{(M)}:= 
	\begin{cases}
		g_{N}\,, & \text{on } \big\{ \sum_{x \in \TT_{K}} \eta_{x}=N\big\} \text{ whenever } N \leqslant M\,,\\
		0 \,, & \text{otherwise}\,,
	\end{cases}
	\end{aligned}
	\end{equation*}
	which is admissible as an argument in \eqref{eq_def_hminus1_finitevol}. 
	Next, we notice
	\begin{equation*}
	\begin{aligned}
	\|h\|_{-1,\TT_K}^2
		&\geq
		2\langle h,g^{(M)}\rangle_{\rho,\TT_K}
	-
	\langle g^{(M)},-\gens^{(K)}g^{(M)}\rangle_{\rho,\TT_K}
		\\
		&=
		\sum_{N=0}^M
		\nu_\rho^{\TT_K}\Big( \sum_{x \in \TT_{K}} \eta_{x}=N\Big)
		\|h\|_{-1,\TT_K,N}^2
		\to \EE_{\rho}^{(K)}
		\big[
		\|h\|_{-1,\TT_K, \sum_{x \in \TT_{K}} \eta
		_{x}}^2
		\big]\,,
	\end{aligned}
	\end{equation*}
	where in the last step we took the limit $ M \to \infty$.
Conversely, writing 
$ g_{N}$ for the restriction to $\big\{ \sum_{x \in \TT_{K}} \eta_{x}=N\big\}$ of any admissible function $g$ from \eqref{eq_def_hminus1_finitevol},  we obtain
\begin{equation*}
\begin{aligned}
 2\langle h,g\rangle_{\rho,\TT_K}
-\langle g,-\gens^{(K)}g\rangle_{\rho,\TT_K}
&=
\sum_{N=0}^{\infty}
\nu_\rho^{\TT_K}\Big(\sum_{x \in \TT_K}\eta_x=N\Big)
\big(
2\langle h,g_{N}\rangle_{\TT_K,N}
-\langle g_{N},-\gens^{(K)}g_{N}\rangle_{\TT_K,N}\big)
\\
&\leq
\sum_{N=0}^{\infty}
\nu_\rho^{\TT_K}\Big(\sum_{x \in \TT_K}\eta_x=N\Big)
\|h\|_{-1,\TT_K,N}^{2}
=
\EE_{\rho}^{(K)}
		\big[
		\|h\|_{-1,\TT_K, \sum_{x \in \TT_{K}} \eta
		_{x}}^2
		\big]
\,.
\end{aligned}
\end{equation*}
Taking the supremum over $g$ proves the reverse inequality, which concludes \eqref{eq_kv_supp3}.

To complete the proof, we take $ \liminf_{K \to \infty}$ on both sides of \eqref{eq_kv_supp2} and apply the Portmanteau theorem to the left-hand side, jointly with Proposition~\ref{prop_ip_exists}. Indeed, it can be checked that the functional
	\begin{equation*}
	D([0,T];\NN_0^\ZZ)\ni \omega \mapsto
		\sup_{t\leq T}
		\left|
		\int_0^t F(s,\omega(s))\ud s
		\right|^2
	\end{equation*}
	is continuous, which allows the application of Portmanteau.
\end{proof}

\subsection{Non-local extension of Dynkin's formula}\label{sec_dynkin}

The extension of Dynkin's formula to non-local and polynomial observables only needs a $\nu_{\rho}$-stationary
process satisfying the martingale problem for \emph{bounded local} test functions, which we constructed in Proposition~\ref{prop_ip_exists}.
% The first step is a simple reduction. Using integration by parts, the time-dependent Dynkin formula follows from the
% time-independent one for functions of the form $F(t,\eta)=g(t)f(\eta)$, with
% $g\in C^1([0,T])$ and $f$ bounded and local.
% By linearity, the same formula holds for finite sums of such functions.
To this end, let $P$ be an arbitrary polynomial, $v\in C^1([0,T];\ell^1(\ZZ))$, and set
$P_R(n):=\mathds{1}_{\{n\le R\}}P(n)$.  We define the bounded local functions 
\begin{equation}\label{eq_form_F_dynkin_ext}
	F_R(t,\eta)
	:=
	\sum_{|x|\le R}
	v_{t,x}\,
	P_R(\eta_x)
	\qquad \text{approximating}\qquad
	F(t,\eta)
	:=
	\sum_{x\in\ZZ} v_{t,x}P(\eta_x)\,.
\end{equation}

The aim is to pass from $F_R$ to $F$ in Dynkin's formula
\eqref{eq_dynkin_mart}.  
We start by establishing $L^2(\nu_\rho)$-convergence of the fields themselves. 
First, we notice that for every
$t\in[0,T]$,
\begin{equation*}
	\begin{aligned}
		\big\| F_R(t,\eta)-F(t,\eta)
			\big\|_{L^{2}( \nu_{\rho} )}
		 & \leqslant
		\sum_{|x|\le R}|v_{t,x}|
		\big\|\mathds{1}_{ \eta_{0} \leqslant R }
		P(\eta_0)-P(\eta_0)\big\|_{L^{2}( \nu_{\rho} )}
		+
		\sum_{|x|>R}|v_{t,x}|
		\big\|P(\eta_0)\big\|_{L^{2}( \nu_{\rho} )}.
	\end{aligned}
\end{equation*}
The first term vanishes because $ \lim_{R \to \infty} \mathds{1}_{ \cdot > R}P ( \cdot)=0$  in $L^2(\nu_\rho)$, and the
second term vanishes uniformly in $t$ because
$v\in C([0,T];\ell^1(\ZZ))$ and \eqref{eq_eta_mean_var}. Thus,
\begin{equation*}
	\lim_{ R \to \infty}	\sup_{t\le T}
	\Big\{
	\big\|
	F_R(t,\eta)-F(t,\eta)
	\big\|_{L^{2}( \nu_{\rho} )}
	\vee
	\big\|\partial_{t}F_R(t,\eta)-\partial_{t}F(t,\eta)
	\big\|_{L^{2}( \nu_{\rho} )}
	\Big\}
	=0\,,
\end{equation*}
where the second part of the statement follows by replacing $v_{t,x}$ above with $\partial_t
	v_{t,x}$.

Similarly, we control the generator terms and show that
\begin{equation*}
	\begin{aligned}
		\lim_{ R \to \infty}
		\sup_{ t \in [0, T]} \big\|
		\gen F_{R} (t, \eta )
		- \gen F (t, \eta)
		\big\|_{L^{2}( \nu_{\rho} )} =0\,,
	\end{aligned}
\end{equation*}
with the canonical choice $ \gen F (t, \eta ) := \sum_{x \in \ZZ} v_{t,x} \gen \mP_{x} (
	\eta) $, where $ \mP_{x} ( \eta) := P ( \eta_{x})$. Indeed, the triangle inequality gives, uniformly in $ t \in
	[0,T]$
\begin{equation*}
	\begin{aligned}
		 & \bigg\|
		\sum_{|x| \leqslant R} v_{t,x}
		\gen \mP_{R,x} ( \eta)
		-
		\sum_{x \in \ZZ } v_{t,x}
		\gen \mP_{x} ( \eta)
		\bigg\|_{L^{2}( \nu_{\rho} )}
		\\
		 & \lesssim
		\| \gen \mP_{0}( \eta)-\gen \mP_{R,0} (
		\eta)\|_{L^2(\nu_\rho)}
		\sum_{|x|\le R}|v_{t,x}|
		+
		\| \gen \mP_{0} ( \eta) \|_{L^{2}( \nu_{\rho} )}
		\sum_{|x|>R}|v_{t,x}|\,.
	\end{aligned}
\end{equation*}
Here $\gen \mP_{R,0} ( \eta)  \to  \gen \mP_{0}(\eta)$ pointwise and thus in $ L^{2}(
	\nu_{\rho}) $, and the second moment in the last term is controlled by \eqref{eq_eta_mean_var}.

Now, Dynkin's formula can be applied to each $F_R$, which yields the martingales
\begin{equation*}
	M_t^R
	:=
	F_R(t,\eta(t))-F_R(0,\eta(0))
	-\int_0^t
	(\partial_s+\gen)F_R(s,\eta(s)) \ud s\,.
\end{equation*}
Taking the limit
$R\to\infty$,  the
preceding estimates imply
$\lim_{R \to \infty }
\sup_{t \in [0,T]}\| M_{t} - M_{t}^{R} \|_{L^{2}( \mathbf P_{\rho})} = 0$, 
where
\begin{equation*}
	\begin{aligned}
		M_t
		:=
		F(t,\eta(t))-F(0,\eta(0))
		-\int_0^t
		(\partial_s+\gen)F(s,\eta(s)) \ud s\,.
	\end{aligned}
\end{equation*}
Once more, the martingale property of $M$ follows from $ M^{R}$ by testing martingale increments against bounded cylindrical functionals of the past, cf. \eqref{eq_supp1_exIP}.

Along the same lines, we can show that
$M_t^2-\int_0^t\Gamma F(s,\eta(s))\ud s
$ is a martingale, which identifies the predictable quadratic variation. 
Here $ \Gamma F $ is the canonical extension of $ \Gamma$ to non-local functions of the
form~\eqref{eq_form_F_dynkin_ext}.

\section{Analytic lattice estimates}

\begin{lemma}\label{lem_usefulestimates_phi}
	For every $ \varphi \in \mS ( \RR)$ and every $p\geqslant 1$, we have
	\begin{equation*}
		\begin{aligned}
			\textnormal{(i)}   & \quad
			\sup_{L\in\NN}\sup_{t\in[0,T]}
			L^{-1}\sum_{x\in\ZZ}|\nabla_x^L\varphi_t|^p
			\leqslant \| \varphi' \|_{L^{p}( \RR )}^{p},  \\
			\textnormal{(ii)}  & \quad
			\sup_{L\in\NN}\sup_{t\in[0,T]}
			L^{-1}\sum_{x\in\ZZ}|\Delta_x^L\varphi_t|^p
			\leqslant \| \varphi'' \|_{L^{p}( \RR )}^{p}, \\
			\textnormal{(iii)} & \quad
			\lim_{L\to\infty}\sup_{t\in[0,T]}
			L^{-1}\sum_{x\in\ZZ}
			\big|
			\nabla_x^L\varphi_t
			-\partial_x\varphi_t(L^{-1}x)
			\big|^p
			=0,                                           \\
			\textnormal{(iv)}  & \quad
			\lim_{L\to\infty}\sup_{t\in[0,T]}
			L^{-1}\sum_{x\in\ZZ}
			\big|
			\Delta_x^L\varphi_t
			-\partial_x^2\varphi_t(L^{-1}x)
			\big|^p
			=0.
		\end{aligned}
	\end{equation*}
\end{lemma}

\begin{proof}
	We recall the useful identities (by the fundamental theorem of calculus)
	\begin{equation}\label{eq_supp1_nabla}
		\begin{aligned}
			\nabla_x^L \varphi
			=
			L \int_{ L^{-1}x}^{L^{-1}(x+1)}
			\varphi'(y) \ud y
		\end{aligned}
	\end{equation}
	and
	\begin{equation}\label{eq_supp1_delta}
		\begin{aligned}
			\Delta_x^L \varphi
			 & =
			%L^{2} \int_{ L^{-1}x}^{L^{-1}(x+1)}
			%\varphi'(y) \ud y
			%- L^{2}
			%\int_{ L^{-1}x}^{L^{-1} (x+1)}
			%\varphi'(y - L^{-1}) \ud y
			%\\
			%& =
			%L^{2}
			%\int_{ L^{-1}x}^{L^{-1}(x+1)}
			%\varphi'(y) -	\varphi'(y - L^{-1}) \ud y\\
			%&=
			L^{2}
			\int_{ L^{-1}x}^{L^{-1}(x+1)}
			\int_{y - L^{-1}}^{y} \varphi''(z) \ud z \ud y\,.
		\end{aligned}
	\end{equation}

	\begin{enumerate}
		\item[(i)] Using \eqref{eq_supp1_nabla} and Jensen's
		      inequality, we can write
		      \begin{equation*}
			      \begin{aligned}
				      \sum_{x \in \ZZ}
				      |\nabla_x^L \varphi_{r}|^{p}
				      \leqslant
				      \sum_{x \in \ZZ}
				      L
				      \int_{ L^{-1}x}^{L^{-1}
					      (x+1)} |\varphi_{r}'(y)|^{p} \ud y
				      =
				      L
				      \int_{\RR} |\varphi'(y)|^{p}  \ud y\,,
			      \end{aligned}
		      \end{equation*}
		      where we used that $ \varphi_{r}$ is only a shifted
		      version of $ \varphi$.
		      Because $ \varphi$ is a Schwartz function, the integral
		      on the right--hand side converges.
		\item[(ii)] Applying again Jensen's inequality jointly with
		      \eqref{eq_supp1_delta} yields a bound similar to the previous
		      bullet
		      \begin{equation*}
			      \begin{aligned}
				      \sum_{x \in \ZZ}
				      \big|\Delta_x^L \varphi_{r}\big|^{p}
				      \leqslant
				      %\sum_{x \in \ZZ}
				      %L^{2}
				      %\int_{ L^{-1}x}^{L^{-1}(x+1)}
				      %\int_{y - L^{-1}}^{y} \varphi_{r}''(z)^{2} \ud z \ud y
				      %=
				      L^{2}
				      \int_{\RR} \int_{y - L^{-1}}^{y}
				      |\varphi_r''(z)|^{p}  \ud z \ud y
				      = L \int_{\RR}
				      |\varphi''(y)|^{p} \ud y \,.
			      \end{aligned}
		      \end{equation*}
		      Once more, the integral on the right--hand side is finite
		      since $ \varphi \in \mS ( \RR ) $.

		\item[(iii)]
		      We use the representation \eqref{eq_supp1_nabla}. For every
		      $ r\in[0,T]$,
		      \begin{equation*}
			      \begin{aligned}
				      \nabla_x^L\varphi_r
				      -
				      \partial_x\varphi_r(L^{-1}x)
				       & =
				      L
				      \int_{L^{-1}x}^{L^{-1}(x+1)}
				      \Big(
				      \varphi_r'(y)
				      -
				      \varphi_r'(L^{-1}x)
				      \Big)
				      \ud y\,.
			      \end{aligned}
		      \end{equation*}
		      By the mean value theorem, for
		      $ y\in[L^{-1}x,L^{-1}(x+1)]$,
		      \begin{equation*}
			      \begin{aligned}
				      \big|
				      \varphi_r'(y)
				      -
				      \varphi_r'(L^{-1}x)
				      \big|
				      \leqslant
				      L^{-1}
				      \sup_{|w-L^{-1}x|\leqslant L^{-1}}
				      |\varphi_r''(w)|\,.
			      \end{aligned}
		      \end{equation*}
		      Hence,
		      \begin{equation}\label{eq_grad_approx}
			      \begin{aligned}
				      \big|
				      \nabla_x^L\varphi_r
				      -
				      \partial_x\varphi_r(L^{-1}x)
				      \big|
				      \leqslant
				      L^{-1}
				      \sup_{|w-L^{-1}x|\leqslant L^{-1}}
				      |\varphi_r''(w)|\,.
			      \end{aligned}
		      \end{equation}
		      Therefore, for every $p\geqslant1$,
		      \begin{equation*}
			      \begin{aligned}
				      L^{-1}
				      \sum_{x\in\ZZ}
				      \big|
				      \nabla_x^L\varphi_r
				      -
				      \partial_x\varphi_r(L^{-1}x)
				      \big|^p
				      \leqslant
				      L^{-p}
				      L^{-1}
				      \sum_{x\in\ZZ}
				      \sup_{|w-L^{-1}x|\leqslant L^{-1}}
				      |\varphi_r''(w)|^p.
			      \end{aligned}
		      \end{equation*}
		      The last Riemann sum is uniformly bounded in $r\in[0,T]$
		      and $L\in\NN$, since $\varphi_r''$ is a spatial translate
		      of the Schwartz function $\varphi''$. Thus the right--hand
		      side is of order $L^{-p}$, uniformly in $r$, which proves
		      (iii).

		\item[(iv)]
		      The proof is analogous to the previous
		      bullet. We use the representation
		      \eqref{eq_supp1_delta},
		      which yields a bound in terms of $
			      \varphi'''$ which is controlled in the same
		      manner.
	\end{enumerate}
\end{proof}

% \bibliography{refs}
% \bibliographystyle{alpha}

\newcommand{\etalchar}[1]{$^{#1}$}

\end{document}